\documentclass[]{siamart171218}

\usepackage{amsmath,amssymb}
\usepackage{dsfont}
\usepackage{graphicx}
\usepackage{xcolor}
\usepackage{comment}

\usepackage{nicefrac}
\usepackage{algpseudocode}
\usepackage{tikz}
\usetikzlibrary{calc}
\usetikzlibrary{arrows}
\usepackage{breqn}
\usepackage{todonotes}
\usepackage{etoolbox}
\makeatletter
\patchcmd{\@addmarginpar}{\ifodd\c@page}{\ifodd\c@page\@tempcnta\m@ne}{}{}
\makeatother
\reversemarginpar
\usepackage[normalem]{ulem} %[normalem] option does not change italics when striking out using \sout{}
\newcommand\str{\bgroup\markoverwith
{\textcolor{red}{\rule[0.5ex]{2pt}{1.5pt}}}\ULon} %does \sout but in RED color
\ifpdf
  \DeclareGraphicsExtensions{.eps,.pdf,.png,.jpg}
\else
  \DeclareGraphicsExtensions{.eps}
\fi

\numberwithin{theorem}{section}

\newcommand{\TheTitle}{A distributed ADMM-like method for resource sharing %under conic constraints\\
over time-varying networks}
\newcommand{\TheShortTitle}{An ADMM-like method for resource sharing}
\newcommand{\TheAuthors}{Necdet Serhat Aybat, and Erfan Yazdandoost Hamedani}

\headers{\TheShortTitle}{\TheAuthors}

\title{{\TheTitle}\thanks{In our prior work~\cite{aybat16_dual_static}, published in the Proceedings of the $50^{th}$ Asilomar Conference on Signals, Systems and Computer, we studied the resource sharing problem over static and undirected communication networks -- see Section~\ref{sec:static} for details.
\funding{This research was partially supported by NSF grants CMMI-1400217 and CMMI-1635106, and ARO grant W911NF-17-1-0298.}}}

\author{
  Necdet Serhat Aybat\thanks{Industrial \& Manufacturing Engineering Department, The Pennsylvania State University, PA
  (\email{nsa10@psu.edu},\email{evy5047@psu.edu}).}
  \and
  Erfan Yazdandoost Hamedani\footnotemark[2]}

\usepackage{amsopn}
\DeclareMathOperator{\diag}{diag}

\ifpdf
\hypersetup{
  pdftitle={\TheTitle},
  pdfauthor={\TheAuthors}
}
\fi
\def\grad{\nabla}

\def\ba{\mathbf{a}}
\def\bb{\mathbf{b}}
\def\bc{\mathbf{c}}

\def\be{\mathbf{e}}

\def\bp{\mathbf{p}}
\def\bq{\mathbf{q}}

\def\bs{\mathbf{s}}

\def\bu{\mathbf{u}}
\def\bv{\mathbf{v}}
\def\bw{\mathbf{w}}
\def\bx{\mathbf{x}}  %{\mbox{\boldmath $\lambda$}}
\def\by{\mathbf{y}}
\def\bz{\mathbf{z}}

\def\bC{\mathbf{C}}
\def\bD{\mathbf{D}}

\def\bI{\mathbf{I}}
\def\bJ{\mathbf{J}}

\def\bL{\mathbf{L}}

\def\bP{\mathbf{P}}
\def\bQ{\mathbf{Q}}

\def\cB{\mathcal{B}}
\def\cC{\mathcal{C}}
\def\cD{\mathcal{D}}
\def\cE{\mathcal{E}}
\def\cF{\mathcal{F}}
\def\cG{\mathcal{G}}

\def\cI{\mathcal{I}}

\def\cK{\mathcal{K}}
\def\cL{\mathcal{L}}

\def\cN{\mathcal{N}}
\def\cO{\mathcal{O}}
\def\cP{\mathcal{P}}
\def\cQ{\mathcal{Q}}
\def\cR{\mathcal{R}}
\def\cS{\mathcal{S}}

\def\cW{\mathcal{W}}
\def\cX{\mathcal{X}}
\def\cY{\mathcal{Y}}
\def\cZ{\mathcal{Z}}

\def\smskip{\smallskip}

\def\texitem#1{\par\smskip\noindent\hangindent 25pt
               \hbox to 25pt {\hss #1 ~}\ignorespaces}

\def\norm#1{\left\|#1\right\|}

\newcommand{\BEAS}{\begin{eqnarray*}}
\newcommand{\EEAS}{\end{eqnarray*}}
\newcommand{\BEA}{\begin{eqnarray}}
\newcommand{\EEA}{\end{eqnarray}}
\newcommand{\BEQ}{\begin{eqnarray}}
\newcommand{\EEQ}{\end{eqnarray}}
\newcommand{\BIT}{\begin{itemize}}
\newcommand{\EIT}{\end{itemize}}
\newcommand{\BNUM}{\begin{enumerate}}
\newcommand{\ENUM}{\end{enumerate}}

\newcommand{\BA}{\begin{array}}
\newcommand{\EA}{\end{array}}

\newcommand{\ones}{\mathbf 1}

\newcommand{\reals}{\mathbb{R}}
\newcommand{\integers}{\mathbb{Z}}

\newcommand{\dom}{\mathop{\bf dom}}

\newcommand{\intr}{\mathop{\bf int}}
\newcommand{\relint}{\mathop{\bf rel int}}

\def\green#1{\textcolor{green}{#1}}

\def\purple#1{\textcolor{purple}{#1}}

\newif\ifpagenumbering
\pagenumberingtrue

\pagenumberingfalse

\newsavebox{\theorembox}
\newsavebox{\lemmabox}
\newsavebox{\defnbox}
\newsavebox{\corollarybox}
\newsavebox{\propositionbox}
\newsavebox{\remarkbox}
\newsavebox{\assbox}
\savebox{\theorembox}{\noindent\bf Theorem}
\savebox{\lemmabox}{\noindent\bf Lemma}
\savebox{\defnbox}{\noindent\bf Definition}
\savebox{\corollarybox}{\noindent\bf Corollary}
\savebox{\propositionbox}{\noindent\bf Proposition}
\savebox{\remarkbox}{\noindent\bf Remark}
\savebox{\assbox}{\noindent\bf Assumption}
\newtheorem{assumption}{\usebox{\assbox}}
\newtheorem{remark}{\usebox{\remarkbox}}[section]
\newtheorem{defn}{\usebox{\defnbox}}

\DeclareMathOperator*{\argmin}{\arg\!\min}

\def\fprod#1{\left\langle#1\right\rangle}
\def\prox#1{\mathbf{prox}_{#1}}
\def\ind#1{\mathds{1}_{#1}}

\def\id{\mathbf{I}}
\def\zero{\mathbf{0}}
\def\one{\mathbf{1}}

\def\sa#1{{#1}}

\def\bxi{\boldsymbol{\xi}}

\def\bet{\boldsymbol{\eta}}
\def\bnu{\boldsymbol{\nu}}
\def\bgm{\boldsymbol{\gamma}}

\def\st{\mathbf{s.t.}}

\def\ey#1{\purple{#1}}
\def\eyh#1{\green{#1}}
\def\eyy#1{\textcolor{black}{#1}}
\def\nv#1{{#1}}
\def\Dg{\mathbf{D}_\gamma}
\def\xt{\hbox{$\tilde{\bx}^{k}$}^{\top}}
\def\Ct{\widetilde{\cC}}

\begin{document}

\maketitle

% REQUIRED
\begin{abstract}
We consider cooperative multi-agent resource sharing problems over time-varying {communication networks, where only local communications are allowed.} %those agents connected by an edge can directly communicate.
The objective is to minimize the sum of agent-specific composite convex functions subject to a conic constraint that couples agents' decisions. We propose a distributed primal-dual algorithm DPDA-D to solve the saddle point formulation of the sharing problem on time-varying (un)directed communication networks; %, and to
{and we show that primal-dual iterate sequence converges to a point defined by a primal optimal solution and a consensual dual price for the coupling constraint.}
%; and next we show how to extend this method to handle time-varying directed communication networks.
Furthermore, we provide convergence rates for suboptimality, infeasibility and consensus violation of agents' dual price assessments;
%for the ergodic-average sequence;
examine the effect of underlying network topology on the convergence rates of the proposed decentralized algorithm; and compare %our method
DPDA-D with %Prox-JADMM
\sa{centralized methods
on the basis pursuit denoising and multi-channel power allocation problems}.
\end{abstract}

% REQUIRED

% REQUIRED
\vspace*{-2mm}
\section{Introduction}
Let $\{\cG^t\}_{t\in\reals_+}$ denote a time-varying graph of $N$ computing nodes. More precisely, for $t\geq 0$, the graph has the form $\cG^t=(\cN,\cE^t)$, where $\cN\triangleq\{1,\ldots,N\}$ and $\mathcal{E}^t\subseteq \mathcal{N}\times \mathcal{N}$ %denotes
is the set of \nv{\emph{directed}} edges at time $t$.
%-- without loss of generality, suppose %assume that
%$(i,j)\in \mathcal{E}$ implies $i< j$.
Suppose each node $i\in\cN$ has a \emph{private} constraint function $g_i:\reals^{n_i}\rightarrow\reals^{m}$ and a \emph{private} cost function $\varphi_i:\reals^{n_i}\rightarrow\reals\cup\{+\infty\}$ such that
{\small
\begin{equation}
\label{eq:F_i}
\varphi_i(\xi_i)\triangleq \rho_i(\xi_i) + f_i(\xi_i),
\end{equation}
}%
where $\rho_i: \mathbb{R}^{n_i} \rightarrow \mathbb{R}\cup\{+\infty\}$ is a proper, closed convex function (possibly \emph{non-smooth}), $f_i: \mathbb{R}^{n_i} \rightarrow \mathbb{R}$ is a \emph{smooth} convex function. Assuming each node $i\in\cN$ has only access to $\varphi_i$, $g_i$ and a closed convex cone $\mathcal{K}\subseteq{R}^{m}$,
%is $\cK$ along with its objective,
%Assume $f_i$ is differentiable on an open set containing $\dom \rho_i$, and $\grad f_i$ is Lipschitz with constant \nv{$L_{f_i}$}; the prox map of $\rho_i$, \vspace*{-1mm}
%{\small
%\begin{equation}
%\label{eq:prox}
%\prox{\rho_i}(\xi_i)\triangleq\argmin_{x_i \in \reals^{n_i}} \left\{ \rho_i(x_i)+\tfrac{1}{2}\norm{x_i-\xi_i}^2 \right\},
%\vspace*{-1mm}
%\end{equation}
%}%
%is \emph{efficiently} computable for $i\in\cN$, where $\norm{.}$ denotes the Euclidean norm. \nv{Moreover, suppose each node $i\in\cN$ has a \emph{private} constraint function $g_i:\reals^{n_i}\rightarrow\reals^{m}$ such that $g_i$ is $\cK$-convex~\cite[Chapter~3.6.2]{boyd2004convex}, Lipschitz continuous with constant $C_{g_i}$, and has a Lipschitz continuous Jacobian, $\bJ g_i$, with constant $L_{g_i}$.}
%\nv{Assuming each node $i\in\cN$ has only access to $g_i$ and $\cK$ along with its objective $\varphi_i$}
%\triangleq\rho(\xi_i)+f(\xi_i)$
%defined in \eqref{eq:F_i},
consider the following %minimization
problem:
\vspace*{-1mm}
{\small
\begin{align}\label{eq:central_problem}
\min_{\bxi\in\reals^n}\ {\varphi}(\bxi)\triangleq\sum_{i\in \mathcal{N}}\varphi_i(\xi_i)\quad \hbox{s.t.}\quad \nv{g(\bxi)\triangleq\sum_{i\in\cN}g_i(\xi_i) \in -\mathcal{K}},
\end{align}}%
\noindent where %$\mathcal{K}\subseteq{R}^{m}$ is a closed convex cone %$R_i\in \mathbb{R}^{m\times n_i}$, and $r_i\in \mathbb{R}^{m}$ are the problem data such that
%-- in this context,
$\xi_i \in \reals^{n_i}$ denotes the {\it local} decision %vector
of node $i\in\cN$ and $n\triangleq\sum_{i\in\cN}n_i$.
%We make the following assumption throughout the paper.
\begin{assumption}\label{assum:functions}
\sa{For all $i\in\cN$, the function $f_i$ is differentiable on an open set containing $\dom \rho_i$, and $\grad f_i$ is Lipschitz with constant {$L_{f_i}$}; the prox map of $\rho_i$, \vspace*{-1mm}
{\small
\begin{equation}
\label{eq:prox}
\prox{\rho_i}(\xi_i)\triangleq\argmin_{x_i \in \reals^{n_i}} \left\{ \rho_i(x_i)+\tfrac{1}{2}\norm{x_i-\xi_i}^2 \right\},
\vspace*{-1mm}
\end{equation}
}%
is \emph{efficiently} computable, %for $i\in\cN$,
where $\norm{.}$ denotes the Euclidean norm. Moreover, $g_i$ is $\cK$-convex~\cite[Chapter~3.6.2]{boyd2004convex}, Lipschitz continuous with constant $C_{g_i}$, and has a Lipschitz continuous Jacobian, $\bJ g_i$, with constant $L_{g_i}$.}
\end{assumption}
In this paper, we {design a distributed algorithm for solving \eqref{eq:central_problem} and} provide a unified approach for analyzing the convergence behavior of the proposed method, regardless of
%the topological properties of
whether the communications over the time-varying graph $\{\cG^t\}$ are unidirectional or bidirectional.
\begin{comment}
More precisely, using the proposed technique, we were able to analyze the convergence for both scenarios: for all $t\geq 0$, $\{\cG^t\}$ is either a collection of \emph{undirected} graphs or a collection of \emph{directed} graphs.
\end{comment}
To this aim, \nv{we need some definitions and assumptions related to the time-varying graph $\{\cG^t\}$. To unify the notation, we assume all edges are directed, and consider undirected graphs as a special case of directed graphs.}
\begin{defn}
\label{def:neighbors}
\nv{For any $t\geq 0$, $\cG^t=(\cN,\cE^t)$ is a directed graph; let $\cN^{\,t,{\rm in}}_i\triangleq\{j\in\cN:\ (j,i)\in\cE^t\}\cup\{i\}$ and $\cN^{\,t,{\rm out}}_i\triangleq\{j\in\cN:\ (i,j)\in\cE^t\}\cup\{i\}$ denote the in-neighbors and out-neighbors of node $i\in\cN$ at time $t$, respectively; and let $d_i^t\triangleq |\cN^{\,t,{\rm out}}_i|-1$ be the out-degree of node $i\in\cN$. \nv{$\cG^t=(\cN,\cE^t)$ is called undirected when $(i,j)\in\cE^t$ if and only if $(j,i)\in\cE^t$. For undirected $\cG^t$, let $\cN_i^t\triangleq\cN^{\,t,{\rm in}}_i\setminus\{i\}=\cN^{\,t,{\rm out}}_i\setminus\{i\}$ %$\{j\in\cN:\ (i,j)\in\cE^t\}$
%~\hbox{ or }~(j,i)\in\cE^t\}$
denote the %set of
neighbors of $i \in \mathcal{N}$}, and {$d_i^t\triangleq |\mathcal{N}_i^t|$ represents} the degree of node $i\in \mathcal{N}$ at time $t$.}
\end{defn}
\begin{assumption}
\label{assump:communication_general}
%Suppose $\{\cG^t\}_{t\in\reals_+}$ is a collection of either all directed or all undirected graphs.
\nv{When $\cG^t$ is a (general) directed graph, node $i\in\cN$ can receive data from $j\in\cN$ only if $j\in\cN_i^{\,t,{\rm in}}$, i.e., $(j,i)\in\cE^t$, and can send data to $j\in\cN$ only if $j\in\cN_i^{\,t,{\rm out}}$, i.e., $(i,j)\in\cE^t$; %or $(j,i) \in \cE^t$;
on the other hand, when $\cG^t$ is undirected, node $i\in\cN$ can send and receive data to and from $j\in\cN$ at time $t$ only if $j\in\cN_i^t$, i.e., $(i,j) \in \cE^t$.}
\end{assumption}

%Suppose that projections onto $\cK$ can be computed efficiently, while the projection onto the preimage $A_i^{-1}(\cK_i+b_i)$ is assumed to be \emph{impractical}, e.g., when $\cK_i$ is the positive semidefinite cone, projection to preimage requires solving an SDP.
Our objective is to solve \eqref{eq:central_problem} in a \emph{decentralized} fashion using the computing nodes in $\cN$ {while the information exchange among the nodes is restricted to edges in $\cE^t$ for $t\geq 0$ according to Assumption~\ref{assump:communication_general}}. {
%In particular,
We are interested in designing algorithms which can distribute the computation over the nodes such that each node's computation is based on the local topology of %(possibly directed)
$\cG^t$ and information only available to that node.}

Decentralized optimization over communication networks has %recently
drawn %a variety of
attention from a wide range of application areas: coordination and control in multirobot networks, parameter estimation in wireless sensor networks, %distributed spectrum sensing in cognitive radio networks,
processing distributed big data in machine learning, and distributed power control in cellular networks, to name a few. In these examples, the network size can be prohibitively large for centralized optimization, which requires a fusion center that collects the physically distributed data and runs a centralized optimization method. This process has expensive communication overhead, requires
large enough memory to store and process the data, and also may violate data privacy in case agents are not willing to share their data even though they are collaborative~\cite{olfati2007consensus}. Therefore, a common objective of today's big-data networks is to use decentralized optimization techniques to avoid expensive communication overhead required by the centralized setting and to enhance the data privacy. The communication networks in these application areas may be directed, i.e., communication links can be unidirectional, and/or the network may be time-varying, e.g., communication links in a wireless network can be on/off over time due to failures or the links may exist among agents depending on their inter-distances.
%Decentralized optimization over wired networks is one of the applications where time-varying directed networks can also arise. In particular, bi-directional communication protocols can create deadlocks due to lack of enforcement rule to block a third node when the other two neighbors are exchanging local variables between themselves which emerge a necessity for applying uni-directional asynchronous protocols.
%In Section~\ref{sec:dual}, we show that \emph{resource allocation} type problems of the following form
%{\small
%\begin{align}\label{eq:dual-implement-intro}
%\min_{\{\xi_i\}_{i\in\cN}}\ \sum_{i\in \cN}{\Phi_i(x,\xi_i)}\quad\hbox{s.t.}\quad \sum_{i\in\cN}Q_i x+R_i\xi_i-r_i \in \cK,
%\end{align}}%
%can be handled in a similar way using both primal and dual consensus iterations at the same time, where $\cK\subseteq \reals^m$ is a closed convex cone, $R_i\in\reals^{m\times n_i}$, $Q_i\in\reals^{m\times n}$, and $r_i\in \reals^{m}$ are the problem data such that each node $i\in\cN$ only have access to $R_i$, $r_i$ and $\cK$ along with its objective $\Phi_i(x,\xi_i)\triangleq\rho(x,\xi_i)+f(x,\xi_i)$ defined similarly as in \eqref{eq:F_i}.

In the remainder of this section, as a brief preliminary, we discuss %a specific implementation of
the primal-dual algorithm~(PDA) proposed in~\cite{chambolle2015ergodic} \nv{to solve convex-concave saddle-point problems with a \emph{bilinear} coupling term}, explain its connections to ADMM-like algorithms, and briefly discuss some recent work related to ours. \nv{It is worth noting that the saddle point~(SP) problem formulation of \eqref{eq:central_problem} contains a coupling term that is \emph{not} bilinear due to nonlinear $\{g_i\}_{i\in\cN}$; therefore, PDA is not applicable.}
{Next, in Section~\ref{sec:dynamic}, we propose DPDA-D, a new distributed algorithm \nv{based on PDA and extending it to handle nonlinear constraints, for solving the SP formulation of the multi-agent sharing problem in~\eqref{eq:central_problem}} when the topology of the connectivity graph is \emph{time-varying} with %possibly
{\emph{(un)directed}} communication links, and %in Section~\ref{sec:dynamic} we also
we state the main theorem establishing the convergence properties of DPDA-D; and in Section~\ref{sec:convergence}, we provide the proof of the main theorem.
%show how to extend this method to handle time-varying \emph{directed} communication networks -- \sa{for directed networks information flow on edges are unidirectional, i.e., $(i,j)\in\cE^t$ implies node $i$ can send data to node $j$ at time $t$, but the flow in the other direction is not permitted.
{Subsequently, in Sections~\ref{sec:stepsize}, \ref{sec:static-nonlinear} and~\ref{sec:dual-bound}, we discuss certain details related to the applicability of the method in practice.} In Section~\ref{sec:numerics}, we compare our method with Prox-JADMM~\cite{deng2013parallel} on the basis pursuit denoising problem, \sa{and with Mirror-prox~\cite{he2015mirror} on the multi-channel power allocation problem}; and finally, in~Section~\ref{sec:conclusions} we state our concluding remarks and briefly discuss potential future work.}
\subsection{Preliminary}
In this paper, we study an \emph{inexact} %\ey{and more general}
variant of the primal-dual algorithm~(PDA) proposed in~\cite{chambolle2015ergodic}, \nv{extending it to handle nonlinear constraints, to solve the %saddle-point problem
SP formulation of \eqref{eq:central_problem}} in a decentralized manner over a time-varying communication network. %Therefore,
%We now briefly discuss the convergence properties of PDA in~\cite{chambolle2015ergodic}.
{There has been active research on %developing
efficient algorithms for convex-concave saddle point problems $\min_{\bx}\max_{\by}\cL(\bx,\by)$, e.g.,~\cite{chambolle2011first,chen2014optimal,he2012convergence,nedic2009subgradient}.}
%In a recent paper, Chambolle and Pock~\cite{chambolle2015ergodic} proposed
%Recently, a primal-dual algorithm~(PDA) is proposed in
{PDA~\cite{chambolle2015ergodic} also belongs to this family and is proposed for the %composite
convex-concave %saddle-point
SP problem:}%\vspace*{-1mm}
{\small
\begin{align}
\label{eq:saddle-problem}
\min_{\bx\in\cX}\max_{\by\in\cY}\cL(\bx,\by)\triangleq\Phi(\bx)+\fprod{\nv{T(\bx)},\by}-h(\by),
%\quad \hbox{\normalsize where}\quad \Phi(\bx)\triangleq\rho(\bx)+f(\bx),
\end{align}
}%
where $\cX$ and $\cY$ are finite-dimensional vector spaces, $\Phi(\bx)\triangleq\rho(\bx)+f(\bx)$, $\rho$ and $h$ are possibly non-smooth convex functions, $f$ is a convex function and has a Lipschitz continuous gradient defined on $\dom \rho$ with Lipschitz constant $L$, and $T:\cX\rightarrow\cY$ is a \nv{\emph{linear} map}. Briefly, given $\bx^0\in\cX$, $\by^0\in\cY$ and algorithm parameters $\nu_x,\nu_y>0$, PDA consists of two proximal-gradient steps that can be written as: %-- one primal and one dual -- as follows
{\small
\begin{subequations}
\label{eq:pda}
\begin{align}
\bx^{k+1}\gets\argmin_{\bx\in\cX}~&\rho(\bx)+f(\bx^k)+\fprod{\grad f(\bx^k),~\bx-\bx^k}+\fprod{\nv{T(\bx)},\by^k}+\frac{1}{\nu_x}D_x(\bx,\bx^k)\label{eq:PDA-x}\\
\by^{k+1}\gets\argmin_{\by\in\cY}~&h(\by)-\fprod{\nv{2T(\bx^{k+1})-T(\bx^k)},\by}+\frac{1}{\nu_y}D_y(\by,\by^k),
\label{eq:PDA-y}
\end{align}
\end{subequations}
}%
where $D_x$ and $D_y$ are Bregman distance functions corresponding to some continuously differentiable strongly convex functions $\psi_x$ and $\psi_y$
%, both with convexity modulus 1, and
such that $\dom \psi_x\supset\dom\rho$ and $\dom \psi_y\supset\dom h$. In particular, $D_x(\bx,\bar{\bx})\triangleq\psi_x(\bx)-\psi_x(\bar{\bx})-\fprod{\grad \psi_x(\bar{\bx}),~\bx-\bar{\bx}}$, and $D_y$ is defined similarly. \nv{Abusing the notation, below we use $T$ also to denote the corresponding matrix, i.e., $T(\bx)=T\bx$.}

In~\cite{chambolle2015ergodic}, %\todo{Erfan: Removed a sentence} %a simple proof for the ergodic convergence is provided; indeed,
it is shown that, when the convexity modulus for $\psi_x$ and $\psi_y$ is 1, if {$\nu_x,\nu_y>0$} are chosen such that $(\frac{1}{\nu_x}-L)\frac{1}{\nu_y}\geq\sigma^2_{\max}(T)$, then for any $\bx,\by\in\cX\times\cY$, %the following bound,
{\small
\begin{align}
\label{eq:DPA-rate}
{\cL(\bar{\bx}^K,\by)-\cL(\bx,\bar{\by}^K)}\leq\tfrac{1}{K}\big(\tfrac{1}{\nu_x}D_x(\bx,\bx^0)+\tfrac{1}{\nu_y}D_y(\by,\by^0)-\fprod{T(\bx-\bx^0),\by-\by^0}\big),
\end{align}
}%
holds for all $K\geq 1$, where $\bar{\bx}^K\triangleq\frac{1}{K}\sum_{k=1}^K \bx^k$ and $\bar{\by}^K\triangleq\frac{1}{K}\sum_{k=1}^K \by^k$.

It is worth mentioning the connection between PDA and the alternating direction method of multipliers (ADMM). Indeed, when implemented on $\min_{\bv\in\cX^*,\by\in\cY}\{\Phi^*(\bv)+h(\by):\ \bv+T^\top\by=\mathbf{0}\}$, preconditioned ADMM is equivalent to PDA~\cite{chambolle2011first,chambolle2015ergodic}, where $\cX^*$ denotes the dual space \nv{and $\Phi^*$ is the convex conjugate of $\Phi$}. There is also a strong connection between the linearized ADMM algorithm proposed by Aybat et al. in~\cite{aybat2015distributed} and PDA proposed in~\cite{chambolle2015ergodic} -- for details of these relations,
see Section 1.4.% in the online technical report \cite{aybat2016distributed}.

\textbf{Notation.} %Throughout the paper,
$\norm{\cdot}$ {denotes the Euclidean or the spectral norm depending on its argument, i.e., for a matrix $R$, $\norm{R}=\sigma_{\max}(R)$}. Given a convex set $\cS$, {let \eyy{$\sigma_{\cS}(\cdot)$} denote its support function, i.e., $\sigma_{\cS}(\theta)\triangleq\sup_{w\in \cS}\fprod{\theta,~w}$,}
let $\ind{S}(\cdot)$ denote the indicator function of $\cS$, i.e., $\ind{S}(w)=0$ for $w\in\cS$ and equal to $+\infty$ otherwise, and let $\cP_{\cS}(w)\triangleq\argmin\{\norm{v-w}:\ v\in\cS\}$ denote the Euclidean projection onto $\cS$. For a closed convex set $\cS$, we define the distance function as $d_{\cS}(w)\triangleq\norm{\cP_{\cS}(w)-w}$. Given a convex cone $\cK\in\reals^m$, let $\mathcal{K}^*$ denote its dual cone, i.e., $\mathcal{K}^*\triangleq\{\theta\in\reals^{m}: \ \langle \theta,w\rangle \geq 0\ \ \forall w\in\mathcal{K}\}$, and $\cK^\circ\triangleq -\cK^*$ denote the polar cone of $\cK$. {Note that for any cone $\cK\in\reals^m$, $\sigma_{\cK}(\theta)=0$ for $\theta\in\cK^\circ$ and equal to $+\infty$ if $\theta\not\in\cK^\circ$, i.e., $\sigma_{\cK}(\theta)=\ind{\cK^\circ}(\theta)$ for all $\theta\in\reals^m$.}
%\str{Cone $\cK$ is called \emph{proper} if it is closed, convex, pointed, and it has a nonempty interior.}
Given a convex function $h:\reals^n\rightarrow\reals\cup\{+\infty\}$, its convex conjugate is %defined as
$h^*(w)\triangleq\sup_{\theta\in\reals^n}\fprod{w,\theta}-h(\theta)$, and \nv{for differentiable $h:\reals^n\to\reals^m$, $\bJ h:\reals^n\to\reals^{m\times n}$ denotes the Jacobian of $h$}. Throughout the paper, $\otimes$ denotes the Kronecker product, $\Pi$ denotes the Cartesian product, and $\id_n$ is the $n\times n$ identity matrix. $Q$-norm is defined as $\norm{z}_Q\triangleq (z^\top Q z)^{\tfrac{1}{2}}$ for any positive definite matrix $Q$.

\subsection{Our Previous Work on Resource Sharing} %over Static Undirected $\cG$}
\label{sec:static}
{In~\cite{aybat16_dual_static}, we considered %solving
\eqref{eq:central_problem} \nv{%with a linear conic constraint, i.e.,
when $g_i(\xi)=r_i-R_i\xi_i$ is affine for $i\in\cN$}, over a \emph{static} and \emph{undirected} communication network $\cG=(\cN,\cE)$ as a dual consensus problem. %In particular,
\nv{Using Lagrangian duality, we reformulated it as \sa{an SP} problem, $\min_{\bxi} \max_{y\in\cK^\circ} \sum_{i\in \cN}\varphi_i(\xi_i)+\langle \sum_{i\in \cN}R_i\xi_i-r_i,~y \rangle$ which can be written in a distributed form through creating local copies of dual variable $y\in\reals^n$ as}
$\nv{(P)}:\min_{\substack{\bxi}} \max_{\substack{\by}}\big\{\sum_{i\in \cN}\varphi_i(\xi_i)+\langle R_i\xi_i-r_i,~y_i \rangle:\ y_i\in\cK^\circ\  \forall i\in\cN, \ y_i=y_j~\forall (i,j) \in \cE \big\},$
where $\bxi=[\xi_i]_{i\in\cN}$ and $\by=[y_i]_{i\in\cN}$. Using $M$, the edge-node incidence matrix of $\cG$, the consensus constraints $y_i=y_j$ for $(i,j)\in \mathcal{E}$ can be written as $M\by=\mathbf{0}$. %-- see \cite{aybat16_dual_static} for details.
%, where {$M\triangleq H\otimes \id_m\in \mathbb{R}^{m|\mathcal{E}|\times m|\mathcal{N}|}$} and $H$ is the oriented edge-node incidence matrix, i.e., the entry $H_{(i,j),l}$, corresponding to edge $(i,j)\in \mathcal{E}$ and $l\in \mathcal{N}$, is equal to $1$ if $l=i$, $-1$ if $l=j$, and $0$ otherwise. %Note that $M^\T M=H^\T H\otimes \id_m=\Omega\otimes \id_m$, where $\Omega\in \mathbb{R}^{|\mathcal{N}|\times |\mathcal{N}|}$ denotes the graph Laplacian of $\cG$, i.e., $\Omega_{ii}=d_i$, $\Omega_{ij}=-1$ if $(i,j)\in\cE$ or $(j,i)\in\cE$, and equal to $0$ otherwise.
\nv{Furthermore, by dualizing the consensus constraints,
we %equivalently write
obtain another SP problem, equivalent to $(P)$, in the form of \eqref{eq:saddle-problem}:}
{\small
\begin{align}
\label{eq:saddle-static-comp}
\min_{\substack{\bxi}}\max_{\substack{\by\in\Pi_{i\in\cN}}\cK^\circ}\min_{\substack{\bw}}\cL(\bxi,\bw,\by)
=\min_{\substack{\bxi,\bw}}\max_{\substack{\by\in\Pi_{i\in\cN}\cK^\circ}}\cL(\bxi,\bw,\by),
\end{align}
}%
where $\cL(\bxi,\bw,\by)\triangleq \sum_{i\in \cN}\varphi_i(\xi_i)+\langle R_i\xi_i-r_i,~y_i \rangle -\fprod{\bw,M\by}$. The equality in \eqref{eq:saddle-static-comp} holds as long as $\cK$ is a pointed cone -- hence $\intr\left(\cK^\circ \right)\neq \emptyset$; {therefore, for each fixed $\bxi$, inner $\max_\by$ and $\min_\bw$ can be interchanged.}
%Suppose $\cK$ in \eqref{eq:central_problem} is a proper cone.
%First, %we define the notations %that will be used throughout the paper. Next,
%in Theorem~\ref{thm:general-rate}, we discuss
The saddle-point problem on the right side of \eqref{eq:saddle-static-comp} is special case of \eqref{eq:saddle-problem} with a separable structure. Exploiting this special structure, we customized PDA in~\eqref{eq:pda} and proposed Algorithm~DPDA-S.
%displayed in Fig.~\ref{alg:PDS}, where $\gamma>0$ is an algorithm parameter, $\tau_i$ and $\kappa_i$ are the primal-dual step-sizes for $i\in\cN$.
In~\cite{aybat16_dual_static} we showed that Algorithm~DPDA-S can solve the sharing problem~\eqref{eq:central_problem} \nv{with an affine conic constraint} in a \emph{decentralized} way and established its convergence properties provided that the node-specific primal-dual step-sizes $\{\tau_i,\kappa_i\}_{i\in\mathcal{N}}$ and the algorithm parameter $\gamma>0$ %are chosen such that
satisfy {$\frac{1}{\tau_i}>L_{f_i}$, and $\big({1\over \tau_i}-L_{f_i}\big)\big( {1\over \kappa_i}-2\gamma d_i\big) \geq \norm{R_i}^2$}, for all $i\in\cN$,
%\vspace*{-1mm}
%{\small
%\begin{equation}\label{eq:schur-cond-s}
%\frac{1}{\tau_i}>L_i\quad\hbox{and}\quad \Big({1\over \tau_i}-L_i\Big)\Big( {1\over \kappa_i}-2\gamma d_i\Big) \geq \norm{R_i}^2, %\sigma_{\max}^2(R_i),
%\quad \forall\ i\in\mathcal{N},
%\end{equation}
%\vspace*{-1mm}
%}%
%\hspace*{-2.5mm}
where $d_i$ denotes the degree of $i\in\cN$ for the static $\cG$. %In particular,
Our result in~\cite{aybat16_dual_static} refines the error bound in~\eqref{eq:DPA-rate} %In particular, we were able to
and establishes $\cO(1/k)$ ergodic rate in terms of suboptimality and infeasibility of the DPDA-S iterate sequence -- see Theorem 2 in \cite{aybat16_dual_static}.

{The arguments used for proving Theorem~2 in \cite{aybat16_dual_static} cannot be used for the \emph{time-varying directed} communication network setting considered in this paper {since the undirected network is encoded through the use of $M\by=\mathbf{0}$ constraint}.
%\todo{Erfan: Rephrased}
%The main reason is that the topology of the undirected network is encoded in the saddle-point formulation through the use of $M\by=\mathbf{0}$ constraint.
However, when the topology is time-varying or when the edges are directed, it is not immediately clear how one can represent this problem as \sa{an SP} problem. %In Section~\ref{sec:dynamic},
To extend our previous results to a more general setting of time-varying topology with possibly directed edges, in this paper we develop a new SP formulation that can impose consensus over the dual variables while the formulation is independent of the changing topology.} \nv{Finally, the new method can also handle nonlinear conic constraints on resource sharing in~\eqref{eq:central_problem}.}
\subsection{Related Work}
Now we briefly review some recent work on the distributed resource sharing problem.
%\todo{Erfan: Rephrased}
%distributed solution of a resource sharing problem among a set of agents.
%, $\cN$, communicating over the network $\cG=(\cN,\cE)$.
{From the application perspective, %there are several work on
%an important special case of resource allocation problem is
algorithms and their basic convergence analysis have been studied for the economic dispatch problem~(EDP), e.g., \cite{seifi2011electric} for power-flow networks and \cite{guo2016distributed,zhang2014efficient} for smart-grids. The variants of %the economic dispatch problem
EDP considered in~\cite{guo2016distributed,seifi2011electric,zhang2014efficient} are special cases of~\eqref{eq:central_problem}. In particular, each node $i\in\cN$ has a convex objective function $f_i$, usually a quadratic function; $\rho_i(\xi_i)=\ind{\cX_i}(\xi_i)$, where $\cX_i$ is a local simple convex constraint set, $g_i(\xi_i)=\xi_i-r_i$ and $\cK=\{{\bf 0}\}$. In~\cite{seifi2011electric}, the aim is to optimize the total power generation cost in a DC %direct current
power-flow model, %-- see Chapter 1 in~\cite{seifi2011electric} for details and a simple optimization problem in Appendix B.
\cite{guo2016distributed,zhang2014efficient} also study a similar problem considering random wind power injection
%\todo{Erfan: Removed a sentence}
%In~\cite{zhang2014efficient}, an ADMM based method is proposed, and in~\cite{guo2016distributed} a distributed projected gradient method is given
-- both papers establish basic convergence results without any rate guarantees.
Distributed resource allocation problem can also arise in controlling and coordinating internet services over hybrid edge-cloud networks; for which a distributed ADMM algorithm is proposed in~\cite{huang2017collaborative} to solve a problem in form of \eqref{eq:central_problem} with $\cK=\{{\bf 0}\}$ and {$g_i(\xi_i)=\xi_i-r_i$}.} {\cite{yang2017distributed} studies EDP considering communication delays in directed time-varying network topology, and %for handling communication over directed networks, reference~\cite{yang2017distributed}
{an algorithm based on push-sum protocol is proposed.}
%In \cite{nedic2015distributed} it is shown that their method can solve unconstrained consensus minimization of a nonsmooth function with bounded subgradient having $\cO(\log(k)/\sqrt{k})$ rate of convergence.
}

{From the theoretical point of view, there has been active research on distributed resource allocation problem.
%Next, we discuss some recent work proposed for special cases of \eqref{eq:central_problem}.
In~\cite{doan2016distributed}, a distributed Lagrangian method (DLM) has been proposed for solving a particular case of \eqref{eq:central_problem} on a \emph{static} network; more precisely, the objective is to minimize sum of local convex functions subject to local convex \emph{compact} sets and a coupling constraint of the form $\sum_{i\in\cN}\xi_i-r_i=\mathbf{0}$. In~\cite{doan2016distributed}, the authors establish convergence rate of $\cO(\log(k)/\sqrt{k})$ {for the} dual function values estimated at the time-weighted average of dual iterates. Reference~\cite{doan2017distributed} gives a gradient balancing protocol to solve \eqref{eq:central_problem} in which $\rho_i(\cdot)=0$, {$g_i(\xi_i)=\xi_i-r_i$} and $\cK=\{{\bf 0}\}$. The authors show that the generated sequence $\bxi^k=[\xi_i^k]_{i\in\cN}$ satisfies $\sum_{i\in\cN}f_i(\xi_i^k) - \varphi^*\leq \cO(1/k)$ and is feasible for all $k$ under the assumption that the initial point $\bxi^0=[\xi_i^0]_{i\in\cN}$ is feasible \sa{-- $\varphi^*$ denotes the optimal value}; moreover, a linear rate is established when each $f_i$ is strongly convex. For a similar formulation as in~\cite{doan2017distributed}, an asynchronous gradient-descent method is proposed in~\cite{lakshmanan2008decentralized} for time-varying undirected communication networks; the proposed algorithm produces a feasible iterate sequence such that {$\min_{\ell=1,\ldots,k}\max_{i,j\in\cN}\norm{\grad f_i(\xi_i^\ell)-\grad f_i(\xi_j^\ell)}\leq\cO(1/\sqrt{k})$} %rate in terms of magnitude of gradient of the objective function
when each $f_i$ is convex and has a Lipschitz %continuous
gradient. However, none of these methods can solve \eqref{eq:central_problem} in its full generality over a time-varying and directed communication network.}

In~\cite{chang2015multi}, %an inexact
a method based on ADMM is proposed to reduce the computational work %complexity of implementing
of ADMM due to exact minimizations in each iteration. First, a \emph{dual consensus} ADMM is proposed for solving~\eqref{eq:central_problem} {over an undirected static %communication
network} in a distributed fashion when $\cK=\{\mathbf{0}\}$, {$g_i(\xi_i)=R_i\xi_i-r_i$}, and $\varphi_i(\xi_i)=\rho_i(\xi_i)+f_i(A_i\xi_i)$ for $\rho_i$ and $f_i$ as in~\eqref{eq:F_i}.
%\todo{Erfan: Removed a sentence}
%Under strong duality assumption, it is shown that dual iterate sequence converges and every limit point of the primal sequence is optimal without giving a rate result. Next,
To avoid exact minimizations in ADMM, an inexact variant taking proximal-gradient steps is analyzed. Convergence of primal-dual sequence is shown when each $f_i$ is strongly convex -- without a rate result; and a linear rate is established in the absence of the non-smooth $\rho_i$, i.e., $\varphi_i(\xi_i)=f_i(A_i\xi_i)$, and assuming each $A_i$ has full column-rank and $f_i$ is strongly convex, i.e., $\varphi_i$ is strongly convex.
%\todo{Ref1 says that since the problem is closely related to its dual optimal consensus problem. We should review some related papers." See~\cite{chang2014proximal,chang2015multi,mateos2017distributed}}
%In particular, they proposed dual consensus ADMM (DC-ADMM) and inexact dual consensus ADMM (IDC-ADMM) in which the objective is to minimize sum of convex functions with coupling equality constraints under static network topology. \ey{Generally, they assumed that for any agent the objective function is a convex composite function of a smooth convex function and a possibly nonsmooth convex function, the objective function has at least one bounded subgradient at every xi, and strong duality holds. Under these assumptions, they proved that DC-ADMM algorithm is convergent without providing rate result. Moreover, by using proximal first-order approximation they were able to propose IDC-ADMM with less computational complexity than DC-ADMM. However, convergence of IDC-ADMM depends on satisfying a condition of a penalty parameter involving in the algorithm which depends on the network topology, and more restrictive assumption, in which the smooth function need to be Lipschitz differentiable and strongly convex. Furthermore, in the absence of the nonsmooth part it shown that IDC-ADMM will converge linearly.}

{In~\cite{chang2014proximal}, a proximal \emph{dual consensus} ADMM method, PDC-ADMM,
is proposed by Chang to minimize $\sum_{i\in\cN}\varphi_i$ subject to coupling equality and agent-specific constraints over both static and time-varying \emph{undirected} networks -- for the time-varying topology, they assumed that agents are on/off and communication links fail randomly with certain probabilities.
%\todo{Erfan: Removed a sentence}
%They consider a special structure on
%Each agent-specific set is assumed to be an intersection
%where for any agent the constraints comprise a polyhedra constraint
%of a polyhedron and a ``simple" compact set. More precisely, the goal
{The goal in the paper} is to solve $\min_{\bxi}\{\sum_i\varphi_i(\xi_i):\ \sum_{i\in\cN}R_i\xi_i=r,~\xi_i\in\cX_i,~i\in\cN\}$ where $\varphi_i$ is closed convex, $\cX_i=\{\xi_i\in\cS_i:\ C_i\xi_i\leq d_i\}$ and $\cS_i$ is a convex compact set for each $i\in\cN$. The polyhedral constraints {$\xi_i\in\cX_i$} are handled using a penalty formulation without requiring projections onto them. %, for which it is easy to compute the projection.
It is shown that both for static and time-varying cases, PDC-ADMM have $\cO(1/k)$ ergodic convergence rate in the mean for suboptimality and infeasibility; that said, in each iteration, costly \emph{exact} minimizations involving $\varphi_i$ are needed. To alleviate this burden, Chang also proposed an inexact PDC-ADMM taking prox-gradient steps when \nv{$\varphi_i(\xi_i)=\rho_i(\xi_i)+f_i(A_i \xi_i)$ and %is composite convex
%for some
$A_i$ is a linear map for each $i\in\cN$,}
%similar to \eqref{eq:F_i},
and showed $\cO(1/k)$ ergodic convergence rate when each $f_i$ is \emph{strongly convex} and differentiable with a Lipschitz continuous gradient for $i\in\cN$.}%\todo{Note that this method cannot handle problems with logistic regression objective.}

In~\cite{chang2014distributed}, a consensus-based distributed primal-dual perturbation (PDP) algorithm using a diminishing step-size sequence is proposed. The objective is to minimize a composition of a global network function (smooth) with the sum of local objective functions (smooth), i.e., $\cF(\sum_{i\in\cN}f_i(x))$, subject to local compact sets and an inequality constraint, $\sum_{i\in\cN}g_i(x)\leq 0$, over a time-varying directed network. It is shown that the %local
primal-dual iterate sequence converges to an %global
optimal primal-dual solution; however, no rate result is provided.
%\todo{Erfan: Removed a sentence and the reference [34] in v14}
%The proposed PDP method can also handle non-smooth constraints with similar convergence guarantees.
%More recently, in~\cite{niederl2015distributed}, distributed continuous-time coordination algorithms are proposed to minimize convex separable objective subject to coupling equality and convex inequality constraints.
%Assuming the objective and constraint functions are Lipschitz, point-wise convergence is established without providing a rate result.\todo{We should briefly add the other previous papers we found. I added those papers that are time-varying.}
\newpage
There are fewer %studies conducted on
papers on resource allocation over time-varying directed networks. %\cite{gu2018distributed,xu2017distributed}.
\cite{gu2018distributed} considers a special case of \eqref{eq:central_problem} with $\cK=\{\mathbf{0}\}$, $g_i(\xi_i)=\xi_i-r_i$, $f_i$ is convex, and $\rho_i(\xi_i)=\ind{\cX_i}(\xi_i)$ where $\cX_i$ is convex and compact for $i\in\cN$. Assuming a Slater point exists which implies boundedness of dual optimal set, the authors proved $\cO(\log(k)/\sqrt{k})$ %convergence
rate result. Reference~\cite{xu2017distributed} has the same setting in~\cite{gu2018distributed} with $\cX_i=[\underline{\xi}_i,\bar{\xi}_i]$. Assuming each $f_i$ is smooth and strongly convex,
%proposed a nonnegative surplus-based
a distributed method is proposed and its convergence is shown without providing a rate result.
%\todo{you also mentioned [18] later}
Finally, while we were preparing this paper, we became aware of a recent work~\cite{mateos2017distributed,nedic2017improved}. \cite{mateos2017distributed} also uses
%related to ours in certain ways: i)
Fenchel conjugation and
%Laplacian averaging play key roles in both papers, ii)
\emph{dual consensus} formulation %is used
to decompose separable constraints. %The authors proposed
A distributed algorithm on time-varying {\emph{balanced}\footnote{A directed graph $\cG$ is balanced when each node has equal number of in-degree and out-degree.}} directed
%\todo{Erfan: Look at Theorem IV.6 in their paper and $C_u$ in (26). They assume the graphs to be ``weight-balanced'', and it looks like time-varying regular digraphs. Do you agree?}
communication networks is proposed for solving saddle-point problems subject to consensus constraints.
%\todo{Erfan: Removed a sentence}
%The algorithm can also be applied to solve consensus optimization problems with inequality constraints that can be written as the sum of local convex functions of local and global variables. %Moreover, their algorithm were able to handle semi-definite constraints.
Assuming each agents' local iterates and subgradient sets are uniformly bounded,
%Assuming i) each agent's local variable lies in locally known compact convex set, and the global variable lies in a globally known compact convex set, ii) a ball centered at the origin is given such that it contains the optimal dual solution set, iii) subgradients generated by the algorithm are bounded,
it is shown that %under some assumptions and using a carefully selected decreasing step-size sequence,
the ergodic average of primal-dual sequence converges with $\mathcal{O}(1/\sqrt{k})$ rate in terms of saddle-point evaluation error; however, when the method is applied to constrained optimization problems, \emph{no} rate in terms of suboptimality and infeasibility is provided.
{The other recent work in~\cite{nedic2017improved} investigates the connection between decentralized resource allocation problem and decentralized consensus optimization problem where the objective is to minimize sum of convex functions subject to local closed convex sets and $\sum_{i\in\cN}\xi_i-r_i={\bf 0}$ over \emph{static} \emph{undirected} networks. Utilizing the mirror relationship between the optimality conditions of these problems, they proposed a method for solving the decentralized resource allocation problem and proved $o(1/{k})$ rate of convergence in terms of \emph{squared} residuals of first-order optimality conditions.} %hence, $o(1/\sqrt{k})$ rate for infeasibility.}
%%%%%
%%%%
\subsection{Connection of PDA to ADMM}
\label{sec:connection}
Consider the following convex-concave saddle point problem:
\vspace*{-3mm}
{\small
\begin{align}\label{saddlepoint}
\min_\bx \max_\by \rho(\bx)+\fprod{T\bx,\by}-h(\by),
\end{align}}%
where $\rho$ and $h$ are closed convex functions, and $T$ is a linear map. The corresponding primal minimization problem take the form of $\min_\bx \rho(\bx)+h^*(T\bx)$, or equivalently
\vspace*{-3mm}%
{\small
\begin{align}\label{primal-min}
\min_{\bx,\bw} \rho(\bx)+h^*(\bw)\quad \st \quad T\bx-\bw=0.
\end{align}}%
Moreover, the dual of \eqref{primal-min} can be written as follows:
\vspace*{-3mm}%
{\small
\begin{align}\label{dual-min}
\min_{\bv, \by} \rho^*(\bv)+h(\by)\quad \st \quad \bv+T^\top\by=0.
\end{align}}%
In an earlier paper~\cite{chambolle2011first}, Chambolle and Pock proposed a primal-dual algorithm for solving \eqref{primal-min}. Given $\theta\in(0,1]$, in each iteration two proximal steps are computed as follows:
\vspace*{-3mm}%
{\small
\begin{subequations}
\label{alg:CP11}
\begin{align}
\by^{k+1}&\gets\prox{\nu_y h}(\by^k+\nu_y T\bz^{k}),\label{alg:CP11_y} \\
\bx^{k+1}&\gets\prox{\nu_x\rho}(\bx^k-\nu_x T^\top \by^{k+1}), \label{alg:CP11_x}\\
\bz^{k+1}&\gets\bx^{k+1}+\theta(\bx^{k+1}-\bx^k).
\label{alg:CP11_z}
\end{align}
\end{subequations}}%
For $\theta=1$, by reindexing the $\{\by^k\}_k$ iterate sequence, the iterations in~\eqref{alg:CP11} can be written as: $\bx^{k+1}=\prox{\nu_x\rho}(\bx^k-\nu_x T^\top \by^k)$ and $\by^{k+1}=\prox{\nu_y h}\big(\by^k+\nu_y T(2\bx^{k+1}-\bx^k)\big)$ -- note that this is the same iterate sequence generated by \eqref{eq:pda} implemented on \eqref{saddlepoint} when both $D_x$ and $D_y$ are chosen as $\tfrac{1}{2}\norm{\cdot}^2$. In particular, \eqref{alg:CP11} proposed in~\cite{chambolle2011first} is a special case of \eqref{eq:pda} %which is
proposed in~\cite{chambolle2015ergodic}.

When ADMM~\cite{eckstein1992douglas} is applied on problem \eqref{dual-min}, it generates the following iterates:
\vspace*{-3mm}%
{\small
\begin{subequations}\label{eq:ADMM}
\begin{align}
&\by^{k+1}\gets~\argmin_\by\Big\{h(\by)-\fprod{T^\top\by,\bx^k}+{c'\over 2}\norm{T^\top\by+\bv^{k}}^2\Big\}  \\
&\bv^{k+1}\gets~\argmin_\bv\Big\{\rho^*(\bv)-\fprod{\bv,\bx^k}+{c'\over 2}\norm{T^\top\by^{k+1}+\bv}^2\Big\}  \\
&\bx^{k+1}\gets~\bx^{k}-c'~(T^\top\by^{k+1}+\bv^{k+1})
\end{align}
\end{subequations}}%
for some penalty parameter $c'>0$. It is shown in~\cite{chambolle2011first} that when $T=\bI$, $\theta=1$, and $\nu_x=c'$, $\nu_y=\tfrac{1}{c'}$ for $c'>0$, the algorithm in~\eqref{alg:CP11} is equivalent to the ADMM implementation in \eqref{eq:ADMM}. %Douglas-Rachford splitting (DRS) algorithm

Let $M_1$, $M_2$ be positive semidefine matrices. Alternating direction proximal method of multipliers (AD-PMM), proposed in~\cite{shefi2014rate} to solve~\eqref{primal-min}, computes the iterates as follows:
\vspace*{-2mm}%
{\small
\begin{subequations}
\label{alg:AD-PMM}
\begin{align}
\bx^{k+1}\gets~&\argmin_\bx\Big\{\rho(\bx)+{c\over 2}\norm{T\bx-\bw^k+c^{-1}\by^k}^2+{1\over 2}\norm{\bx-\bx^k}^2_{M_1}\Big\}, \\
\bw^{k+1}\gets~&\argmin_\by\Big\{h^*(\bw)+{c\over 2}\norm{T\bx^{k+1}-\bw+c^{-1}\by^k}^2+{1\over 2}\norm{\bw-\bw^k}^2_{M_2}\Big\}, \\
\by^{k+1}\gets~&\by^{k}+c~(T\bx^{k+1}-\bw^{k+1}).
\end{align}
\end{subequations}}%
Another variant of \eqref{alg:CP11} with the same convergence guarantees can be obtained by simply replacing \eqref{alg:CP11_z} with $\bz^{k+1}=\by^{k+1}+\theta(\by^{k+1}-\by^k)$ and switching the order of updates in \eqref{alg:CP11_x} and \eqref{alg:CP11_y}, i.e., $\bx^{k+1}$ is computed before $\by^{k+1}$. When $\theta=1$, the iterations for this variant can be written as: $\bx^{k+1}=\prox{\nu_x\rho}(\bx^k-\nu_x T^\top \bz^k)$, $\by^{k+1}=\prox{\nu_y h}\big(\by^k+\nu_y T\bx^{k+1}\big)$, $\bz^{k+1}=2\bx^{k+1}-\bx^k$. According to Proposition~3.1 in~\cite{shefi2014rate}, the iterates generated by this %slightly different
variant of \eqref{alg:CP11} for $\nu_y=c>0$ is equivalent to those generated by \eqref{alg:AD-PMM} when $M_1=\nu_x^{-1}\bI_n-c T^\top T$ and $M_2=\mathbf{0}$.
%The equivalency of \eqref{alg:CP11} and \eqref{alg:AD-PMM} is shown in the following proposition.
%\begin{proposition}\cite{shefi2014rate}
%Let $\{\bx^k,\by^k,\bz^k\}$ generated by \eqref{alg:CP11} with $\theta=1$. Then, iterates in \eqref{alg:CP11} reduces to iterates \eqref{alg:AD-PMM} with $M_2=0$ and $M_1=\tau^{-1}\bI_n-c A^\top A$.
%\end{proposition}

In~\cite{aybat2015distributed}, Aybat et al. proposed proximal-gradient ADMM (PG-ADMM) for solving
$
\min_{\bx,\bw}\Big\{h^*(\bw)+\sum_{i=1}^N\varphi_i(x_i):\ A_ix_i+B_i\bw=b_i,\quad i=1,\ldots N\Big\},
$
where $\bx=[x_i]_{i=1}^N$ for $x_i\in\reals^{n_i}$, $\varphi_i=\rho_i+f_i$ is composite convex as in~\eqref{eq:F_i} and $h$ is a closed convex function. PG-ADMM is only different from ADMM in $x_i$-subproblems where $x_i^{k+1}$ is computed by minimizing the linear approximation of the augmented Lagrangian~(AL) function after fixing $\bw$ at $\bw^k$ and linearizing the whole smooth part of the AL including $f_i$ around $\bx^k$ -- this leads to taking a prox-gradient step to compute each $x_i$ iterate. In the extreme case that $\rho_i=0$ for $i\in\cN$, PG-ADMM reduces to GADM, studied in~\cite{Ma-Zhang-EGADM-2013} and \cite{Gao-Jiang-Zhang-2014} -- GADM takes gradient steps to compute $x_i$; Gao et al.~\cite{Gao-Jiang-Zhang-2014} prove the $O(1/t)$ convergence rate for GADM. PG-ADMM has also $\cO(1/t)$ rate and it can be viewed as an extension of G-ADMM where convex $\rho_i$'s are allowed.

There is a strong connection between PG-ADMM and PDA in~\cite{chambolle2015ergodic}. For all $\bx\in\cX$, let $\Phi(\bx)=\rho(\bx)+f(\bx)$ as in~\eqref{eq:F_i} such that $\grad f$ is Lipschitz with constant $L$; implementing PG-ADMM on $\min_{\bx,\bw}\{\Phi(\bx)+h^*(\bw):\ T\bx-\bw=\mathbf{0}\}$ generates the following iterate sequence:
\vspace*{-2mm}%
{\small
\begin{subequations}\label{eq:pgadmm}
\begin{align}
\bx^{k+1}&\gets\prox{\tau\rho}\Big(\bx^k-\tau\Big[\grad f(\bx^k)+T^\top\big(\by^k+c~(T\bx^k-\bw^k)\big)\Big]\Big), \label{eq:pgadmm_x1}\\
\bw^{k+1}&\gets\argmin \Big\{h^*(\bw)+\tfrac{c}{2}\norm{T\bx^{k+1}-\bw+c^{-1}\by^k}^2\Big\}, \label{eq:pgadmm_w1}\\
\by^{k+1}&\gets\by^k+c~(T\bx^{k+1}-\bw^{k+1}). \label{eq:pgadmm_y1}
\end{align}
\end{subequations}}%
For PG-ADMM iterate sequence, the suboptimality and infeasibility converges to 0 in the ergodic sense
for any $c>0$ when $\frac{1}{\tau}\geq L+c\norm{T}^2$. Note \eqref{eq:pgadmm_x1} %in PG-ADMM iterations
can be rewritten as $\bx^{k+1}=\prox{\tau\rho}(\bx^k-\tau[\nabla f(\bx^k)+T^\top(2\by^k-\by^{k-1})])$.
%{\small
%\begin{align}
%\bx^{k+1}&\gets\prox{\tau\rho}\Big(\bx^k-\tau\big[\grad f(\bx^k)+T^\top(2\by^k-\by^{k-1})\big]\Big). \label{eq:pgadmm_x2}
%\end{align}}%
Using Moreau proximal decomposition on $\bw$-updates in \eqref{eq:pgadmm_w1}, we get
{\small
\begin{align}
\bw^{k+1}= T\bx^{k+1}+ c^{-1}\by^k-{1\over c}~\prox{c h}(\by^k+c T\bx^{k+1}). \label{eq:pgadmm_w2}
\end{align}
}%
Combining \eqref{eq:pgadmm_w2} and \eqref{eq:pgadmm_y1} shows $\by^{k+1}=\prox{c h}(\by^k+c T\bx^{k+1})$. Thus, \eqref{eq:pgadmm} can be written as
\vspace*{-2mm}%
{\small
\begin{subequations}\label{eq:pgadmm-final}
\begin{align}
\bx^{k+1}&\gets\prox{\tau\rho}\Big(\bx^k-\tau\big[\grad f(\bx^k)+T^\top(2\by^k-\by^{k-1})\big]\Big),\\
\by^{k+1}&\gets\prox{c h}(\by^k+c T\bx^{k+1}).
\end{align}
\end{subequations}
}%
The iterative scheme in \eqref{eq:pgadmm-final} is a variant of PDA iterations in~\cite{chambolle2015ergodic}. In particular, PG-ADMM as written in \eqref{eq:pgadmm-final} generates the same iterate sequence as $(\bx^{k+1},\by^{k+1})=\cP\cD_{\tau,c}(\bx^{k},\by^{k},\bx^{k+1},2\by^{k}-\by^{k-1})$ in \cite{chambolle2015ergodic} when the Bregman functions $D_x$ and $D_y$ chosen as $\tfrac{1}{2}\norm{\cdot}^2$. Moreover, one can easily prove that Theorem 1 in \cite{chambolle2015ergodic} is still true for this variant of PDA for any $\tau,c>0$ such that $(\frac{1}{\tau}-L)\frac{1}{c}\geq \norm{T}^2$.
}
%%%%%
%%%%%
\noindent \section{A Distributed Algorithm for Time-varying Network Topology}
\label{sec:dynamic}
In this section, we develop a distributed %primal-dual
algorithm for solving \eqref{eq:central_problem} when the communication network topology is \emph{time-varying}, %We \sa{study} \eqref{eq:central_problem}
under the following assumption. %\todo{Mention that there are well-known papers adopting the same assumption}
\begin{assumption}
\label{assump:existence}
{A primal-dual %optimal
solution to \eqref{eq:central_problem} exists and the duality gap is 0.}
\end{assumption}
Clearly this assumption holds if a Slater point for \eqref{eq:central_problem} exists, i.e., there exists some \nv{$\bar{\bxi}\in\relint(\dom \varphi \cap\dom g)$ such that $g(\bar{\bxi}) \in \intr(-\cK)$}. \nv{Existence of a Slater point is also assumed in many related papers, e.g., \cite{chang2014distributed,gu2018distributed,mateos2017distributed,nedic2017improved,nedic2009subgradient}}.
%\todo{Erfan: Removed the reference [24] in v14 which is not by Nedic}
When $\cK=\{\mathbf{0}\}$ and $g_i(\xi)=R_i\xi-r_i$ for $i\in\cN$, Assumption~\ref{assump:existence} trivially holds if there exists some $\bar{\bxi}\in\relint(\dom \varphi)$ that is feasible, i.e., $\sum_{i\in\cN}R_i\bar{\xi}_i-r_i=\mathbf{0}$.

\nv{Since $\ind{\cK}(\cdot)=\sup_{y\in\reals^m}\{\fprod{y,~\cdot}-\sigma_{\cK}(y)\}$,} one can reformulate $\eqref{eq:central_problem}$ as
%the following saddle point problem:
\vspace*{-1mm}%
{\small
\begin{equation}\label{eq:dual-saddle-d}
\min_{\bxi}\max_{y\in\reals^m}\Big\{\sum_{i\in \cN}{\varphi_i(\xi_i)}\nv{- \Big\langle\sum_{i\in\cN}g_i(\xi_i),~y \Big\rangle}-\sigma_{\cK}(y) \Big\}.\vspace*{-1mm}
\end{equation}}%
According to Assumption~\ref{assump:existence}, a dual optimal solution $y^*\in\cK^\circ$ exists and the duality gap is 0 for \eqref{eq:central_problem}. Suppose each node $i\in\cN$ has its own estimate $y_i\in\reals^m$ of a dual optimal solution; and ${\by}=[y_i]_{i\in\cN}$ denotes these estimates in long-vector form. We define the \emph{consensus set} as \vspace*{-2mm}
{\small
\begin{equation}
\label{eq:Cd}
\cC\triangleq\{{\by}\in \mathbb{R}^{m|\mathcal{N}|}:\ \exists \bar{y}\in\mathbb{R}^m \ \mbox{ s.t. }\ y_i=\bar{y}\quad \forall i\in\mathcal{N}\}.
\end{equation}}%

Suppose we are given a \nv{(possibly trivial) bound $B\in(0,\infty]$ such that $\norm{y^*}{\leq} B$. For instance, if a Slater point is available, then a nontrivial bound $B\in(0,\infty)$ on dual solutions can be obtained by solving a convex problem in a distributed way; on the other hand, when Slater condition holds for~\eqref{eq:central_problem} but a Slater point is not available, then the nodes can collectively compute a Slater point -- see Section~\ref{sec:dual-bound}.}
%for details see Section 5 of the online technical report \cite{aybat2016distributed}.
% we will discuss how to obtain such a bound in Section~\ref{sec:dual-bound}.
Let $\cB_0\triangleq\{y\in\reals^m:\ \norm{y}\leq {2B}\}$ and $\mathcal{B}\triangleq\Pi_{i\in\cN}\cB_0$, i.e., $\cB=\{\by: \ \|y_i\|\leq {2B},\ i\in\cN\}$. Finally, we also define the \emph{bounded consensus set}, %as follows:
{\small \begin{equation}
\label{eq:C}
\tilde{\cC}\triangleq\cC\cap \cB=\{{\by}\in \mathbb{R}^{m|\mathcal{N}|}:\ \exists \bar{y}\in\cB_0\subset\mathbb{R}^m \ \mbox{ s.t. }\ y_i=\bar{y}\quad \forall i\in\mathcal{N}\}.
\end{equation}}%

%similarly as we defined $C$ in Section~\ref{sec:dynamic}.
%Since $\cG$ is {\it connected},
We can equivalently reformulate \eqref{eq:dual-saddle-d} as the following dual consensus problem:
\vspace*{-1mm}%
{\small
\begin{equation}\label{eq:dual-saddle-dist}
\min_{\substack{\bxi}}\max_{\substack{\by\in \Ct}}\ L(\bxi,\by)\triangleq\sum_{i\in \cN}\Big({\varphi_i(\xi_i)}-\nv{\fprod{g_i(\xi_i),~y_i }}-\sigma_{\cK}(y_i) \Big),
\end{equation}}%
{i.e., any saddle point of~\eqref{eq:dual-saddle-dist} is also a saddle point of \eqref{eq:dual-saddle-d},} which follows from the definitions of $\sigma_{\cK}(\cdot)$ and $\Ct$. Define $\cL:\reals^n\times\reals^{m|\cN|}\times\reals^{m|\cN|}\rightarrow\reals\cup\{\pm\infty\}$ such that
{\small
\begin{equation}
\label{eq:lagrangian-dual-implementation}
\cL(\bxi,\bw,\by)\triangleq\sum_{i\in \cN}\Big({\varphi_i(\xi_i)}-\nv{\fprod{g_i(\xi_i),~y_i }}-\sigma_{\cK}(y_i)\Big)-\langle \bw,{\by} \rangle +\sigma_{\Ct}(\bw)\nv{-\ind{\cB}(\by)}.
\end{equation}}%
Note that for any $\bxi\in\dom{\varphi}$, we have $\max_{\by\in \Ct}L(\bxi,\by)=\nv{\max_{\by}}\min_{\bw}\cL(\bxi,\bw,\by)$; hence, $\eqref{eq:dual-saddle-dist}$ can be equivalently written as follows:
\vspace*{-1mm}
{\small
\begin{equation}\label{eq:dual-dist-problem}
\min_{\substack{\bxi}}\big\{\nv{\max_{\substack{\by}}}\min_{\substack{\bw}}\cL(\bxi,\bw,\by)\big\}\
=\ \min_{\substack{\bxi,\bw}}\nv{\max_{\by}}\cL(\bxi,\bw,\by),
\end{equation}}%
\nv{where %the equality directly follows from Fenchel duality; indeed,
interchanging %inner
$\max_{\by}$ and $\min_{\bw}$ is trivially justified when $\cB$ is bounded; in case $B=+\infty$, i.e., $\cB_0=\reals^m$, one can directly verify that $\min_{\substack{\bw}}\max_{\by}\cL(\bxi,\bw,\by)=\min_{\substack{\bw}}\max_{\by}\cL(\bxi,\bw,\by)$ and is equal to $\varphi(\bxi)$ if $g(\bxi)\in-\cK$, and $+\infty$ otherwise.} %\sa{when $\intr(\cK^\circ)\neq\emptyset$ (see Theorem~3.3.5 in~\cite{borwein2010convex} -- second condition there holds because $\Ct\cap \Pi_{i\in\cN}\intr(\cK^\circ)\neq\emptyset$). CHECK THESE CONDITIONS....} %The second equality follows from the strong duality assumption for \eqref{eq:dual-implement}.
%\todo{Comment: Discuss that we can also handle if $B$ is not available. Also check if we can handle the case $B$ does not exist, i.e., is dual iterate boundedness enough?}

Since we can equivalently solve %$\min_{\btheta}\max_{\substack{\bx,\bw}}~-\cL(\bx,\bw,\btheta)$, which is exactly in the same form stated in Theorem~\ref{thm:general-rate} except for the switch of roles in primal and dual variables.
$\min_{\substack{\bxi,\bw}}\max_{\by}\cL(\bxi,\bw,\by)$ in~\eqref{eq:dual-dist-problem} to solve \eqref{eq:central_problem}, we next \sa{generalize} PDA iterations in \eqref{eq:PDA-x}-\eqref{eq:PDA-y} to solve this %equivalent
saddle-point problem.
\begin{defn}
\label{def:bregman}
Let $\cX\triangleq\Pi_{i\in\cN}\reals^{n_i}\times\Pi_{i\in\cN}\reals^{m}$ and $\cX\ni\bx=[\bxi^\top \bw^\top]^\top$ for $\bxi=[\xi_i]_{i\in\cN}\in\reals^n$ and {$\bw\in\reals^{n_0}$}, where $n\triangleq\sum_{i\in\cN}n_i$ and $n_0\triangleq m|\cN|$; let $\cY\triangleq\Pi_{i\in\cN}\reals^{m}$ and $\cY\ni\by=[y_i]_{i\in\cN}\in\reals^{n_0}$. Given parameters $\gamma>0$, and $\tau_i,\kappa_i>0$ for $i\in\cN$, let $\mathbf{D}_\gamma\triangleq\frac{1}{\gamma}\id_{n_0}$, $\mathbf{D}_\tau\triangleq\diag([\frac{1}{\tau_i}\id_{n_i}]_{i\in\cN})$, and $\mathbf{D}_\kappa\triangleq\diag([\frac{1}{\kappa_i}\id_{m}]_{i\in\cN})$. Defining $\psi_x(\bx)\triangleq\frac{1}{2}\bxi^\top\mathbf{D}_\tau\bxi+\frac{1}{2}\bw^\top\mathbf{D}_\gamma\bw$ and $\psi_y(\by)\triangleq\frac{1}{2}\by^\top\mathbf{D}_\kappa\by$ leads to the following Bregman distance functions: $D_x(\bx,\bar{\bx})=\frac{1}{2}\norm{\bxi-\bar{\bxi}}_{\mathbf{D}_\tau}^2+\frac{1}{2}\norm{\bw-\bar{\bw}}_{\mathbf{D}_\gamma}^2$, and $D_y(\by,\bar{\by})=\frac{1}{2}\norm{\by-\bar{\by}}_{\mathbf{D}_\kappa}^2$. To simplify notation, also define $\cZ\triangleq\cX\times\cY$ and $\cZ\ni\bz=[\bx^\top \by^\top]^\top$.
\end{defn}
%, which is almost in the same form stated in Theorem~\ref{thm:general-rate}.
%and $\by=[\btheta]=[\theta_i]_{i\in\cN}$,
\begin{defn}
\label{def:saddle-definitions}
Suppose $\varphi_i=\rho_i+f_i$ is a composite convex function defined as in~\eqref{eq:F_i} for $i\in\cN$. Let $\Phi(\bx)\triangleq\rho(\bx)+f(\bx)$ and $h(\by)\triangleq\sum_{i\in\cN}h_i(y_i)$ for all $\by\in\cY$, where $\rho(\bx)\triangleq \sigma_{\Ct}(\bw)+\sum_{i\in\cN}\rho_i(\xi_i)$, $f(\bx)\triangleq\sum_{i\in\cN}f_i(\xi_i)$ and \nv{$h_i(y_i)\triangleq \sigma_{\mathcal{K}}(y_i)+\ind{\mathcal{B}_0}(y_i)$ for $i\in\cN$. Let $G:\reals^n\rightarrow\reals^{n_0}$ such that $G(\bxi)\triangleq [g_i({\xi_i})]_{i\in\cN}$ for all $\bx\in\cX$ and define $T:\reals^n\times\reals^{n_0}\rightarrow \reals^{n_0}$ such that
$T(\bx)\triangleq -G(\bxi) -\bw$; hence, $\bJ T(\bx)= [-\bJ G(\bxi)~-\id_{n_0}]$.}
\end{defn}
With the aim of solving \eqref{eq:central_problem} as an %saddle-point
SP problem, let $\Phi$, $h$, and $T$ be as given in Definition~\ref{def:saddle-definitions}, and consider
%computing a saddle-point for
$\min_{\bx\in\cX}\max_{\by\in\cY}\Phi(\bx)+\fprod{\nv{T(\bx)},\by}-h(\by)$. %Note that $\sigma_{\mathcal{K}}(.)=\ind{\cK^\circ}(.)$. Therefore
Hence, given the initial iterates $\bxi^0,\bw^0,\by^0$ and parameters $\gamma>0$, $\tau_i,\kappa_i>0$ for $i\in\cN$, choosing Bregman functions $D_x$ and $D_y$ as in Definition~\ref{def:bregman}, and setting $\nu_x=\nu_y=1$, \nv{we propose a modified version of PDA iterations to handle nonlinear $T(\cdot)$; indeed, after linearizing $T(\bx)$ around $\bx^k$ in \eqref{eq:PDA-x}, the iterations in \eqref{eq:pda} can be written %explicitly
as follows for $k\geq 0$:}
%\todo{These are really not PDA iterations, we should mention that we extend PDA to handle nonlinear $T$  maps.}
%Given $\bxi^0,\sa{\bv^0},\by^0$, %set $\bv^0\gets\bw^0$; and
%for all $i\in\cN$ and $k\geq 0$ compute
%when we implement PDA in \eqref{eq:PDA-x}-\eqref{eq:PDA-y} as stated in Theorem~\ref{thm:general-rate}, the iterations
\vspace*{-2mm}%
{\small
\begin{subequations}
\label{eq:pock-pd-3}
\begin{align}
&\bw^{k+1}\gets\argmin_{\bw} \sigma_{\Ct}(\bw)-\langle {\by}^k,~\bw \rangle +{1\over 2\gamma}\|\bw-\bv^k\|^2, \label{eq:pock-pd-3-v}\\
&\bv^{k+1}\gets\bw^{k+1}, \label{eq:pock-pd-v_exact}\\
&\xi_i^{k+1}\gets\argmin_{\xi_i} \rho_i(\xi_i) %+f_i(\xi_i^k)
+\langle\nabla f_i(\xi_i^k)-\nv{\bJ g_i(\xi_i^k)^\top y_i^k},\xi_i-\xi_i^k \rangle +{1\over 2\tau_i}\|\xi_i-\xi_i^k\|^2,\ i\in\cN, \label{eq:pock-pd-3-xi}\\
&y_i^{k+1}\gets\argmin_{y_i\in\cK^\circ\cap\cB_0} \nv{\langle 2g_i(\xi_i^{k+1})-g_i(\xi_i^k)+2v_i^{k+1}-v_i^k,~y_i\rangle}
%+\langle 2v_i^{k+1}-v_i^k,~y_i\rangle
+{1\over 2\kappa_i}\|y_i-y_i^k\|^2,\ i\in\cN,  \label{eq:pock-pd-3-y}
\end{align}
\end{subequations}
}%
where we initialize $\bv^0=\bw^0$. The reason we introduced an auxiliary sequence $\{\bv^k\}_{k\geq 0}$ such that $\bv^k=[v_i^k]_{i\in\cN}$ will be explained %in detail
shortly. Briefly, in its currently stated form, \nv{the computation in \eqref{eq:pock-pd-3} %is a naive implementation of
can be considered as linearized PDA iterations -- $T(\cdot)$ in \eqref{eq:PDA-x}-\eqref{eq:PDA-y} is linearized around $\bx^k$}; however, this naive scheme is not suitable for our purposes, i.e., the $\bw^{k+1}$ update in~\eqref{eq:pock-pd-3-v} is not practical to be computed in a \emph{distributed} manner. Therefore, instead of setting $\bv^{k+1}$ to $\bw^{k+1}$, we will replace \eqref{eq:pock-pd-v_exact} and assign $\bv^{k+1}$ to an approximation of $\bw^{k+1}$ such that this approximation can be efficiently computed in a distributed way -- this modified version of \eqref{eq:pock-pd-3} will be analyzed as \nv{an \emph{inexact} variant of \emph{linearized} PDA.}

Using the extended Moreau decomposition for proximal operators, {for $k\geq 0$,} %$\bw^{k+1}$ can be written as
\vspace*{-2mm}%
{\small
\begin{align}\label{eq:exact-w-computation}
\bw^{k+1}&=\argmin_{\bw} \sigma_{\Ct}(\bw)+{1\over 2\gamma}\norm{\bw-(\bv^k+{\gamma}\by^k)}^2=\prox{\gamma\sigma_{\Ct}}(\bv^k+\gamma\by^k),\nonumber\\
&=\gamma\Big[\frac{1}{\gamma}\bv^k+\by^k-\mathcal{P}_{\Ct}\big( \frac{1}{\gamma}\bv^k+\by^k\big)\Big].
\end{align}
}%
%{\small
%\begin{equation}
%\label{eq:exact-w-computation}
%\bw^{k+1}\gets\bv^k+\gamma\by^k-\gamma~\mathcal{P}_{C_d}\left( \frac{1}{\gamma}\bv^k+\by^k\right).
%\end{equation}}%
For an arbitrary $\by=[y_i]_{i\in\cN}\in \mathbb{R}^{n_0}$, $\mathcal{P}_{\Ct}(\by)$ can be computed as
$
\mathcal{P}_{\Ct}(\by)=\one\otimes \argmin_{x\in\cB_0} \sum_{i\in\mathcal{N}}\|x-y_i\|^2=\one\otimes\argmin_{x\in\cB_0}\|x-{1\over |\mathcal{N}|}\sum_{i\in\mathcal{N}}y_i\|^2,
$
where $\one\in\reals^{|\cN|}$ denotes the vector of all ones. Hence, we can write
$\mathcal{P}_{\Ct}(\by)=\mathcal{P}_{\mathcal{B}}\left((W\otimes\id_m)\by\right)$, where $W\triangleq\frac{1}{|\cN|}\one\one^\top\in\reals^{|\cN|\times|\cN|}.$
%\todo{We should put here a discussion on $B=\infty$. Why do we have $\tilde{p}(\cdot)$, did we use $p(\cdot)$ anywhere?}
Equivalently,
\vspace*{-3mm}%
{\small
\begin{equation}
\label{eq:proj-2}
\mathcal{P}_{\Ct}(\by)=\mathcal{P}_{\mathcal{B}}\left(\one\otimes \nv{p(\by)}\right),\quad \hbox{\normalsize where}\quad \nv{p(\by)}\triangleq{1\over |\mathcal{N}|}\sum_{i\in\mathcal{N}}y_i.
\end{equation}
}%
\nv{Note that $\mathcal{P}_{\mathcal{B}}(\by)=\by$ for all $\by\in\cY$ when $B=\infty$;} and for $B<\infty$, $\cP_{\cB}(\cdot)$ is easy to compute locally since $\cB=\Pi_{i\in\cN}\cB_0$ and $\cP_{\cB_0}(y)=y\min\{1,~{2B/\norm{y}}\}$ for $y\in\reals^m$. Furthermore, $\bxi$-step and $\by$-step of the PDA implementation in \eqref{eq:pock-pd-3} can also be computed locally at each node; however, computing $\bw^{k+1}$ requires communication among the nodes. Indeed, evaluating the average operator \nv{${p}(\cdot)$} is \emph{not} a simple operation in a decentralized computational setting which only allows for communication among neighboring nodes {-- see Assumption~\ref{assump:communication_general}}. %\todo{Erfan:Removed a sentence}
%{In particular, we assume that for undirected communication networks any node can send/receive data to/from its neighboring nodes, while for directed communication networks any node can only receive data from its in-neighbors and can send data to its out-neighbors -- see Assumption~\ref{assump:communication_general}}.
To overcome the issue with decentralized computation of the \nv{averaging operator \eyy{${p}(\cdot)$}}, we will use \emph{multi-communication rounds} to approximate \nv{${p}(\cdot)$}, and analyze the resulting primal-dual iterations as an \emph{inexact} primal-dual algorithm. In~\cite{chen2012fast}, the idea of using \emph{multi-communication rounds} has also been exploited within a distributed primal algorithm for \emph{unconstrained} consensus optimization problems over \emph{undirected} communication networks. %For this purpose,

We define a \emph{communication round} as an operation over $\cG^t$ such that every node simultaneously sends and receives data to and from its neighboring nodes according to Assumption~\ref{assump:communication_general} %one time exchanging local variables among neighboring nodes
-- the details of this operation will be discussed shortly. %and after the exchange each node $i\in\cN$ computes a convex combination to closely align the local decisions.}
We assume that communication among neighbors occurs \emph{instantaneously}, and nodes operate \emph{synchronously}; and we further assume that for each iteration $k\geq 0$, there exists an approximate averaging operator $\cR^k(\cdot)$ which can be computed in a decentralized fashion and approximates $\mathcal{P}_{\Ct}(\cdot)$ with decreasing approximation error in $k$.
\begin{assumption}
\label{assump:approximate-average}
Given a time-varying network $\{\cG^t\}_{t\in\reals_+}$ such that $\cG^t=(\cN,\cE^t)$ for $t\geq 0$, suppose that there is a global clock known to all $i\in\cN$. Assume that the local operations in \eqref{eq:pock-pd-3-xi} and \eqref{eq:pock-pd-3-y} can be completed between two tics of the clock for all $i\in\cN$ and $k\geq 1$, and  every time the clock ticks a communication round with instantaneous messaging between neighboring nodes takes place subject to Assumption~\ref{assump:communication_general}. Suppose that for each $k\geq 0$ there exists $\cR^k(\cdot)=[\cR_i^k(\cdot)]_{i\in\cN}$ such that $\cR_i^k(\cdot)$ can be computed with local information available to node $i\in\cN$ and decentralized computation of $\cR^k$ requires $q_k$ communication rounds. Furthermore, we assume that there exist \nv{$\Gamma>0$} and $\alpha\in (0,1)$ such that \nv{$N\Gamma\geq 1$} and for all $k\geq 0$, $\cR^k$ satisfies
\vspace*{-1mm}%
{\small
\begin{align} \label{eq:approx_error-for-full-vector}
{\mathcal{R}^k(\bw)\in\cB,\qquad \|\mathcal{R}^k(\bw)-\mathcal{P}_{\Ct}(\bw)\| \leq N~\Gamma \alpha^{q_k}\norm{\bw}},\quad\forall~\bw\in\reals^{n_0}.
\end{align}}
\end{assumption}
The ``unit time" is defined to be the length of the interval between two tics of the clock. The assumption that every node $i\in\cN$ can finish its $\xi_i$ and $y_i$ updates %computations
in one unit time is mainly for the sake of notational simplicity throughout the analysis. All of our results still hold as long as there exists a uniform bound $\Delta\in\mathbb{Z}_+$ such that the local operations in \eqref{eq:pock-pd-3-xi} and \eqref{eq:pock-pd-3-y} can be completed in $\Delta$ unit time for all $i\in\cN$ and $k\geq 1$. In the rest, we assume that $\Delta=1$ as in Assumption~\ref{assump:approximate-average}.

Consider the $k$-th iteration of PDA as shown in \eqref{eq:pock-pd-3}. Instead of setting $\bv^{k+1}$ to $\bw^{k+1}$ as in \eqref{eq:pock-pd-v_exact}, we propose approximating $\bw^{k+1}$ using the inexact averaging operator $\cR^k(\cdot)=[\cR_i^k(\cdot)]_{i\in\cN}$ of Assumption~\ref{assump:approximate-average} and set $\bv^{k+1}$ to this approximation. This way, we can skip \eqref{eq:pock-pd-3-v} step and avoid explicitly computing $\bw^{k+1}$ as in \eqref{eq:exact-w-computation} which requires using the exact averaging to compute $\mathcal{P}_{\Ct}(\cdot)$.
More precisely, to obtain an \emph{inexact} variant of \eqref{eq:pock-pd-3}, we replace \eqref{eq:pock-pd-v_exact} with the following:
%Now, using \eqref{eq:approx-average-dual}, we approximate $\bw^{k+1}$ computation in \eqref{eq:exact-w-computation} with the following update rule:
\vspace*{-1mm}%
{\small
\begin{equation}\label{eq:inexact-rule-dual}
{\bv}^{k+1}\gets\gamma\left[\tfrac{1}{\gamma}{\bv}^k+\by^k-\cR^k\left(\tfrac{1}{\gamma}{\bv}^k+\by^k\right)\right].
\end{equation}
}%
%and replace the exact computation $\bv^{k+1}\gets\bw^{k+1}$ in \eqref{eq:pock-pd-3} with the inexact iteration rule in \eqref{eq:inexact-rule-dual}.
Thus, PDA iterations in~\eqref{eq:pock-pd-3}, for solving the saddle-point formulation, \nv{$\min_{\bxi,\bw}\max_{\by}$ $\cL(\bxi,\bw,\by)$}, of the distributed resource allocation problem in \eqref{eq:central_problem}, can be computed inexactly, but in \emph{decentralized} way for a time-varying connectivity network $\{\cG^t\}_{t\geq 0}$ provided that $\cR^k$ satisfying Assumption~\ref{assump:approximate-average} exists for $\{\cG^t\}_{t\geq 0}$. We call this inexact version of the \nv{linearized} PDA as Algorithm~DPDA-D and the node-specific computations of DPDA-D are displayed in Fig.~\ref{alg:PDDual} below. {Indeed, the iterate sequence $\{\bxi^k,\bv^k,\by^k\}_{k\geq 0}$ generated by Algorithm~DPDA-D %displayed in Fig.~\ref{alg:PDDual}
is the same sequence generated by the recursion in \eqref{eq:inexact-rule-dual}, \eqref{eq:pock-pd-3-xi}, and \eqref{eq:pock-pd-3-y}. As emphasized previously, the sequence $\{\bw^k\}_{k\geq 0}$ will not be explicitly computed, instead we will use it in the analysis of the inexact algorithm.} Next, we discuss the existence of inexact average operators $\cR^k$ satisfying Assumption~\ref{assump:approximate-average} under various assumptions on time-varying network $\{\cG^t\}_{t\geq 0}$.
\begin{figure}[htpb]
\vspace*{-4mm}%
\centering
\framebox{\parbox{0.98\columnwidth}{
{\small
\textbf{Algorithm DPDA-D} ( $\bxi^{0},\gamma,\{\tau_i,\kappa_i\}_{i\in\cN}$ ) \\[1.5mm]
Initialization: $v_i^0\gets\mathbf{0}$,\quad $y_i^0\gets\mathbf{0}$ \quad $i\in\cN$\\
Iteration $k$: ($k \geq 0$)\\
\text{ } 1. $v^c_i\gets \tfrac{1}{\gamma}v_i^k+y_i^k,\qquad v_i^{k+1}\gets \gamma v_i^c-\gamma \cR^k_i({\bv^c}), \quad i \in \cN,\quad \hbox{where}\quad \bv^c=[v_i^c]_{i\in\cN}$\\
%\text{ } 2. For $\ell=1,\ldots,q_k$\\
%\text{ } 3.\hspace*{10mm} $v^c_i\gets \sum_{j\in\cN_i(t_k+\ell)\cup\{i\}} V^{t_k+\ell}_{ij} v^c_j, \quad i \in \cN$\\
%\text{ } 4. End For\\[1mm]
\text{ } 2. $\xi_i^{k+1}\gets\prox{\tau_i\rho_i}\bigg(\xi_i^k-\tau_i\bigg(\nabla f_i(\xi_i^k)-\nv{\bJ g_i(\xi_i^k)^\top y^k_i}\bigg)\bigg),\quad i\in\cN$\\
%-\gamma\cR_i^k\left(\tfrac{1}{\gamma}\bv^k+\by^k\right), \quad i \in \cN$\\
\text{ } 3. $y_i^{k+1}\gets\cP_{\cK^\circ\cap\cB_0}\bigg(y_i^k-\kappa_i\Big(\nv{2g_i(\xi_i^{k+1})-g_i(\xi_i^k)}+{(2v_i^{k+1}-v_i^k)}\Big)\bigg), \quad i \in \cN$\\
 }\vspace*{-2mm}}}
 \vspace*{-2mm}
\caption{\small Distributed Primal-Dual Algorithm for Time-Varying $\{\cG^t\}$ (DPDA-D)}
\label{alg:PDDual}
\end{figure}
\vspace*{-4mm}%
\subsection{Inexact averaging operators}%\todo{Update the bounds on $\Gamma$ and $\zeta$.}
\label{sec:inexact-averaging}
Let $t_k\in\integers_{+}$ be the total number of \emph{communication rounds} done before the $k$-th iteration of DPDA-D, %shown
in Figure~\ref{alg:PDDual}, and let $q_k\in\integers_{+}$ be the number of communication rounds to be performed within the $k$-th iteration while evaluating $\cR^k$. %Recall that
According to Assumption~\ref{assump:approximate-average}, each node $i\in\cN$ can finish $\xi_i^{k+1}$ and $y_i^{k+1}$ computation within \emph{one} unit time, i.e., between two consecutive tics of the clock, for all $k\geq 0$, and communication rounds occur every time the global clock tics; hence, $\cG^t$ represents the connectivity network at the time of $t$-th communication round for all $t\in\integers_+$. Thus, only $\{\cG^t\}_{t\in\integers_+}$ among $\{\cG^t\}_{t\in\reals_+}$ is relevant since the topology of the time-varying network is only pertinent at those times when communication happens among neighboring nodes. \nv{For implementation in practice, it is sufficient for each node to count the number of global clock tics since the last update.}
\begin{defn}
Let $V^t\in\reals^{|\cN|\times|\cN|}$ be a matrix encoding the topology of $\cG^t=(\cN,\cE^t)$ in some way for $t\in\integers_+$. We define $W^{t,s}\triangleq V^tV^{t-1}...V^{s+1}$ for any $t,s\in\integers_+$ such that $t\geq s+1$.
\end{defn}
\nv{Let $\{\cG^t\}$ be a time-varying directed graph;
%$\cN_i^t$ is defined as in Definition~\ref{def:neighbors}, and $d_i^t=|\cN_i^t|$ for $i\in\cN$.
we adopt the information exchange model in \cite{nedic2009distributed-2} satisfying the assumptions stated in Assumption~\ref{assump:undirected}.} %-- see also \cite{nedic2009distributed}.
\begin{assumption}\label{assump:undirected}
%Let $V^t\in\reals^{|\cN|\times|\cN|}$ be the weight matrix corresponding to $\cG^t=(\cN,\cE^t)$ at the time of $t$-th consensus step and $\cN_i(t)\triangleq\{j\in\cN:\ (i,j)\in\cE^t~\hbox{ or }~(j,i)\in\cE^t\}$. %As in~\cite{chen2012fast}, we assume
\nv{%Suppose
For all $t\in\integers_+$: \textbf{(i)} every $i\in\cN$ knows $\cN^{\,t,{\rm out}}_i$ and there exists $\zeta\in(0,1)$ such that for %any
$i\in\mathcal{N}$, $V^t_{ij}\geq \zeta$ if $j\in\cN_i^{\,t,{\rm in}}$, and $V^t_{ij}=0$ otherwise. \textbf{(ii)} %Moreover,
$\cG^t$ is $M$-strongly-connected, i.e., there exist an $\integers \ni M{\geq}1$ (possibly unknown to nodes) such that the graph with edge set $\cE_M^k=\bigcup_{t=kM}^{(k+1)M-1}\cE^t$ is strongly connected for %every
$k\in\integers_+$.}
\begin{comment}
(\textbf{i}) $V^t$ is %a
doubly stochastic; %matrix;
(\textbf{ii}) there exists $\zeta\in(0,1)$ such that for %any
$i\in\mathcal{N}$, $V^t_{ij}\geq \zeta$ if {$j\in\mathcal{N}_i^t\cup\{i\}$}, and $V^t_{ij}=0$ if $j\notin\mathcal{N}_i^t$; (\textbf{iii}) $\cG^\infty=(\cN,\cE^{\infty})$ is connected where $\cE^{\infty}\triangleq\{(i,j)\in\cN\times\cN: (i,j)\in\cE^t\ \hbox{ for infinitely many }\ t\in\integers_+\}$, and there exists $\integers \ni M{\geq}1$ such that if $(i,j)\in \cE^{\infty}$, then $(i,j)\in \cE^t\cup \cE^{t+1}\cup ...\cup \cE^{t+M-1}$ for all $t\geq 1$.
%\begin{itemize}
%\item[(i)] For any $t$, $V^t$ is doubly stochastic matrix.
%\item[(ii)] There exists a scalar $\zeta\in(0,1)$, such that for any $i\in\mathcal{N}$ at time $t$, $V^t_{ij}\geq \zeta$ if $j\in\mathcal{N}_i$, and $V^t_{ij}=0$ if $j\notin\mathcal{O}_i$.
%\item[(iii)] Let $E_t\triangleq\{(j,i): \ j \ \text{receives the estimate of $i$ at time $t$}\}$ and $E_{\infty}\triangleq\{(j,i): \ j \ \text{receives the estimate of $i$ for infinitely many $t$}\}$, then $E_{\infty}$ is connected and if $(j,i)\in E_{\infty}$, then $(j,i)\in E_t\cup E_{t+1}\cup ...\cup E_{t+\ell-1}$ for some integer $\ell>1$.
%\end{itemize}
\end{comment}
\end{assumption}
\subsubsection{\bf Undirected $\{\cG^t\}_{t\in\integers_+}$}
\label{sec:undirected}
Let $\{\cG^t\}$ be a time-varying undirected graph; $\cN_i^t$ is defined as in Definition~\ref{def:neighbors}, and $d_i^t=|\cN_i^t|$ for $i\in\cN$. For undirected case, we assume %the sequence
$\{V^t\}_{t\in\integers_+}$ is \emph{doubly stochastic} and satisfies Assumption~\ref{assump:undirected}.
%can be chosen without coordination among all the nodes; and only simple communication among neighboring nodes are needed for this. In particular,
For instance, $V^t$ can be set as the Metropolis edge weight matrix~\cite{boyd2004fastest} corresponding to $\cG^t$, i.e., for each $i\in\cN$ set $V^t_{ij}=(\max\{d_i^t,~d_j^t\}+1)^{-1}$ for $j\in\cN^t_i$, $V^t_{ij}=0$ for $j\not\in\cN_i^t\cup\{i\}$ and $V_{ii}^t=1-\sum_{j\in\cN^t_i}V_{ij}^t$. Suppose that there exists $d_{\max}$ such that $d_i^t\leq d_{\max}$ for all $i\in\cN$ and $t\in\integers_+$. Under this assumption, it is trivial to check $\zeta=(d_{\max}+1)^{-1}$.

For %such a matrix
$V^t$ satisfying (\textbf{\emph{i}}) in Assumption~\ref{assump:undirected}, given any $w\in\reals^{|\cN|}$, the matrix-vector multiplication $V^tw\in\reals^{|\cN|}$ can be computed in a distributed way, i.e., the $i$-th component $(V^tw)_i=\sum_{j\in\cN_i\cup\{i\}}V^t_{ij}w_j$ can be computed at node $i\in\cN$ requiring only local communication of $i$ with nodes in $\cN_i^t$. \nv{The next result shows how this distributed operation can be used to approximate the average -- also see~\cite{nedic2009distributed}.}
%\begin{lemma}[\nv{Geometric rate~\cite{nedic2009distributed-2}}]\todo{We should update this result.}\label{lem:approximation}
%Under Assumption~\ref{assump:undirected}, %and $W^{t,s}=V^tV^{t-1}...V^{s+1}$ for $t\geq s+1$.
%for any given $s\geq 0$, the entries of $W^{t,s}$ converges to ${1\over N}$ as $t\rightarrow \infty$ with a geometric rate, i.e., %In particular,
%for all $i,j\in\mathcal{N}$, one has $\left|W^{t,s}_{ij} -{1\over N}\right|\leq \Gamma \alpha^{t-s}$,
%where $\Gamma\triangleq 2(1+\zeta^{-\bar{T}})/(1-\zeta^{\bar{T}})$, $\alpha\triangleq (1-\zeta^{\bar{T}})^{1/\bar{T}}$, and $\bar{T}\triangleq (N-1)M$.
%\end{lemma}
\nv{\begin{lemma}\label{lem:approximation}
%(Lemma 5 in \cite{nedic2009distributed-2})%\todo{Erfan updated this result.}
Let $\{V^t\}_{t\in\integers_+}$ be a sequence of \emph{doubly stochastic} matrices satisfying Assumption~\ref{assump:undirected}.
%\todo{No!! under much weaker assumptions...},
For any $s,t\in\integers_+$ such that $t\geq s$,
%the entries of $W^{t,s}$ converges to $\frac{1}{N}$ as $t\to \infty$ with a geometric rate, i.e., for all $i,j\in\cN$, one has
$\norm{(W^{t,s}\otimes\id_m)\bw-\one\otimes p(\bw)}\leq \tfrac{8}{7}\alpha^{t-s}\norm{\bw}$ for any $\bw=[w_i]_{i\in\cN}\in\reals^{n_0}$,
%such that $w_i\in\reals^m$ for $i\in\cN$,
where $\alpha=(1-\frac{\zeta}{2N^2})^{\frac{1}{2M}}$.
\end{lemma}}
\begin{proof}
\nv{The proof immediately follows from Lemma 5 in~\cite{nedic2009distributed-2}.}
\end{proof}
For $\bw=[w_i]_{i\in\cN}\in\reals^{n_0}$ such that $w_i\in\reals^m$ for $i\in\cN$, define
\vspace*{-1mm}%
{\small
\begin{equation}
\label{eq:approx-average-dual}
\cR^k(\bw)\triangleq\mathcal{P}_{\mathcal{B}}\left((W^{t_k+q_k,t_k}\otimes\id_m)~\bw\right)
\end{equation}}%
to approximate $\mathcal{P}_{\Ct}(\cdot)$ in \eqref{eq:exact-w-computation}. Note that $\mathcal{R}^k(\cdot)$ can be computed in a \emph{distributed fashion} requiring $q_k$ communications with the neighbors for each node.
In particular, components of $\cR^k(\bw)$ can be computed at each node as
$\cR^k(\bw)=[\cR_i^k(\bw)]_{i\in\cN}$ such that $\cR_i^k(\bw)\triangleq\cP_{\cB_0}\Big(\sum_{j\in\cN_i\cup\{i\}}W^{t_k+q_k,t_k}_{ij}w_j\Big)$.
Moreover, the approximation error, %of using inexact rule \eqref{3.3.1} is
$\mathcal{R}^k(\bw)-\mathcal{P}_{{\Ct}}(\bw)$, for any $\bw$ can be bounded as in %\eqref{eq:approx_error} due to
\eqref{eq:approx_error-for-full-vector} using the non-expansivity of $\cP_{\cB}$ and Lemma~\ref{lem:approximation}. %\sa{Let $\norm{.}_p$ denote the $\ell_p$-norm.}
\begin{comment}
Indeed, from \eqref{eq:proj-2}, we get for all $i\in\cN$,
{\small
\begin{align}
\|\mathcal{R}_i^k(\bw)-\cP_{\cB_0}\big(\tilde{p}(\bw)\big)\|&=\|\cP_{\cB_0}\big(\sum_{j\in\cN}W^{t_k+q_k,t_k}_{ij}w_j\big)-\cP_{\cB_0}\big(\tfrac{1}{N}\sum_{j\in\mathcal{N}}w_j\big)\| \nonumber\\
&\leq\|\sum_{j\in\cN}\big(W_{ij}^{t_k+q_k,t_k}-\tfrac{1}{N}\big)w_j\|\leq \sqrt{N}~\Gamma \alpha^{q_k}\norm{\bw}. \label{eq:approx_error}
\end{align}
}%
\end{comment}
%Thus, \eqref{eq:proj-2} and \eqref{eq:approx_error} imply that one can bound the approximation error as in \eqref{eq:approx_error-for-full-vector}.
More precisely, $\cR^k$ defined in \eqref{eq:approx-average-dual} satisfies Assumption~\ref{assump:approximate-average}.
\subsubsection{\bf Directed $\{\cG^t\}_{t\in\integers_+}$}
\label{sec:directed}
Let $\{\cG^t\}$ be a time-varying directed graph, and $\cN^{\,t,{\rm in}}_i$, $\cN^{\,t,{\rm out}}_i$ be defined as in Definition~\ref{def:neighbors} for $i\in\cN$. \nv{Recall $d_i^t=|\cN^{\,t,{\rm out}}_i|-1$.} Since the definition of $\Ct$ in \eqref{eq:C} does not depend on the topology of the network, using the push-sum protocol~\cite{kempe2003gossip} within DPDA-D, one can also handle time-varying \emph{directed} communication networks. Indeed, given any $\bw=[w_i]_{i\in\cN}$, nodes can inexactly compute $\mathcal{P}_{\Ct}(\bw)$ in a distributed fashion %via push-sum iterations
with increasing approximation quality; %as the number of communications increases. Indeed,
consider the weight-matrix sequence $\{V^t\}_{t\in\integers_+}$: for any $t\geq 0$, %define $V^t\in\reals^{|\cN|\times |\cN|}$:
\vspace*{-1mm}
{\small
\begin{align}
\label{eq:directed-weights}
V^t_{ij}={1\over \nv{d_j^t+1}}\ \hbox{ if }\ j\in\cN^{\,t,{\rm in}}_i;\quad V^t_{ij}=0\ \hbox{ if }\ j\not\in\cN^{\,t,{\rm in}}_i,\quad i\in\cN.
\end{align}
}%where $d_j^t=|\cN^{\,t,{\rm out}}_j|$ for $j\in\cN$.
For $\bw=[w_i]_{i\in\cN}\in\reals^{n_0}$ such that $w_i\in\reals^m$ for $i\in\cN$, define
{\small
\begin{equation}
\label{eq:approx-average-dual-directed}
\cR^k(\bw)\triangleq\mathcal{P}_{\mathcal{B}}\left({\diag(W^{t_k+q_k,t_k}\ones_N {\otimes\id_m})^{-1}}~(W^{t_k+q_k,t_k}\otimes\id_m)~\bw\right)
\end{equation}
}%
to approximate $\mathcal{P}_{\Ct}(\cdot)$ in \eqref{eq:exact-w-computation}. $\mathcal{R}^k(\cdot)$ can be computed in a \emph{distributed fashion} requiring $q_k$ communication rounds, and \nv{is a compact representation of push-sum operation.}
\begin{lemma}
\label{lem:push-sum}
Consider $\cR^k$ defined in \eqref{eq:approx-average-dual-directed} for $k\geq 0$. %the sequence $\{\vartheta_i^t\}_{t\in\integers_+}$ generated by push-sum iterations in \eqref{push-sum} for $i\in\cN$.
Assuming %that the digraph sequence
$\{\cG^t\}_{t\in\integers_+}$ is uniformly strongly connected ($M$-strongly connected), \eqref{eq:approx_error-for-full-vector} holds for some
%the following holds for all $\bw=[w_i]_{i\in\cN}\in\reals^{n_0}$ and $k\geq 0$:
%\vspace*{-1mm}
%{\small
%\begin{align}
%|\vartheta_i^{t}-{\ones^\top {w}\over N}|\leq {\Gamma}\alpha^t \norm{{w}}_1, \quad \forall i\in\cN,
%\end{align}
%}%
$\Gamma >0$ and $\alpha\in(0,1)$ such that
$\Gamma\leq {8 N^{NM}}$ and $\alpha\leq \left(1-{1\over N^{NM}}\right)^{1\over M}$.
%\todo{Consider the extra $1/\alpha$ term for $\Gamma$.}
\end{lemma}
\begin{proof}
The result follows from the proof Lemma~1 in~\cite{nedic2015distributed}. %and non-expansivity of $\cP_{\cB}$.
\end{proof}
\begin{comment}
Moreover, the approximation error, %of using inexact rule \eqref{3.3.1} is
$\mathcal{R}^k(\bw)-\mathcal{P}_{{\Ct}}(\bw)$, for any $\bw$ can be bounded as in \eqref{eq:approx_error-for-full-vector} due to non-expansivity of $\cP_{\cB}$ and using Lemma~\ref{lem:push-sum}. More precisely, $\cR^k$ defined in \eqref{eq:approx-average-dual-directed} also satisfies Assumption~\ref{assump:approximate-average}.
\end{comment}
\section{Convergence of Algorithm DPDA-D}
\label{sec:convergence}
\sa{Define $\bar{C}_g\triangleq\sum_{i\in\cN}C_{g_i}/N$ and $\bar{R}_x\triangleq \max\{\norm{\bxi^*-\bxi^0}_{\bL_g},\norm{\bxi^*-\bxi^0}_{\bL'}\}/\sqrt{N}$, where $\bL'\triangleq\diag([(1+L_{f_i}+C_{g_i})\id_{n_i}]_{i\in\cN})$ and $\bL_g\triangleq\diag([(L_{g_i})\id_{n_i}]_{i\in\cN})$.}
\begin{theorem}\label{col:dual-error-bound}
\nv{{Suppose Assumptions~\ref{assum:functions}, \ref{assump:communication_general}, \ref{assump:existence} and~\ref{assump:approximate-average} hold}. For any $\gamma>0$, let the primal-dual step-sizes $\{\tau_i,\kappa_i\}_{i\in\cN}$ be chosen such that for some $\beta>0$,
\nv{\small
\begin{equation}\label{eq:step-rule-dual}
\tau_i=(\max\{1, L_{f_i}+\beta L_{g_i}\}+C_{g_i})^{-1},\quad \kappa_i=(C_{g_i}+\frac{5\gamma}{2})^{-1},\quad \forall\ i\in\cN.
\end{equation}}%
%Let $({\bf \bxi}^*,\bw^*,\by^*)$ be an arbitrary saddle-point for $L$ in \eqref{eq:lagrangian-dual-implementation}.
\nv{Given $B\in(0,\infty]$,} starting from $\bv^0=\by^0=\mathbf{0}$ and an arbitrary $\bxi^0$, let $\{(\bxi^k,\bv^k)\}_{k\geq 0}$ be the primal, and $\{\by^k\}_{k\geq 0}$ be the dual iterate sequence generated by Algorithm DPDA-D, displayed in Fig.~\ref{alg:PDDual}, using $q_k\in \integers_+$ communication rounds for the $k$-th %tick of the global clock
iteration such that \sa{$C_0\triangleq\sum_{k=0}^\infty \alpha^{q_{k}}(k+1) < \infty$}.
%\todo{$C_0\triangleq\sum_{k=0}^\infty \alpha^{q_{k}}(k+1) < \infty$}.
%for all $k\geq 0$} for some $c>0$.
For any $\gamma>0$, if $\beta>0$ is chosen as discussed below, then $\{({\bxi}^k,{\by}^k)\}_{k\geq 0}$ converges to $(\bxi^*,\by^*)$ such that $\by^*=\one \otimes y^*$ and $(\bxi^*,y^*)$ is {an optimal} primal-dual solution to \eqref{eq:central_problem}. Moreover,
%the following bounds on
both infeasibility, $F(\bar{\bxi}^K,\bar{\by}^K)$, and suboptimality, $|\varphi(\bar{\bxi}^K)-\varphi({\bxi}^*)|$ are $\cO(1/K)$, i.e., for all $K\geq 1$:
\nv{\small
\begin{eqnarray}
&&F(\bar{\bxi}^K,\bar{\by}^K)\triangleq d_{\cC}(\bar{\by}^K)+\norm{y^*}d_{\mathcal{K}}\left(- g(\bar{\bxi}^K)\right)\leq \frac{\Lambda(\gamma,\beta)}{K}, \label{eq:infeasibility}\\
&&0\leq \varphi(\bar{\bxi}^K)-\varphi({\bxi}^*)+\norm{y^*} d_{\mathcal{K}}\left(-g(\bar{\bxi}^K)\right)\leq \frac{\Lambda(\gamma,\beta)}{K}-F(\bar{\bxi}^K,\bar{\by}^K),\label{eq:suboptimality}
%\lefteqn{\Lambda_1\triangleq{1\over 2\gamma}+\sum_{i\in\mathcal{N}}\bigg[{1\over \tau_i}\|\xi^*_i-\xi_i^0\|^2+{4\over \kappa_i}\|y^*\|^2\bigg],}\nonumber \\
%&&\Lambda_2(K) \triangleq 3\max_{k=0,\ldots,K-1}\Big\{1+\norm{\bw^{k+1}}+\gamma\norm{\by^{k+1}}+2\norm{\by^*}\Big\}\sum_{k=0}^{K-1}\norm{\be^{k+1}} \nonumber \\
\end{eqnarray}
}%
for some $\Lambda(\gamma,\beta)\in\reals_+$, where $\bar{\bxi}^K={1\over K}\sum_{k=1}^K\bxi^k$ and $\bar{\by}^K={1\over K}\sum_{k=1}^K\by^k$ for $K\geq 1$.} %\todo{$F(\bar{\bxi}^K,\bar{\by}^K)=\cO(\tfrac{1}{K})$ and $|\varphi(\bar{\bxi}^K)-\varphi({\bxi}^*)|=\cO(\tfrac{1}{K})$}

\nv{(CASE 1): If a dual bound is known, i.e., $B<\infty$, then \eqref{eq:infeasibility} and \eqref{eq:suboptimality} hold for $\beta=2B$; moreover, setting the free parameter {$\gamma=(N^{3/2}\Gamma C_0 B)^{-1}$} gives %$\Lambda_1+\Lambda_2(K)=\cO\big(N^2\Gamma^2C_0^2+NB(\bar{R}_x^2+\bar{C}_g B)\big)$
{\small
\begin{align}
\label{eq:Lambda_Bknown}
\Lambda(\gamma,\beta)=\cO\Big(NB(\bar{R}_x^2+\bar{C}_g B)+N^{3/2}\Gamma C_0 B\Big).
\end{align}}%
\indent (CASE 2): If the dual bound does not exist, set $B=\infty$ within DPDA-D. Assuming $q_k\geq \log_{1/\alpha}(24 N\Gamma(k+1))$ for $k\geq 0$, there exists $\bar{\beta}>0$ such that \eqref{eq:infeasibility} and \eqref{eq:suboptimality} hold for all $\beta\geq \bar{\beta}$; moreover, selecting \sa{$\gamma=N^{\frac{3}{2}}\Gamma {C_0}\bar{R}_x^2$} gives $\Lambda(\gamma,\beta)=\cO(N^{\frac{9}{2}}\Gamma^3 {C_0^3}\bar{R}_x^2\sa{\max\{1,\norm{y^*}^2\}})$. Finally, when $g_i$ is affine %linear functions
for $i\in\cN$ and $\{\tau_i\}$ are independent of $\beta$, \sa{$\gamma=(N^{\frac{3}{2}}\Gamma {C_0})^{-1}$} leads to $\Lambda(\gamma,\beta)=\cO(N^3\Gamma^2 {C_0^2}(\bar{R}_x^2+\bar{C}_g\sa{\max\{1,\norm{y^*}^2\}}))$.}
\end{theorem}
%\begin{remark}
%Note that for all $i\in\cN$ by choosing \ey{$\tau_i= {1\over c_i+L_{f_i}+2BL_{g_i}}$ and $\kappa_i<{c_i \over \gamma + C_{g_i}^2}$} for any $c_i>0$, the condition in \eqref{eq:step-rule-dual} is satisfied.
%\end{remark}
\begin{remark}
We assume agents know $q_k$ as a function of $k$ at the initialization; hence, synchronicity can be achieved among nodes if simply each node counts the number of times the global clock tics, where at each tic one communication round occurs according to Assumption~\ref{assump:approximate-average}.
%-- assuming that at any time $t$, all the nodes know their neighbors in $\cG^t$ and communication between neighbors happens instantaneously.
\end{remark}
\begin{remark}
\label{rem:qk}
\begin{comment}
Note that the suboptimality, infeasibility and consensus violation at the $K$-th iteration is $\cO(\Lambda_2(K)/K)$, where $\Lambda_2(K)$ denotes the error accumulation due to approximation errors %defined in~\eqref{eq:prox-error-seq},
$\mathcal{P}_{\Ct}(\cdot)-\mathcal{R}^k(\cdot)$, and $\Lambda_2(K)$ can be bounded above for all $K\geq 1$ as $\Lambda_2(K)\leq C_2 \sum_{k=1}^K\alpha^{q_{k-1}} k$. Therefore,
\end{comment}
\nv{Suppose we are given $(0,1)\ni\bar{\alpha}\geq \alpha$. For any $c>0$, choosing
%$\{q_k\}_{k\in\integers_+}$ as stated
$q_k=\big\lceil(2+c)\log_{\tfrac{1}{\bar{\alpha}}}(k+1)\big\rceil$ for $k\geq 0$ satisfies the condition in Theorem~\ref{col:dual-error-bound}, i.e., $C_0=\sum_{k=0}^{\infty}\alpha^{q_k}\sa{(k+1)}\leq \tfrac{1}{c}+1$.}
%ensures that %$\sum_{k=1}^{\infty}\alpha^{q_{k-1}}k(k+1)$
%\eyh{$\sup_{K\in\integers_+}\Lambda_2(K)< C_2 \left(\tfrac{1}{c}+1\right)$.}
%; hence, $\Lambda_2(K)\leq \tfrac{2}{c}C$.
%\end{remark}
%\begin{remark}
%\label{rem:qk}
%For some $c>0$, setting $q_k=(3+c)\log_{\tfrac{1}{\alpha}}(k+1)$ for $k\geq 0$
\nv{Moreover, this choice of $\{q_k\}_{k\in\integers_+}$ implies that
the total number of communication rounds right before the $K$-th iteration is equal to
$t_K=\sum_{k=0}^{K-1}q_k=(2+c)\big[(K-1)\log_{\tfrac{1}{\bar{\alpha}}}(K)+\log_{\tfrac{1}{\bar{\alpha}}}(e)\big]$ where $e$ is Euler's number.}
\end{remark}
\begin{comment}
There are other possible choices for $\{q_k\}_{k\geq 0}$. The following lemma is slight extension of Proposition 3 in~\cite{chen2012fast}, where it is stated for $p=1$; its proof is omitted.
%due to limited space.
\begin{lemma}\label{lemsum}
Let $\alpha\in(0,1)$, $p\geq 1$ is a rational number, and $d\in\integers_+$. Define $P(k,d)=\{\sum_{i=0}^d c_i k^i:\ c_i\in\reals\ i=1,\ldots,d \}$
denote the set of polynomials of $k$ with degree at most $d$. Suppose \sa{$h^{(k)}\in P(k,d)$} for $k\geq 1$, then
$\sum_{k=0}^{+\infty}\sa{h^{(k)}}\alpha^{\sqrt[\leftroot{-3}\uproot{3}p]{k}}$ is finite.
\end{lemma}
Lemma~\ref{lemsum} implies that the error accumulation can be controlled for this choice of $\{q_k\}_{k\geq 0}$.
\end{comment}
\begin{corollary}
\label{cor:T-rate}
\nv{
Under the premise of Theorem~\ref{col:dual-error-bound}, let $\{\cG^t\}$ be an undirected time-varying graph and $\{q_k\}$ be as in Remark~\ref{rem:qk} with $(0,1)\ni\bar{\alpha}=\iota \alpha$ for some $\iota>1$. Let $Q(\epsilon)$ be the total number of communications needed to compute an $\epsilon$-optimal and $\epsilon$-feasible solution $(\bxi^\epsilon,\by^\epsilon)$ for $\gamma=1/\cO(\sqrt{N})$, i.e., $F(\bxi^\epsilon,\by^\epsilon)<\epsilon$ and $|\varphi(\bxi^\epsilon)-\varphi(\bxi^*)|<\epsilon$. If a dual bound $B<\infty$ is known, then $Q(\epsilon)=\cO(\frac{N^4}{\epsilon}\log(\frac{N}{\epsilon}))$. If a Slater point does not exist, i.e., $B=\infty$, then $Q(\epsilon)=\cO(\frac{N^{4.5}}{\epsilon}\log(\frac{N^{1.5}}{\epsilon}))$; moreover, $Q(\epsilon)=\cO(\frac{N^{4}}{\epsilon}\log(\frac{N}{\epsilon}))$ is achieved when $g_i$ is an affine function for $i\in\cN$.}
%\todo{Maybe we should change the notation $T(\epsilon)$ to $\cT(\epsilon)$, to prevent from confusion with map $T(\bx)$.}
\end{corollary}
\begin{proof}
\nv{Theorem~\ref{col:dual-error-bound} implies that $(\bxi^\epsilon,\by^\epsilon)$ can be computed in $K^\epsilon=\Lambda(\gamma,\beta)/\epsilon$ DPDA-D iterations which requires $t_{K^\epsilon}=\cO(K^\epsilon\log(K^\epsilon)/\log(\frac{1}{\alpha}))$ communications in total -- see Remark~\ref{rem:qk}. Lemma~\ref{lem:approximation} implies that $\Gamma=1/N$; hence, setting $\gamma$ as described in Theorem~\ref{col:dual-error-bound}, we bound $\Lambda(\gamma,\beta)$ with $\cO(N)$ for CASE 1, $\cO(N^{1.5})$ for CASE 2 in general and with $\cO(N)$ when $g_i$'s are linear. Thus, the result follows from $\log(\frac{1}{\alpha})\geq \zeta/N^2$, where $\zeta$ can be as small as $\cO(1/N)$.}
\end{proof}
\nv{Note that when $\{\cG^t\}$ is a general time-varying directed graph, we employ push-sum protocol with $\Gamma=N^{NM}$ (see \cref{lem:push-sum}) which leads to exponential $\cO(1)$ bounds, e.g., $\Lambda(\gamma,\beta)=\cO(N^{NM+\frac{3}{2}}B)$ for CASE 1. To our best knowledge, polynomial bounds for directed graphs in $N$ is still an open question~\cite{nedic2016stochastic}.
%\todo{\eyh{I added a paper}}.
That said, setting $\{q_k\}$ as in Remark~\ref{rem:qk}, DPDA-D can compute an $\epsilon$-solution in $\cO(\frac{1}{\epsilon}\log(\frac{1}{\epsilon}))$ communications even for general directed graphs {and choosing $q_k=(2+c)\log_{\frac{1}{\bar{\alpha}}}(k+1)$ in CASE 1 leads to $\cO(1)$ constant $N^{2NM+1.5}$} which is better than $\cO(\frac{1}{\epsilon^2}\log^2(\frac{1}{\epsilon}))$ result in~\cite{gu2018distributed,nedic2015distributed} {with $\cO(1)$ constant of $N^{2NM+2}$}. The method in~\cite{gu2018distributed} has $\log(k)/\sqrt{k}$ rate and requires exact minimization of convex $f_i$ over compact $\cX_i$ at each iteration. The method in~\cite{nedic2015distributed} can be used to solve the dual of \eqref{eq:central_problem} when $\rho_i$ is the indicator function of some compact convex set $\cX_i$ and $f_i$ is convex for $i\in\cN$; but, the subproblem that needs to be solved at each iteration is fairly complicated as in~\cite{gu2018distributed}.}
\begin{remark}\label{rem:poly}
Since $\sum_{k=1}^\infty\alpha^{\sqrt[\leftroot{-3}\uproot{3}p]{k}} k<\infty$ for any $p\geq 1$, if one chooses $q_k=\sqrt[\leftroot{-3}\uproot{3}p]{k}$ for $k\geq 1$, then
%the total number of communications per node until the end of $K$-th iteration can be bounded by
$t_K=\sum_{k=0}^{K-1} q_k=\cO(K^{1+1/p})$. This choice of $\{q_k\}_{k\in\integers_+}$, unlike the one in Remark~\ref{rem:qk},
is independent of the parameter $\alpha\in(0,1)$; but \sa{leads to} a larger $C_0=\sum_{k=0}^{\infty}\alpha^{q_k}\sa{(k+1)}=\cO(\alpha/\log^{2p}(\alpha))$ for $\alpha\in(1/e, 1)$. \nv{On the other hand, apriori running DPDA-D, a practical way to estimate $\alpha\in (0,1)$ is to run an average consensus iterations with a random initialization until iterates stagnate around the average; this leads to a rate coefficient $\alpha_i$ for $i\in\cN$. Next, nodes can do a max consensus to compute $\bar{\alpha}=\max_{i\in\cN}\alpha_i$ and use it to set $q_k=(2+c)\log_{\frac{1}{\bar{\alpha}}}(k+1)$.} %which is dependent on
%which is determined by the dynamics of $\{\cG^t\}_{t\in\integers_+}$ and may not be known by the agents in $\cN$.
%For large $p$, $q_k$ grow slowly, it makes the method more practical at the cost of longer convergence time due to increase in $\cO(1)$ constant. Note that $q_k= (\log(k+1))^2$ also works and it grows very slowly.
\end{remark}
\begin{remark}\label{remark:transient_step}
\nv{Suppose the dual bound is not available. %Let ${\tilde{T}\triangleq 24N\Gamma}$. %where $\lambda_{\min}$ is the minimum eigenvalue of $\bf\bar{Q}$.
If $q_k=(2+c)\log_{1/\bar{\alpha}}(k+1)$ for some $c>0$ and $(0,1)\ni\bar{\alpha}\geq\alpha$, then $q_k\geq \log_{1/\alpha}(24N\Gamma(k+1))$ for all $k\geq \tilde{K}\triangleq \lceil(24N\Gamma)^{1/(1+c)} \rceil$.
If $q_k=\sqrt[\leftroot{-3}\uproot{3}p]{k}$ for some $p\geq 1$, then $q_k\geq \log_{1/\alpha}(24N\Gamma(k+1))$ for all $k\geq \tilde{K}= \lceil (\log_{1/\alpha}(24N\Gamma)+p\log_{1/\alpha}p)^p \rceil$.} %-- see Appendix \ref{append} for the proof.
\nv{Hence, the rate results of Theorem~\ref{col:dual-error-bound} will hold after the transient period of $\tilde{K}$ iterations.}
%\todo{you have not defined $\bar{Q}$ yet.}
\end{remark}
\subsection{Auxiliary results to prove Theorem~\ref{col:dual-error-bound}}
Let $\{\bxi^k,\bv^k,\by^k\}_{k\geq 0}$ be the iterate sequence generated by DPDA-D as shown in Figure~\ref{alg:PDDual} and $\{\bw^{k}\}_{k\geq 0}$ be the auxiliary sequence where $\bw^k$ is given in~\eqref{eq:exact-w-computation} for $k\geq 1$ and we set $\bw^0\triangleq\bv^0=\mathbf{0}$. We first define the %proximal
error sequence $\{\be^k\}_{k\geq 0}$: let $\be^{k}\triangleq (\bv^{k}-\bw^{k})/\gamma$ for all $k\geq 0$; hence, \sa{$\be^0=\be^1=\mathbf{0}$}
%\todo{Note that $\be^1=\mathbf{0}$}
and for $k\geq 0$, we have
%in \eqref{eq:prox-error-seq}, which will be used later for analyzing the convergence of Algorithm~DPDA-D displayed in Fig.~\ref{alg:PDDual}.
\vspace*{-2mm}
{\small
\begin{equation}
\label{eq:prox-error-seq}
\be^{k+1} = \cP_{\Ct}\left(\tfrac{1}{\gamma}{\bv}^k+{\by}^{k}\right)-\cR^k\left(\tfrac{1}{\gamma}{\bv}^k+{\by}^k \right).
\vspace*{-2mm}
\end{equation}
}%
%hence, $\bv^{k}=\bw^{k}+\gamma \be^{k}$ for $k\geq 1$. In the rest, we assume that $\bv^0=\bw^0=\mathbf{0}$; \nv{and set }.

In order to prove Theorem~\ref{col:dual-error-bound}, we first prove Lemma~\ref{thm:dynamic-rate} which help us to bound $\mathcal{L}(\bxi^k,\bv^k,\by)-\mathcal{L}(\bxi,\bv,\by^k)$ for any given $(\bxi,\bv,\by)\in\cZ$ and $k\geq 1$, where $\cL$ is defined in~\eqref{eq:lagrangian-dual-implementation}; and then we provide a few other technical results which will be used together with Lemma~\ref{thm:dynamic-rate} to show the asymptotic convergence of $\{\bxi^k,\bv^k,\by^k\}$ in Theorem~\ref{col:dual-error-bound}.
\begin{defn}
\label{def:Q}
Let $\mathbf{D}_\gamma$ and $\mathbf{D}_\kappa$ be the diagonal matrices given in Definition~\ref{def:bregman}.
\nv{Define a diagonal matrix $\bC\triangleq \diag([C_{g_i}]_{i\in\cN})$, and $H\triangleq [\bC~~\id_{N}]$. %\in\reals^{m|\cN|\times (n+m|\cN|)}$
%be the matrix specified in Definition~\ref{def:saddle-definitions}.
Given some $\beta>0$, define the symmetric matrix
{\small
$\mathbf{\bar{Q}}(\beta)\triangleq\begin{bmatrix}
\mathbf{\bar{D}}(\beta) & -H^\top\\
-H & \mathbf{\bar{D}}_\kappa
\end{bmatrix}$}, where
{\small $\mathbf{\bar{D}}(\beta)\triangleq
\begin{bmatrix}
\mathbf{\bar{D}}_\tau(\beta) & \zero\\
\zero & \tfrac{1}{\gamma}\id_N
\end{bmatrix}$,} $\mathbf{\bar{D}}_\tau(\beta)\triangleq
\diag([{1\over \tau_i}-\max\{1,L_{f_i}+\beta L_{g_i}\}]_{i\in \mathcal{N}})$ and $\mathbf{\bar{D}}_\kappa\triangleq \diag([\tfrac{1}{\kappa_i}]_{i\in\cN})$. Let
%for any $\bxi$, $\bar{\bxi}$, $\bv$, $\bar{\bv}$, $\by$ and $\bar{\by}$
$\bu:\cZ\times\cZ\rightarrow\reals^{3N}$ such that $\bu(\bz,\bar{\bz}) \triangleq \begin{bmatrix}[\norm{\xi_i-\bar{\xi}_i}]^\top_{i\in\cN} & [\norm{w_i-\bar{w}_i}]^\top_{i\in\cN} & [\norm{y_i-\bar{y}_i}]^\top_{i\in\cN}\end{bmatrix}^\top\in\reals^{3N}$.}
%to be the corresponding vector of norms where $\bz$ is defined in Definition \ref{def:bregman}
%where $\cZ\triangleq\cX\times\cY$ and $\cZ\ni\bz=[\bx^\top \by^\top]^\top$.
\end{defn}
\begin{lemma}\label{thm:dynamic-rate}
%Let $\bx=[\bxi^\top \bv^\top]^\top$ such that $\bv\in\reals^{m|\cN|}$, $\bxi=[\xi_i]_{i\in \mathcal{N}}\in\reals^n$, and $n\triangleq\sum_{i\in\cN}{n_i}$; and
Let $\cX$, $\cY$ and $\cZ$ be the spaces defined in Definition~\ref{def:bregman}. Suppose $\{\tilde{\bx}^k\}_{k\geq 0}\subset\cX$ be the primal and
$\{\by^k\}_{k\geq 0}\subset\cY$ be the dual iterate sequences generated by
%{using step-size sequence $\{[\tau_i]_{i\in\cN},[\sigma_i]_{i\in\cN},\gamma\}_{k\geq0}$} in
Algorithm DPDA-D %, displayed
in Fig.~\ref{alg:PDDual}, using some positive stepsizes: $\{\tau_i,\kappa_i\}_{i\in\mathcal{N}}$ and $\gamma$, and initializing from an arbitrary %$\tilde{\bx}^0\in\cX$ and $\by^0\in\cY$
$\bxi^0$ and $\bv^0=\by^0=\mathbf{0}$, where $\tilde{\bx}^k=[{\bxi^k}^\top~{\bv^k}^\top]^\top$ for $k\geq 0$. %Let $\bar{\bx}^{K}\triangleq\tfrac{1}{K}\sum_{k=1}^{K}\bx^k$, and $\bar{\by}^{K}\triangleq\tfrac{1}{K}\sum_{k=1}^{K}\by^k$ for $K\geq 1$. Suppose primal-dual step-sizes $\{\tau_i,\kappa_i\}_{i\in\cN}$ and $\gamma>0$ be chosen such that $\tau_i,\kappa_i>0$ and
%{\small
%\begin{align}\label{eq:step-rule-dynamic}
%\Big(\frac{1}{\tau_i}-L_i\Big)\Big(\frac{1}{\kappa_i}-\gamma\Big) \geq \sigma^2_{\max}(R_i),\quad \forall\ i\in\mathcal{N}.
%\end{align}
%}%
{Define $\{\bx^k\}$ and $\{\bz^k\}$ %be the modified sequences
such that $\bx^k=[{\bxi^k}^\top~{\bw^k}^\top]^\top\in\cX$ and ${\bz}^k=[{\bx^k}^\top {\by^k}^\top]^\top\in\cZ$ for $k\geq 0$. Let $\{\beta_k\}_{k\geq 0}$ \sa{be} such that $\beta_k\geq \max_{i\in\cN}\norm{y_i^k}$ for $k\geq 0$, then for any ${\bx}=[\bxi^\top~\bw^\top]^\top\in \cX$, and $\by\in\cY$,
%the primal-dual iterate sequence
$\{{\bz}^k\}_{k\geq 0}\subset\cZ$ satisfies %the following inequality,
\vspace*{-1mm}
\nv{\small
\begin{align}
\label{thmeq:inexact-lagrangian-bound}
\mathcal{L}({\bx}^{k+1},\by)-\mathcal{L}({\bx},\by^{k+1})\leq &\left[D_x({\bx},{\bx}^k)+D_y(\by,\by^k)-\fprod{T({\bx})-T({\bx}^k),~\by-\by^k}\right] \nonumber \\
&-\left[D_x({\bx},{\bx}^{k+1})+D_y(\by,\by^{k+1})-\fprod{T({\bx})-T({\bx}^{k+1}),~\by-\by^{k+1}}\right]\nonumber \\
& +E^{k+1}(\bz)-\tfrac{1}{2}\bu({\bz}^{k+1},{\bz}^k)^\top\mathbf{\bar{Q}}(\beta_k)~\bu({\bz}^{k+1},{\bz}^k),\quad \forall~k\geq 0,
\end{align}
}%
%for all $k\geq 0$,
where \nv{$\bu(\cdot,\cdot)$ is given in Definition~\ref{def:Q}}, ${\bz}^k=[{\bx^k}^\top\ {\by^k}^\top]^\top$,
$D_x$ and $D_y$ are Bregman functions %defined as
%given
in Definition~\ref{def:bregman},
%$T=[R ~-\bI_{m|\cN|}]$ for block-diagonal matrix $R\triangleq\diag([R_i]_{i\in\mathcal{N}})\in\reals^{m \times n|\cN|}$,
\nv{$T(\cdot)$ is given in Definition~\ref{def:saddle-definitions}, $E^{k+1}(\bz)\triangleq \norm{\be^k}\norm{\bw-\bw^{k+1}}+\gamma\norm{2\be^{k+1}-\be^k}\norm{\by-\by^{k+1}}$ and $\be^{k}\triangleq (\bv^{k}-\bw^{k})/\gamma$ for $k\geq 0$.} %, and $\{\be^k\}_{k\geq 0}$ is the proximal error sequence. %defined as in \eqref{eq:prox-error-seq} with $\be^0=\mathbf{0}$.
}
%\ey{Moreover, $\mathbf{\bar{Q}}^k$ is defined similar to $\mathbf{\bar{Q}}$ by replacing $\mathbf{\bar{D}}_\tau$ with $\mathbf{\bar{D}}^k_\tau\triangleq \diag([\frac{1}{\tau_i}-(L_{f_i}+\norm{y_i^k}L_{g_i})]_{i\in\cN})$ in Definition \ref{def:Q}, for $k\geq 0$.}
%\|\be^{k}\| \big({4\gamma\sqrt{N}~B}+{\|\bv-\bv^{k}\|}\big)$ for $k\geq 1$.}
%Then, for any $\bx$ and $\by$, the following holds
%{\small
%\begin{align}\label{eq:dynamic-saddle-rate}
%\mathcal{L}(\bar{\bf x}^K,\by)-\mathcal{L}(\bx,\bar{\by}^K) \leq &{1\over K} \big[D_x(\bx,\bx^0)+D_y(\by,\by^0)-\fprod{T(\bx-\bx^0),~\by-\by^0}\big] \nonumber \\
%&+{1\over K}\sum_{k=0}^{K-1}\bigg[\|\be^{k+1}\| \big(2\gamma\sqrt{N}~B_d+\|\bv-\bv^{k+1}\|\big)\bigg],
%\end{align}
%}%
%where $D_x$, $D_y$ are Bregman functions defined as in Definition~\ref{def:bregman}, $T=[R^\top R_0^\top]^\top$ for block-diagonal matrix $R\triangleq\diag([R_i]_{i\in\mathcal{N}})\in\reals^{m|\cN| \times n}$ and $R_0=-\id_{m|\cN|}$.
\end{lemma}
\begin{proof}
%See Appendix.
%Recall that DPDA-D iterate sequence, $\{\tilde{\bz}^k\}_{k\geq 0}$, generated as in Fig.~\ref{alg:PDDual} is the same as the iterate sequence generated by the recursion in \eqref{eq:inexact-rule-dual}, \eqref{eq:pock-pd-3-xi}, and \eqref{eq:pock-pd-3-y}, where $\tilde{\bz}^k=[\xt, {\by^k}^\top]^\top\in\cZ$ and $\tilde{\bx}^k=[{\bxi^k}^\top~{\bv^k}^\top]^\top\in\cX$ for $k\geq 0$. Moreover,
{Given $\{\bv^k\}_{k\geq 0}$ generated as in Fig.~\ref{alg:PDDual}, let $\{\bw^k\}_{k\geq 0}$ sequence be defined according to \eqref{eq:pock-pd-3-v} -- %it is worth emphasizing
recall that {$\{\bw^k\}_{k\geq 0}$} sequence is never actually computed in practice; this sequence will help us in our analysis of DPDA-D.
%For
%%the sake of
%notational simplicity, we
%%also
%define $\bx^k=[{\bxi^k}^\top~{\bw^k}^\top]^\top\in\cX$ and ${\bz}^k=[{\bx^k}^\top {\by^k}^\top]^\top\in\cZ$ for $k\geq 0$.
}

Let $\Phi$, $h$, and \nv{possibly \emph{nonlinear} map $T(\cdot)$} be as given in Definition~\ref{def:saddle-definitions}; hence, our objective is to compute a saddle-point for $\min_{\bx\in\cX}\max_{\by\in\cY}\Phi(\bx)+\nv{\fprod{T(\bx),\by}}-h(\by)$ %in order
to solve \eqref{eq:central_problem}. Using this notation, {and the fact that $\bv^k=\bw^k+\gamma\be^k$} for $k\geq 0$, we can represent $\{\bxi^k\}$, $\{\bw^k\}$ and $\{\by^k\}$ sequences in a more compact form as follows:
{\small
\begin{subequations}
\label{eq:dpda}
\begin{eqnarray}
&&\qquad \ \ \bx^{k+1}=\argmin_{\bx\in\cX}~{\rho(\bx)+f({\bx}^k)+\fprod{\grad f({\bx}^k)+\nv{\bJ T({\bx}^k)^\top \by^k+U\be^k},~\bx-\nv{\bx^k}}+D_x(\bx,\bx^k)},\label{eq:DPDA-x}\\
&&\qquad \ \ \by^{k+1}=\argmin_{\by\in\cY}~h(\by)-\nv{\fprod{2T({\bx}^{k+1})-T({\bx}^k)-\gamma (2\be^{k+1}-\be^k),\by}}+D_y(\by,\by^k),
\label{eq:DPDA-y}
\end{eqnarray}
\end{subequations}
}%
\nv{where $U=[\mathbf{0}~\id_{n_0}]^\top\in\reals^{(n+n_0)\times n_0}$} and $\{\bv^k\}$ is updated according to \eqref{eq:inexact-rule-dual}.
%\sa{**********************}\\
\nv{Let $S^k_1(\bw)\triangleq \fprod{\be^k,~\bw-\bw^{k+1}}$.} Since $\rho$ is a proper, closed, convex function and $D_x$ is a Bregman function, Property~1 in~\cite{Tseng08_1J} applied to \eqref{eq:DPDA-x} implies that for any ${\bx}\in\cX$, %we have
\vspace*{-1mm}%
{\small
\begin{multline}
\rho({\bx})-\rho(\bx^{k+1})+\fprod{\grad f({\bx}^k)+\nv{\bJ T({\bx}^k)^\top\by^k},~{\bx}-\bx^{k+1}} \geq \\
{D_x({\bx},\bx^{k+1})-D_x({\bx},{\bx}^k)+D_x(\bx^{k+1},{\bx}^k)}\nv{-S^k_1(\bw)}. \label{eq:triangular-ineq-Bregman}
\end{multline}}%
%Note that $\norm{\by}\leq 2\sqrt{N}~B$ for ${\bf y}\in \Ct$; hence, for any $\bv$ and $\bw$, we have
%$\sigma_{\Ct}({\bv})=\sup_{{\bf y}\in \Ct}\langle \bw,{\bf y} \rangle+\langle \bv-\bw,{\bf y} \rangle \leq \sigma_{\Ct}(\bw)+{2\sqrt{N}~B}~\|\bv-\bw\|.$
%Moreover, from the definition of $\{\be^k\}_{k\geq 1}$ sequence in \eqref{eq:prox-error-seq}, we have ${\bv}^{k}={\bw}^{k}+\gamma \be^k$ for all $k\geq 1$; hence, it follows that for any $\bv\in\reals^{n_0}$, we also have
%\vspace*{-1mm}%
%{\small
%\begin{multline}
%\sigma_{\Ct}({\bw^{k+1}})-\langle \by^k,~{\bw^{k+1}}\rangle+\tfrac{1}{2\gamma}\|{\bv}-{\bw}^{k+1}\|^2+\tfrac{1}{2\gamma}\|{\bw}^{k+1}-\bv^k\|^2 \geq\\
%\sigma_{\Ct}({\bv}^{k+1})-\langle \by^k,~{\bv}^{k+1}\rangle+\tfrac{1}{2\gamma}\|{\bv}-{\bv}^{k+1}\|^2+\tfrac{1}{2\gamma}\|{\bv}^{k+1}-{\bv}^k\|^2- S^k(\bv),\label{eq:support-function-bound}
%\end{multline}
%}%
%where the error term $S^k$ is defined as
%{\small
%\begin{equation}
%\label{eq:Sk}
%S^k(\bv)\triangleq {2\gamma \sqrt{N}~B }~\|\be^{k+1}\|-\gamma\|\be^{k+1}\|^2-\fprod{\be^{k+1},~\bv-2\bv^{k+1}+\bv^k+\gamma\by^k}.
%\end{equation}
%}%
%Therefore, combining \eqref{eq:triangular-ineq-Bregman} and \eqref{eq:support-function-bound} implies that for any ${\bx}\in\cX$ we have
%\vspace*{-1mm}%
%{\small
%\begin{multline}
%\rho({\bx})-\rho(\tilde{\bx}^{k+1})+\fprod{\grad f(\tilde{\bx}^k)+\ey{\grad T(\tilde\bx^k)^\top\by^k},~{\bx}-\tilde{\bx}^{k+1}}\geq \\
%D_x({\bx},\tilde{\bx}^{k+1})-D_x({\bx},\tilde{\bx}^k)+D_x(\tilde{\bx}^{k+1},\tilde{\bx}^k)-S^k(\bv).
%\label{eq:triangular-ineq-Bregman-modified}
%\end{multline}}%
Moreover, convexity of $f_i$ and Lipschitz continuity of $\grad f_i$ implies that for any $\xi_i\in\reals^{n_i}$,
{\small
\begin{align*}
f_i(\xi_i)&\geq f_i(\xi_i^k)+\fprod{\grad f_i(\xi_i^k),\xi_i-\xi_i^k} \geq f_i(\xi_i^{k+1})+\fprod{\grad f_i(\xi_i^{k}),\xi_i-\xi_i^{k+1}}-\frac{{L_{f_i}}}{2}\|{\xi_i^{k+1}-\xi_i^k}\|^2.
\end{align*}}%
\nv{Similarly, since $-y_i^k\in\cK^*$, $\cK$-convexity of $g_i$ and Lipschitz continuity of $\bJ g_i$ imply
\begin{align*}
-\fprod{g_i(\xi_i),~y_i^k}&\geq -\fprod{g_i(\xi_i^k),~y_i^k}-\fprod{\bJ g_i(\xi_i^k)^\top y_i^k,~\xi_i-\xi_i^k} \\
&\geq -\fprod{g_i(\xi_i^{k+1}),~y_i^k}-\fprod{\bJ g_i(\xi_i^k)^\top y_i^k,~\xi_i-\xi_i^{k+1}}-\frac{\beta_k L_{g_i}}{2}\norm{\xi_i^{k+1}-\xi_i^k}^2.
\end{align*}}%
Summing \nv{the last two inequalities first for each $i$, then summing over $i\in\cN$,} and combining the sum with \eqref{eq:triangular-ineq-Bregman}, we get
\vspace*{-1mm}%
{\small
\begin{multline}
\Phi({\bx})-\Phi(\nv{\bx^{k+1}})+\nv{\fprod{T({\bx})-T({\bx}^{k+1}),~\by^k}}\geq \\
D_x({\bx},{\bx}^{k+1})-D_x({\bx},{\bx}^{k})+\nv{\tfrac{1}{2}\norm{{\bx}^{k+1}-{\bx}^k}_{\mathbf{\tilde{D}}^k}^2-S_1^k(\bw)},
\label{eq:lemma-x}
\end{multline}}%
\nv{where %we define the primal error term as
$\mathbf{\tilde D}^k\triangleq
\begin{bmatrix}
\mathbf{\tilde D}^k_\tau & \zero\\
\zero & \bD_\gamma
\end{bmatrix}$ and $\mathbf{\tilde D}^k_\tau\triangleq \diag([(\frac{1}{\tau_i}-(L_{f_i}+\beta_kL_{g_i}))\id_{n_i}]_{i\in\cN})$}. %using the definition of $\mathbf{D}$ -- see~Definition~\ref{def:Q}.

Finally, since $h$ is a proper, closed, convex function and $D_y$ is a Bregman function, Property~1 in~\cite{Tseng08_1J} applied to \eqref{eq:DPDA-y} implies
\vspace*{-2mm}%
{\small
\begin{multline}
h(\by)-h(\by^{k+1})-\fprod{\nv{2T({\bx}^{k+1})-T({\bx}^k)},~\by-\by^{k+1}}\geq \\
D_y(\by,\by^{k+1})-D_y(\by,\by^{k})+\tfrac{1}{2}\norm{\by^{k+1}-\by^k}_{\mathbf{D}_\kappa}^2\nv{-S_2^k(\by)}, \label{eq:lemma-y}
\end{multline}}%
where %the dual error is defined as
\nv{$S_2^k(\by)=\gamma \fprod{(2\be^{k+1}-\be^k),~\by-\by^{k+1}}$}.
%Let \nv{$S^k(\bz)\triangleq S_1^k(\bw)+S_2^k(\by)$.}
Summing \eqref{eq:lemma-x} and \eqref{eq:lemma-y}, and rearranging the terms yields
\vspace*{-1mm}%
\nv{\footnotesize
\begin{align*}
%\label{eq:inexact-lagrangian-bound}
\mathcal{L}({\bx}^{k+1},\by)-\mathcal{L}({\bx},\by^{k+1})\leq &
S^k(\bz)+\Big[D_x({\bx},{\bx}^k)+D_y(\by,\by^k)-\fprod{T({\bx})-T({\bx}^k),~\by-\by^k}\Big] \\
&-\Big[D_x({\bx},{\bx}^{k+1})+D_y(\by,\by^{k+1})-\fprod{T({\bx})-T({\bx}^{k+1}),~\by-\by^{k+1}}\Big],
\end{align*}
\begin{equation*}
S^k(\bz)\triangleq S_1^k(\bw)+S_2^k(\by)+\fprod{T(\bx^{k+1})-T(\bx^k),~\by^{k+1}-\by^k}-\frac{1}{2}\norm{\bx^{k+1}-\bx^k}^2_{\mathbf{\tilde D}^k}-\frac{1}{2}\norm{\by^{k+1}-\by^k}_{\bD_\kappa}^2.
\end{equation*}}%
%The last three terms can be further bounded using Cauchy Schwartz inequality and Lipschitz continuity of $g_i(\cdot)$ for any $i\in\cN$, which leads to
\nv{Using Cauchy Schwartz inequality and Lipschitz continuity of $g_i$ for all $i\in\cN$, one can bound $S^k(\bz)$ as follows:}
%{\small
%\begin{align}
%\label{eq:inexact-lagrangian-bound}
%\mathcal{L}({\bx}^{k+1},\by)-\mathcal{L}({\bx},\by^{k+1})\leq \nv{S^k(\bz)}+ \Big[D_x({\bx},{\bx}^k)+D_y(\by,\by^k)-\fprod{T({\bx})-T({\bx}^k),~\by-\by^k}\Big] \nonumber \\
%-\Big[D_x({\bx},{\bx}^{k+1})+D_y(\by,\by^{k+1})-\fprod{T({\bx})-T({\bx}^{k+1}),~\by-\by^{k+1}}\Big]\nv{-\tfrac{1}{2}\norm{\bu({\bz}^{k+1},{\bz}^k)}_{\mathbf{\bar{Q}}}^2},
%\end{align}}%
{\small
\begin{align*}
S^k(\bz)\leq \|\be^k\|\|\bw-\bw^{k+1}\|+\gamma\|2\be^{k+1}-\be^k\|\|\by-\by^{k+1}\|-\tfrac{1}{2}\bu({\bz}^{k+1},{\bz}^k)^\top\mathbf{\bar{Q}}(\beta_k)~\bu({\bz}^{k+1},{\bz}^k)
\end{align*}}%
for all $\bx\in\cX$, $\by\in\cY$, and $k\geq 0$. % where $\mathbf{\bar{Q}}$ is defined as in Definition~\ref{def:Q}.
%\nv{The desired result in~\eqref{thmeq:inexact-lagrangian-bound} is an immediate consequence of %using Caushy Shwartz inequality
%$S^k(\bz)\leq \norm{\be^k}\norm{\bw-\bw^k}+\gamma\norm{2\be^{k+1}-\be^k}\norm{\by-\by^{k+1}}=E^{k+1}(\bz)$.}
\end{proof}
\nv{Given some $\beta>0$, next lemma %will show that if
gives a sufficient condition on the local step-sizes
%condition in \eqref{eq:step-rule-dual} holds (possibly with equality for some $i\in\cN$), then %\ey{$-\tfrac{1}{2}\norm{\bu({\bz}^{k+1},{\bz}^k)}_{\mathbf{\bar{Q}}}^2$}
for \nv{$\mathbf{\bar{Q}}(\beta)$} to be positive (semi)-definite.} %; hence, the last term in \eqref{thmeq:inexact-lagrangian-bound} can be safely dropped}
%which helps to simplify the analysis of Theorem~\ref{col:dual-error-bound}.
\begin{lemma}
\label{lem:schur}
%Given positive $\{\tau_i,\kappa_i\}_{i\in\mathcal{N}}$ and $\gamma$ such that $\gamma>0$, and $\tau_i,\kappa_i>0$ for $i\in\cN$,
\nv{Consider $\mathbf{\bar{Q}}(\beta)$ given in Definition~\ref{def:Q} for some $\beta>0$. If positive $\{\tau_i,\kappa_i\}_{i\in\mathcal{N}}$ and $\gamma$ %are chosen such that
satisfy {$\tau_i\leq\frac{1}{\max\{1,L_{f_i}+\beta L_{g_i}\}}$}, $\kappa_i\leq \frac{1}{\gamma}$ and $({1\over \tau_i}-\max\{1,L_{f_i}+\beta L_{g_i}\})( {1\over \kappa_i}-\gamma) > C_{g_i}^2$ %\eqref{eq:step-rule-dual}
%holds
for all $i\in\cN$, then $\mathbf{\bar{Q}}(\beta)\succ \mathbf{0}$. Moreover, $\mathbf{\bar{Q}}(\beta)\succeq \mathbf{0}$ if the strict inequalities in the last condition %condition
%in \eqref{eq:step-rule-dual}
are relaxed to $\geq$-relation for some $i\in\cN$.}
\end{lemma}
\begin{proof}
Given a permutation matrix ${\small \bP\triangleq \begin{bmatrix} {\bI_{N}} & \mathbf{0} & \mathbf{0}\\ \mathbf{0} & \mathbf{0} & \bI_{N}\\ \mathbf{0} & {\bI_{N}} & \mathbf{0}\end{bmatrix}}$, ${\bf \bar{Q}}(\beta)\succ \mathbf{0}$ is equivalent to $\bP\mathbf{\bar{Q}}(\beta)\bP^{-1}\succ \mathbf{0}$. Since \nv{$\gamma>0$}, Schur complement condition implies %that
%${\small }$ if and only if
\nv{\small
\begin{align}\label{eq:schur-cond-2}
\bP\mathbf{\bar{Q}}(\beta)\bP^{-1}=
\begin{bmatrix}
\mathbf{\bar{D}}_\tau(\beta) & -\bC& \mathbf{0}\\
-\bC & \mathbf{\bar{D}}_\kappa & -\id_N\\
\mathbf{0} & -\id_N & \tfrac{1}{\gamma}\id_N
\end{bmatrix}\succ \mathbf{0}\
\Leftrightarrow\
%\hbox{ if and only if }
\begin{bmatrix}
\mathbf{\bar{D}}_\tau(\beta) & -\bC\\ -\bC & \mathbf{\bar{D}}_\kappa
\end{bmatrix}-\gamma
\begin{bmatrix} \mathbf{0} & \mathbf{0} \\ \mathbf{0} & \id_N \end{bmatrix} \succ \mathbf{0}.
\end{align}}%
%Moreover, since \eqref{eq:step-rule-dual} holds,
Note $\mathbf{\bar{D}}_\tau(\beta)\succ 0$; hence, using Schur complement again, one can conclude that the condition on the right-hand-side of \eqref{eq:schur-cond-2} holds if and only if \nv{$\mathbf{\bar{D}}_\kappa-\gamma\id_N-\bC \mathbf{\bar{D}}_\tau(\beta)^{-1}\bC\succ \mathbf{0}$, equivalently $(\frac{1}{\kappa_i} - \gamma) - (\frac{1}{\tau_i}-\max\{1,L_{f_i}+\beta L_{g_i}\})^{-1} C_{g_i}^2 > 0$} for all $i\in\cN$. Hence, %\eqref{eq:step-rule-dual}
the conditions in Lemma~\ref{lem:schur} are both necessary and sufficient for ${\bf \bar{Q}}(\beta)\succ \mathbf{0}$. %Using the same argument
\sa{If the strict inequalities in the last condition %in \eqref{eq:step-rule-dual}
are relaxed to include equality for some $i\in\cN$, then %the resulting condition
it is sufficient for $\mathbf{\bar{Q}}(\beta)\succeq \mathbf{0}$.}
\end{proof}
\nv{Note if {$\{\by^k\}\subseteq\cB$}, then we can set $\beta^k=2B$ for all $k\geq 0$; hence, Lemma~\ref{lem:schur} implies that if the local step-size condition in \eqref{eq:step-rule-dual} holds (possibly with equality for some $i\in\cN$), then %\ey{$-\tfrac{1}{2}\norm{\bu({\bz}^{k+1},{\bz}^k)}_{\mathbf{\bar{Q}}}^2$}
\nv{$\mathbf{\bar{Q}}(\beta^k)$} in~\eqref{thmeq:inexact-lagrangian-bound} is positive (semi)-definite for all $k\geq 0$,%; hence, the last term in \eqref{thmeq:inexact-lagrangian-bound} can be safely dropped}
which helps to simplify the analysis of Theorem~\ref{col:dual-error-bound}.}
\subsection{Proof of Theorem~\ref{col:dual-error-bound}}
Using the two technical lemmas in the Appendix \ref{append}, we are ready to prove Theorem~\ref{col:dual-error-bound}.
\nv{The proof is divided into three subsections where we first show that the dual iterate sequence $\{\by^k\}$ {stays} bounded even if a dual bound is not provided, i.e., $B=\infty$; second, we prove the convergence of the iterate sequences; finally, we provide rate statements for the infeasibility and suboptimality.}

Under Assumption~\ref{assump:existence}, a saddle point $(\bxi^*,\bw^*,\by^*)$ for $\min_{\bxi,\bw}\max_{\by}\cL(\bxi,\bw,\by)$ exists, where $\cL$ is given in~\eqref{eq:lagrangian-dual-implementation}; moreover, any saddle point $(\bxi^*,\bw^*,\by^*)$ satisfies that $\by^*=\one \otimes y^*$ for some $y^*\in\cB_0$ such that $(\bxi^*,y^*)$ is a primal-dual solution to \eqref{eq:central_problem}. Thus, {$y^*\in\cK^\circ$} and $\mathcal{L}(\bxi^*,\bw^*,\by^*)=\varphi(\bxi^*)$. Indeed, this implies $\fprod{\by^*,\bw^*}-\sigma_{\Ct}(\bw^*)=0$ which leads to {$\sum_{i\in\cN}w_i^*=\mathbf{0}$}, i.e., $\bw^*\in\cC^\circ$. Hence, we have $0=\fprod{\by^*,\bw^*}=\sigma_{\Ct}(\bw^*)$, and it trivially follows that if $(\bxi^*,\bw^*,\by^*)$ is a saddle point of $\cL$ with $\bw^*\neq\mathbf{0}$, then $(\bxi^*,\mathbf{0},\by^*)$ is another saddle point of $\mathcal{L}$. Therefore, under Assumption~\ref{assump:existence}, there is always a saddle point of the form $(\bxi^*,\mathbf{0},\by^*)$, i.e., with $\bw^*=\mathbf{0}$. In the rest, let $\bz^*$ be a saddle point with components $(\bxi^*,\mathbf{0},\by^*)$. %\todo{Check this statement for $B=\infty$}

%Assume that the {step-size} condition in \eqref{eq:step-rule-dual} holds for all $i\in\cN$ with strict inequality, i.e.,
Next, we state few useful observations %that we repeatedly use
later used in the proof. %which we will refer to in different parts of the proof below.
\nv{Given some $\beta>0$,} when primal-dual stepsizes are chosen as stated in \eqref{eq:step-rule-dual}, %in Theorem~\ref{col:dual-error-bound},
\nv{Lemma~\ref{lem:schur} implies that $\mathbf{\bar{Q}}(\beta)\succ 0$ and it follows from definitions of $D_x,D_y$ and $T$ that for all $\bz,\bz'\in\cZ$,}
%Lemma~\ref{lem:schur} implies that $\mathbf{\bar{Q}}\succ 0$; and it follows from the definition of $\bf\bar{Q}$ that %we get that for all $\bz,\bz'$, we have
\vspace*{-2mm}%
\nv{\small
%\begin{dmath}\label{eq:min-eignvalue}
\begin{multline}\label{eq:min-eignvalue}
D_x(\bx,\bx')+D_y(\by,\by')-\fprod{T(\bx)-T(\bx'),~\by-\by'} \\
%&\geq \tfrac{1}{2}\norm{{\bu(\bz,\bz')}}_{\mathbf{\bar{Q}}}^2  \nonumber \\
\geq \nv{\sum_{i\in\cN}\tfrac{1}{2}\max\{1,L_{f_i}+\beta L_{g_i}\}\norm{\xi_i-\xi_i'}^2}+\frac{1}{4\gamma}\norm{\bw-\bw'}^2+\frac{\gamma}{4}\norm{\by-\by'}^2.
\end{multline}}%
%\end{dmath}}%
%\vspace*{-2mm}
%and some \ey{$\lambda_{\min}\in(0,\gamma)$. In fact, by fixing $\gamma$, one can choose $\lambda_{\min}\in(0,\gamma)$ freely and by appropriately choosing the stepsizes $[\tau_i]_{i\in\cN}$ and $[\kappa_i]_{i\in\cN}$. For instance, set $\lambda_{\min}=\frac{\gamma}{2}$, $\tau_i=(L_{f_i}+2BL_{g_i}+C_{g_i}+\nv{1})^{-1}$, $\kappa_i=(C_{g_i}+\frac{5\gamma}{2})^{-1}$ for any $i\in\cN$ will satisfies the inequalities in \eqref{eq:min-eignvalue}.}
\noindent\nv{Moreover, the error term ${E^{k+1}}(\bz)$, %appearing in \eqref{eq:dynamic-saddle-rate} is
defined in Lemma~\ref{thm:dynamic-rate},
%and it is the error term due to approximating $\cP_{\Ct}$ in the $k$-th iteration of the algorithm.
trivially satisfies
{\small
\begin{align}
\label{eq:error-term-bound}
E^{k+1}(\bz)\leq \gamma (2\|\be^{k+1}\|+\|\be^k\|)( \tfrac{1}{\gamma}\|\bw^{k+1}-\bw\|+\|\by^{k+1}-\by\|),\quad\forall\ k\geq 0.
\end{align}}}%
 \vspace*{-5mm}
\subsubsection{Boundedness of dual iterate sequence}\label{sec:bound-dual} %In this section,
Next we show %that
$\{\by^k\}_{k\geq 0}$ and $\{\bw^k\}_{k\geq 0}$ are bounded. %sequences. %and $\{\be^k\}_{k\geq 0}$ satisfies $\norm{\be^{k}}=\cO(\alpha^{q_{k-1}}k)$.
More specifically, our aim is to show that there exist $\bar{\beta},\varsigma,\nu\in\reals_+$ such that if we choose the step-sizes as in \eqref{eq:step-rule-dual} for %some
any $\gamma>0$ and $\beta\geq \bar{\beta}$, then
%, if  %on the total number of iterations $K\geq 1$. To this end,
\vspace*{-1mm}
{\small
\begin{align}
\label{eq:induction}
\max_{i\in\cN}\{\|y_i^k\|\}\leq \beta,\quad \|\bw^k\|\leq \varsigma,\quad %there exists $\nu_0>0$ such that
\|\be^{k}\|\leq \nu \alpha^{q_{k-1}}k,
\end{align}}%
for all $k\geq 0$, where \sa{$q_{-1}\triangleq 0$ and $q_0\triangleq 0$}. %\todo{$q_0=0$}.
Below we provide the analysis for two sperate cases. \nv{We first define two quantities that are repeatedly used in the proof. Define $C_0\triangleq \sum_{k=1}^\infty\alpha^{q_{k-1}}k <+\infty$ --{note $C_0>1$.} Let $A_0\triangleq D_x({\bx^*},{\bx}^0)+D_y(\by^*,\by^0)-\fprod{{T({\bx^*})-T({\bx}^0)},~\by^*-\by^0}$. Since we initialize $\bw^0=\by^0=\mathbf{0}$, the proof of Lemma~\ref{lem:schur} implies that $A_0\leq \norm{\bxi^*-\bxi^0}_{\bD_\tau}^2+\norm{\by^*}_{\bD_\kappa}^2$. %Let $\bL'\triangleq\diag([(1+L_{f_i}+C_{g_i})\id_{n_i}]_{i\in\cN})$ and $\bL_g\triangleq\diag([(L_{g_i})\id_{n_i}]_{i\in\cN})$, then using %step-size condition
%$\tau_i\leq \frac{1}{L_{f_i}+\beta L_{g_i}}$ for any $i\in\cN$,
%in
\sa{Recall the definitions of $\bar{C}_g$ and $\bar{R}_x$ %$\bL'$ $\bL_g$
given in Section~\ref{sec:convergence}.}
Using \eqref{eq:step-rule-dual}, we get
\nv{\small
\begin{align}
\label{eq:A_bound}
A_0\leq \norm{\bxi^*-\bxi^0}_{\bD_\tau}^2+\norm{\by^*}_{\bD_\kappa}^2\leq (\beta+1)N\bar{R}_x^2+(\bar{C}_g+\tfrac{5}{2}\gamma)N\norm{y^*}^2\triangleq\bar{A}_0,
\end{align}}%
%where $\bar{C}_g=\sum_{i\in\cN}C_{g_i}/N$ and $\bar{R}_x\triangleq \max\{\norm{\bxi^*-\bxi^0}_{\bL_g},\norm{\bxi^*-\bxi^0}_{\bL'}\}/\sqrt{N}$.
} {In the rest we assume $\bar{C}_g\geq 1$.}
\paragraph{CASE 1: Bound $B$ on $\norm{y^*}$ is available, i.e., $B\in(0,\infty)$}
\nv{In this part, we assume that a nontrivial dual bound $B\in(0,\infty)$ is available.}
%; later in Section~\ref{col:dual-error-bound}, when proving Theorem~\ref{col:dual-error-bound}, we discuss how we can relax this assumption.}
Suppose we set $\bar{\beta}=2B$ and we choose the step-sizes as in \eqref{eq:step-rule-dual} for some $\gamma>0$ and $\beta\geq \bar{\beta}$. Trivially, from \eqref{eq:pock-pd-3-y}, we have $\max_{i\in\cN}\norm{y_i^k}\leq 2B\leq \beta$ for $k\geq 0$. Hence, Lemma~\ref{thm:dynamic-rate} shows that for all $k\geq 0$, \eqref{thmeq:inexact-lagrangian-bound} holds for $\beta_k=\beta$.
%and we conclude that $\mathbf{\bar Q}^k\succeq \mathbf{\bar Q}$ for all $k\in\cI$.
Moreover, stepsize condition in~\eqref{eq:step-rule-dual} and Lemma~\ref{lem:schur} imply that $\bar{\bQ}(\beta)\succ\mathbf{0}$. Therefore, for any $\ell\geq 0$, dropping the last term in \eqref{thmeq:inexact-lagrangian-bound}, summing over $k\in\{0,\ldots,\ell\}$, and using Jensen's inequality, we get for all $\bz\in\cZ$, %gives %and finally dropping the last term, $D_x({\bx},\tilde{\bx}^{K})+D_y(\by,\by^{K})-\fprod{T({\bx})-T(\tilde{\bx}^{K}),~\by-\by^{K}}\geq 0$, in the telescoping sum gives
\vspace*{-1mm}%
{\small
\begin{eqnarray}\label{eq:dynamic-saddle-rate}
\lefteqn{(\ell+1)(\mathcal{L}(\bar{\bf x}^{\ell+1},\by)-\mathcal{L}({\bx},\bar{\by}^{\ell+1})) \leq} \\
&&\Big[ D_x({\bx},{\bx}^0)+D_y(\by,\by^0)-\fprod{\nv{T({\bx})-T({\bx}^0)},~\by-\by^0}\Big] \nonumber \\
&&-\Big[ D_x({\bx},{\bx}^{\ell+1})+D_y(\by,\by^{\ell+1})-\fprod{\nv{T({\bx})-T({\bx}^{\ell+1})},~\by-\by^{\ell+1}}\Big]+{\sum_{k=0}^{\ell} E^{k+1}(\bz)},
%\ \forall\ \bz\in\cZ,
\nonumber
\end{eqnarray}
}%
where $\bar{\bx}^{\ell+1}\triangleq {1\over (\ell+1)}\sum_{k=1}^{\ell+1}{\bx}^k$ and $\bar{\by}^{\ell+1}\triangleq {1\over (\ell+1)}\sum_{k=1}^{\ell+1}\by^k$. For any $\ell\geq 0$, setting $\bz=\bz^*$ in \eqref{eq:dynamic-saddle-rate}, using $\cL(\bar{\bx}^{\ell+1},\by^*)-\cL(\bx^*,\bar{\by}^{\ell+1})\geq 0$, and \eqref{eq:min-eignvalue} we obtain,
\nv{\small
\begin{eqnarray}\label{eq:iterate-bound-general}
\tfrac{1}{4\gamma}\|\bw^{\ell+1}\|^2+\tfrac{\gamma}{4}\|\by^*-\by^{\ell+1}\|^2\leq A_0+\sum_{k=0}^{\ell} E^{k+1}(\bz^*).
%\quad \forall\ \ell\in\cI,
\end{eqnarray}}%
\nv{Hence, using~\eqref{eq:error-term-bound}, \eqref{eq:iterate-bound-general} and the fact that $\bw^*=0$, for all $\ell\geq 0$, we have
{\small
\begin{align}\label{eq:iterate-quad}
\tfrac{\gamma}{8}(\tfrac{1}{\gamma}\|\bw^{\ell+1}\|+\|\by^*-\by^{\ell+1}\|)^2
& \leq \tfrac{1}{4\gamma}\|\bw^{\ell+1}\|^2+\tfrac{\gamma}{4}\|\by^*-\by^{\ell+1}\|^2 \nonumber \\
& \leq A_0 +\sum_{k=1}^{\ell+1}\gamma(2\|\be^{k}\|+\|\be^{k-1}\|)(\tfrac{1}{\gamma}\|\bw^k\|+\|\by^*-\by^k\|).
%\quad \forall \ell\in\cI. %+\gamma(2\norm{\be^{K}}+\|\be^{K-1}\|)(\tfrac{1}{\gamma}\|\bw^K\|+\|\by^*-\by^K\|)
\end{align}}%
}
Next, we use Lemma~\ref{lem:technical-lemma} with $u_k=\tfrac{1}{\gamma}\|\bw^k\|+\|\by^*-\by^k\|$, $S_k=\tfrac{8}{\gamma}A_0$ for $k\geq 0$, and $\lambda_k=8(2\norm{\be^{k}}+\norm{\be^{k-1}})$ for $k\geq 1$. Note \eqref{eq:min-eignvalue} and $\bw^0=\by^0=\mathbf{0}$ imply that $A_0\geq \frac{\gamma}{4}\norm{\by^*}^2$; hence, we have $u_0^2\leq S_0$. Thus, Lemma~\ref{lem:technical-lemma}
%Let $\beps^{k}\triangleq 2\norm{\be^{k}}+\norm{\be^{k-1}}$ for $k=1,\ldots,K$, and %Invoking Lemma~\ref{lem:technical-lemma} by setting
implies that for all $\ell\geq 0$,
%and using the induction assumption we obtain the following,
{\small
\begin{align}\label{eq:iterate-bound}
\tfrac{1}{\gamma}\|\bw^{\ell+1}\|+\|\by^*-\by^{\ell+1}\|\leq&\tfrac{1}{2}\sum_{k=1}^{\ell+1}\lambda_k +\sqrt{\frac{8A_0}{\gamma}+\left(\tfrac{1}{2}\sum_{k=1}^{\ell+1}\lambda_k\right)^2}\leq 24\sum_{k=1}^{\ell+1}\norm{\be^{k}} +\sqrt{\frac{8A_0}{\gamma}}.
\end{align}}%

%The following observation will also be useful to prove error bounds for DPDA-D iterate sequence.
For each $i\in\cN$ and $k\geq 0$, the definition of $\cR^k$ in \eqref{eq:approx_error-for-full-vector} implies $\cR_i^k(\by)\in\cB_0$ for all $\by$; hence, from \eqref{eq:inexact-rule-dual},
$\|v_i^{k+1}\| \leq \|v_i^k+\gamma y_i^k\|+\gamma\|\cR_i^k\big(\tfrac{1}{\gamma}{\bv}^k
+\by^k\big)\| \leq \|v_i^k\|+{4\gamma B}$; thus, $\max_{i\in\cN}\norm{v_i^k}\leq 4\gamma B k$ for $k\geq 0$, and
%\ey{Similarly, $\|{\bw^{k+1}}\|\leq \|{\sa{\bv^k}+\gamma \by^k}\|+\gamma\|{\cP_{\Ct}(\frac{1}{\gamma}\sa{\bv^k}+\by^k)}\|\leq \|{\bv^k}\|+4\gamma \sqrt{N}~B$, and
%We can repeat the same argument for $\norm{.}_1$.
%Thus,
we trivially get the following bound:
%on $\norm{\bv^k}$ and $\norm{\bw^k}$ for $k\geq 0$:
{\small
\begin{equation}\label{eq:mu-bound}
\|\bv^k\| \leq {4\gamma \sqrt{N}~B~k},\quad \forall\ k\geq 0. %\norm{\bw^k} \leq {4\gamma \sqrt{N}~B~k}.   %,\qquad \|\mu^k\|_1 \leq \|\mu^0\|_1+4\gamma n |\cN| B~k.
\end{equation}
}%
Hence, for $k\geq 0$, since
%\eqref{eq:approx_error} implies that $\|\mathcal{R}^k(\by)-\mathcal{P}_{\sa{\Ct}}(\by)\|\leq N~\Gamma \alpha^{q_k}\norm{\by}$ for all $\by$, and
$\|\by^k\|\leq{2\sqrt{N}~B}$, it follows from \eqref{eq:approx_error-for-full-vector}, \eqref{eq:prox-error-seq} and \eqref{eq:mu-bound} that
{\small
\begin{align}
\label{eq:e-bound}
\|\be^{k+1}\|
%&=\norm{\cP_{\Ct}\left(\tfrac{1}{\gamma}{\bv}^k+\by^k\right)-\cR^k\left(\tfrac{1}{\gamma}{\bv}^k+\by^k\right)} \nonumber \\
%&
\leq N~\Gamma \alpha^{q_k}\|\tfrac{1}{\gamma}{\bv}^k+\by^k\| \leq {2N^{\tfrac{3}{2}}B\Gamma \alpha^{q_k}{(2k+1)}}\ \nv{\implies\ \nu=4N^{\tfrac{3}{2}}B\Gamma.}
\end{align}
}%
Therefore, $\norm{\be^k}$ satisfies \eqref{eq:induction} for $\nu=4N^{\tfrac{3}{2}}B\Gamma$. Using this result within \eqref{eq:iterate-bound}, we obtain \vspace*{-3mm}
{\small
\begin{align}\label{eq:w-bound-final}
\|\bw^{\ell+1}\|\leq 24\gamma\nu\sum_{k=1}^{\ell+1}\alpha^{q_{k-1}}k + \sqrt{8A_0\gamma}\leq \varsigma\triangleq 24\gamma \nu C_0 + \sqrt{8A_0\gamma},\quad \forall\ \ell\geq 0.
\end{align}}%

\paragraph{CASE 2: Bound $B$ on $\norm{y^*}$ is not available, i.e., $B=\infty$} %Suppose we do not have access to a bound $B\in(0,\infty)$; hence,
We set $B=+\infty$ in %the course of the
Algorithm \ref{alg:PDDual}.
%{Lemma~\ref{lem:schur} shows that $\mathbf{\bar{Q}}\succeq 0$ when
We %aim to
prove %the boundedness of $\{\by^k\}$
the claim in \eqref{eq:induction} using induction; indeed, we construct $\bar{\beta},\varsigma,\nu\in\reals_+$ and show for any $\gamma>0$, $\beta\geq \bar{\beta}$ and $K\geq 1$ that if \eqref{eq:induction} holds for all $k\in\cI\triangleq \{0,\hdots,K-1\}$, then $\norm{\be^{K}}\leq \nu\alpha^{q_{K-1}}K$ also holds and this implies $\norm{\bw^K}\leq \varsigma$ and $\max_{i\in\cN}\{\norm{y_i^K}\}\leq \beta$, which would complete the induction.

Since $\bv^0=\bw^0=\mathbf{0}$ and $\by^0=\mathbf{0}$, \eqref{eq:induction} trivially holds for $k=0$. %Now
Suppose \sa{for some $\beta>0$,} \eqref{eq:induction} holds for $k\in\cI$; %The induction assumption implies that $\norm{y_i^k}\leq \beta$ for all $k\in\cI$;
hence, $\max_{i\in\cN}\{\|y_i^k\|\}\leq \beta$ for $k\in\cI$, and using the same arguments as in CASE 1, it can be shown that \eqref{eq:dynamic-saddle-rate}, \eqref{eq:iterate-bound-general}, \eqref{eq:iterate-quad} and \eqref{eq:iterate-bound} hold for all $\ell \in\cI$. Next, using
%with initial points $\by^0=\mathbf{0}$, $\bv^0=\bw^0=\mathbf{0}$ and arbitrary $\bxi^0$. %, which will be specified later. %Under this condition
%Suppose \eqref{eq:induction} holds for all $k\in\cI$; %with $\nu_1=0$;
%hence,
\eqref{eq:prox-error-seq}, \eqref{eq:iterate-bound} implies that %$\norm{\be^K}$ satisfies%will prove the induction statement for $\norm{\be^K}$ as follows:
{\small
\begin{align}\label{eq:error-bound-recursive}
\|\be^{K}\|=&\norm{\cP_{\Ct}\left(\tfrac{1}{\gamma}{\bw}^{K-1}+\be^{K-1}+\by^{K-1}\right)-\cR^{K-1}\left(\tfrac{1}{\gamma}{\bw}^{K-1}+\be^{K-1}+\by^{K-1}\right)} \nonumber \\
%\leq &N~\Gamma \alpha^{q_{K-1}}\|\tfrac{1}{\gamma}{\bw}^{K-1}+\be^{K-1}+\by^{K-1}\|\nonumber \\
\leq &N~\Gamma \alpha^{q_{K-1}} \Big(\|\be^{K-1}\|+24\sum_{k=1}^{K-1}\|\be^k\| +\sqrt{\frac{8A_0}{\gamma}}+\|\by^*\|\Big) \nonumber\\
\leq & N~\Gamma \nu \alpha^{q_{K-1}} \Big(\alpha^{q_{K-2}}(K-1)+24\sum_{k=1}^{K-1}\alpha^{q_{k-1}}k + \Big( \sqrt{\frac{8A_0}{\gamma}}+\norm{\by^*}\Big)/\nu\Big).
\end{align}}%
The assumption, $q_k\geq \log_{1/\alpha}(24 N\Gamma(k+1))$ for $k\geq 0$, and $q_{-1}=0$ imply that $\alpha^{q_{k-1}}k\leq \frac{1}{24 N\Gamma}$ for $k\geq 0$. Thus, for $\nu\triangleq \tfrac{24}{23} N\Gamma(\norm{\by^*}+\sqrt{\frac{8A_0}{\gamma}})$, %we conclude that the right hand side in
\eqref{eq:error-bound-recursive} is indeed bounded above by %$\nu_0 \alpha^{q_{K-1}}(\frac{21}{24}+K-1)\leq$
$\nu \alpha^{q_{K-1}}K$ which proves the induction on $\norm{\be^K}$. %-- here, we assume $N\geq 2$ wlog to get simple bounds.
Hence, using this result within \eqref{eq:iterate-bound} for $\ell=K-1$, we obtain
{\small
\begin{align}\label{eq:iterate-bound-final}
\tfrac{1}{\gamma}\|\bw^K\|+\|\by^{K}\|\leq 24\nu\sum_{k=1}^K\alpha^{q_{k-1}}k + \sqrt{\frac{8A_0}{\gamma}}+\norm{\by^*}\leq 24\nu C_0 + \sqrt{\frac{8A_0}{\gamma}}+\norm{\by^*},
\end{align}}%
where $C_0\triangleq \sum_{k=1}^\infty\alpha^{q_{k-1}}k <+\infty$ and is independent of $\beta$. Thus, $\norm{\bw^K}\leq\gamma\beta$ and $\max_{i\in\cN}\{\norm{y_i^K}\}\leq\beta$ %$\tfrac{1}{\gamma}\norm{\bw^K}+\norm{\by^{K}}\leq \beta$
for all $\beta\geq \big(\tfrac{576}{23}N\Gamma C_0 +1\big)\big(\sqrt{\frac{8A_0}{\gamma}}+\sqrt{N}\norm{y^*}\big)$. %Note that using Schur complement condition
%Recall that $A_0\leq \norm{\bxi^*-\bxi^0}_{\bD_\tau}^2+\norm{\by^*}_{\bD_\kappa}^2$. %$\norm{\bxi^*-\bxi^0}_{\bD_\tau}^2+\norm{\by^*-\by^0}_{\bD_\kappa}^2+\norm{\bw^0}_{\bD_\gamma}^2$.
\nv{Hence, using the bound on $A_0$ in
\eqref{eq:A_bound},} we derive %the following
a sufficient condition on $\beta$: %$\max_{i\in\cN}\{\tfrac{1}{\gamma}\norm{w_i^K}+\norm{y_i^K}\}\leq\beta$ holds for all $\beta$ such that
{\small
\begin{align}
&\beta\geq (\tfrac{576}{23}N\Gamma C_0+1)\sqrt{N}\left(\norm{y^*}+\sqrt{\tfrac{8}{\gamma}\Big(
%{\beta}\norm{\bxi^*-\bxi^0}^2_{\bL_g}+\norm{\bxi^*-\bxi^0}^2_{\bL'}+\norm{\by^*}^2_{\bD_\kappa}
(\beta+1)\bar{R}_x^2+(\bar{C}_g+\tfrac{5}{2}\gamma)\norm{y^*}^2\Big)}\right). \label{eq:beta-cond}
\end{align}}%
%Let $c_0 \triangleq \norm{\by^*}+\sqrt{\frac{8}{\gamma}\left(\norm{\bxi^*-\bxi^0}^2_{\bL_f}+\norm{\by^*-\by^0}^2_{\bD_\kappa}+\norm{\bw^0}^2_{\bD_\gamma} \right)}$. A sufficient condition for $\beta$ to satisfy \eqref{eq:beta-cond} is
%{\small
%\begin{align}
%\beta\geq \left(\frac{\sqrt{8}\norm{\bxi^*-\bxi^0}_{\bL_g^{\frac{1}{2}}}(1+24\nu_1N\Gamma)}{\sqrt{N\gamma}}+\sqrt{\frac{c_0(1+24\bnu_1N\Gamma)}{\sqrt{N}}}\right)^2.
%\end{align}}%
Note \eqref{eq:beta-cond} implies that there exists $\bar{\beta}\in\reals$ such that $\bar{\beta}\geq\norm{\by^*}$ and for all $\beta\geq \bar{\beta}$ and $\varsigma=\gamma\beta$, \eqref{eq:induction} holds when the stepsizes are chosen as in \eqref{eq:step-rule-dual} using $\beta$. Thus, when primal step-sizes $[\tau_i]_{i\in\cN}$ chosen sufficiently small and $\{q_k\}$ chosen such that $q_k\geq \log_{1/\alpha}(24 N\Gamma(k+1))$ and $\sum_{k=1}^\infty \alpha^{q_{k-1}}k<\infty$, both $\{\by^k\}$ and $\{\bw^k\}$ %sequences
are \emph{bounded}.
%we have a bounded dual sequence $\{\by^k\}_{k}$, i.e., $\max_{i\in\cN}\norm{y_i^k}\leq\beta$ for all $k\geq 0$.
Moreover, solving the quadratic inequality in~\eqref{eq:beta-cond}, %-- see appendix for details,
we get
\begin{align}
\nv{\beta=\cO\Big(\tfrac{1}{\gamma}{N^3\Gamma^2C_0^2\bar{R}_x^2}+N^{3/2}\Gamma {C_0}\Big(\norm{y^*}+\sqrt{\tfrac{1}{\gamma}(\bar{R}_x^2+\bar{C}_g \norm{y^*}^2)}\Big)\Big).} \label{eq:beta-cond-sol}
\end{align}
%For detailed proof of deriving \eqref{eq:beta-cond-sol} see appendix.}
%It is worth emphasizing a special case:
If $g_i$ is an affine function ($L_{g_i}=0$) for all $i\in\cN$, then \nv{choosing $q_k$ %appropriately to
as before and setting $\tau_i=(\max\{1, L_{f_i}\}+C_{g_i})^{-1}$ for $i\in\cN$ %works
%one only needs to
%choosing
 guarantees that} $\{\by^k\}_{k}$ and $\{\bw^k\}_{k}$ are bounded. \nv{Moreover, since $\bD_\tau$ does not depend on $\beta$, {the term $({\beta}+1)\bar{R}_x^2$ on the rhs of~\eqref{eq:beta-cond} becomes $\bar{R}_x^2$;} thus, $\beta=\cO\Big(N^{3/2}\Gamma {C_0}\Big(\norm{y^*}+\sqrt{\tfrac{1}{\gamma}(\bar{R}_x^2+\bar{C}_g \norm{y^*}^2)}\Big)\Big)$.}
\subsubsection{Convergence of iterates} In previous section, we showed that there exist $\bar{\beta},\varsigma,\nu\in\reals_+$ such that if we choose the step-sizes as in \eqref{eq:step-rule-dual} for any $\gamma>0$ and $\beta\geq \bar{\beta}$, then \eqref{eq:induction} holds for all $k\geq 0$. Consider a saddle point $\bz^*=[{\bx^*}^\top {\by^*}^\top]^\top$ of $\cL$ in~\eqref{eq:lagrangian-dual-implementation}, where $\bx^*=[{\bxi^*}^\top {\bw^*}^\top]^\top$. %We first show that $\sum_{k=0}^{\infty}E^{k+1}(\bz^*)< \infty$.
Trivially, \eqref{eq:error-term-bound} and \eqref{eq:induction} imply that
{\small
\begin{align}
\label{eq:error-term-sum-bound}
\sum_{k=0}^\infty E^{k+1}(\bz^*)\leq 3\gamma \max_{k\geq 0}\{\tfrac{1}{\gamma}\|\bw^{k+1}-\bw^*\|+\|\by^{k+1}-\by^*\|\} \sum_{k=0}^\infty \|\be^{k+1}\|<\infty.
\end{align}}%
%From \eqref{eq:step-rule-dual}, we have $\mathbf{\bar{Q}}\succ 0$;hence,
Evaluating \eqref{thmeq:inexact-lagrangian-bound} at $\bz=\bz^*$, we get
{\small
\begin{eqnarray} \label{eq:general-bound-dynamic}
0 \leq\mathcal{L}(\bx^{k+1},\by^*)-\mathcal{L}(\bx^*,\by^{k+1})\leq \ba^k-\ba^{k+1}-\bb^k+\bc^k, %E^{k+1}(\bz^*)-\ey{\tfrac{1}{2}\norm{\bu({\bz}^{k+1},{\bz}^k)}^2_{\mathbf{\bar{Q}}}} \\
%&  + \left[D_x(\bx^*,{\bx}^k)+D_y(\by^*,\by^k)-\fprod{T(\bx^*)-T({\bx}^k),~\by^*-\by^k}\right]  \nonumber \\
%&  -\left[D_x(\bx^*,{\bx}^{k+1})+D_y(\by^*,\by^{k+1})-\fprod{T(\bx^*)-T({\bx}^{k+1}),~\by^*-\by^{k+1}}\right], \nonumber
  %&-\frac{1}{2}(\bz^{k+1}-\bz^k)^\top\mathbf{\bar{Q}}(\bz^{k+1}-\bz^k);
\end{eqnarray}}%
for $k\geq 0$, where $\ba^k\triangleq D_x(\bx^*,{\bx}^k)+D_y(\by^*,\by^k)-\fprod{T(\bx^*)-T({\bx}^k),~\by^*-\by^k}$, {$\bb^k\triangleq\tfrac{1}{2}\norm{\bu({\bz}^{k+1},{\bz}^k)}_{\mathbf{\bar{Q}}(\beta)}^2$} and $\bc^k\triangleq E^{k+1}(\bz^*)$ for $k\geq 0$. Clearly, $\bb^k\geq 0$ and $\bc^k\geq 0$ for $k\geq 0$. Moreover, from \eqref{eq:min-eignvalue}, we get $\ba^k\geq 0$ %$\tfrac{1}{2}\norm{\bu({\bz}^k,\bz^*)}^2_{\bar{\mathbf{Q}}}\geq 0$
for $k\geq 0$.
\nv{Since %\eqref{eq:E_mu_bound} also implies that $\sum_{k=0}^{K-1}E^{k+1}(\bz^*)\leq \Lambda_2(K)$.
$\sum_{k=0}^{\infty}E^{k+1}(\bz^*)< \infty$,} %which we later show in Section~\ref{sec:rate}, %. Since $\sup_{K\in\integers_+}\Lambda_2(K)<\infty$,
Lemma~\ref{lem:supermartingale} implies that %$\sum_{k=0}^\infty \norm{\bz^{k+1}-\bz^k}_{\mathbf{\bar{Q}}}^2<\infty$ and
$\lim_{k\rightarrow\infty}\ba^k$ exists. Thus, $\{\ba^k\}$ is a bounded sequence; and due to \eqref{eq:min-eignvalue}, $\{{\bz}^k\}$ is bounded as well. Consequently, there exists a subsequence $\{{\bz}^{k_n}\}_n$ such that ${\bz}^{k_n}\rightarrow{\bz^{\#}}$ as $n\rightarrow\infty$. Thus,
%From the fact that $\bz^{k_n}\rightarrow{\bz^\#}$,
there exists $N_1$ such that for all $n\geq N_1$, we have $\norm{{\bz}^{k_n}-{\bz^\#}}<{\epsilon \over 2}$. Moreover, Lemma~\ref{lem:supermartingale} also implies %that
\nv{$\sum_{k=0}^\infty\norm{\bu({\bz}^{k+1},{\bz}^k)}_{\mathbf{\bar{Q}}(\beta)}^2<\infty$}. Since $\mathbf{\bar{Q}}(\beta)\succ \mathbf{0}$, for any $\epsilon>0$, there exists $N_2$ such that for all $n\geq N_2$, we have $\norm{{\bz}^{k_n+1}-{\bz}^{k_n}}<{\epsilon \over 2}$. Therefore, by letting $N=\max\{N_1,N_2\}$ we get $\norm{{\bz}^{k_n+1}-{\bz^\#}}<\epsilon$, i.e., ${\bz}^{k_n+1}\rightarrow{\bz^\#}$ as $n\rightarrow\infty$.

\nv{Note that %\eqref{eq:e-bound}
\eqref{eq:induction} implies $\norm{\be^k}\rightarrow 0$ as $k\rightarrow\infty$} for any $\{q_k\}$ \nv{such that $\sum_{k=1}^\infty\alpha^{q_{k}}k <+\infty$}. %satisfying the condition in the statement of the theorem.
Recall that $\psi_x(\bx)=\frac{1}{2}\norm{\bxi}_{\mathbf{D}_\tau}^2+\frac{1}{2}\norm{\bw}_{\mathbf{D}_\gamma}^2$, and $\psi_y(\by)=\frac{1}{2}\norm{\by}_{\mathbf{D}_\kappa}^2$ are the strongly convex functions corresponding to Bregman distance functions $D_x$ and $D_y$, respectively. In particular, $D_x(\bx,\bar{\bx})=\psi_x(\bx)-\psi_x(\bar{\bx})-\fprod{\grad \psi_x(\bar{\bx}),~\bx-\bar{\bx}}$, and $D_y$ is defined similarly. The optimality conditions for \eqref{eq:dpda} imply that for all $n\in\integers_+$, $\bq^{n}\in\partial\rho(\bx^{k_n+1})$ and $\bp^{n}\in\partial h(\by^{k_n+1})$, where
\nv{$\bq^{n}\triangleq\grad\psi_x({\bx}^{k_n})-\grad\psi_x(\bx^{k_n+1})-\big(\grad f({\bx}^{k_n})+\bJ T({\bx}^{k_n})^\top\by^{k_n}+U\be^{k_n}\big)$, and $\bp^{n}\triangleq\grad\psi_y(\by^{k_n})-\grad\psi_y(\by^{k_n+1})+2T({\bx}^{k_n+1})-T({\bx}^{k_n})+\gamma(2\be^{k_n+1}-\be^{k_n})$.}
Since $\grad\psi_x$ and $\grad\psi_y$ are continuously differentiable on $\dom \rho$ and $\dom h$, respectively, and since $\rho$ and $h$ are proper, closed convex functions, it follows from Theorem~24.4 in~\cite{rockafellar2015convex} that
\nv{$\partial\rho({\bx}^\#)\ni\lim_n \bq^n = -\grad f({\bx}^\#)-\bJ T(\bx^\#)^\top {\by}^\#$, and $\partial h({\by}^\#)\ni\lim_n \bp^n = T({\bx}^\#)$}, which also implies that ${\bz}^\#$ is a saddle point of \eqref{eq:saddle-problem}.

Since \eqref{eq:general-bound-dynamic} is true for any saddle point $\bz^*$, by setting $\bz^*=\bz^\#$ in \eqref{eq:general-bound-dynamic}, one can conclude that $\bs^\#\triangleq \lim_k \bs^k\geq 0$ exists, where
$\bs^k\triangleq D_x(\bx^\#,{\bx}^k)+D_y(\by^\#,\by^k)-\fprod{T(\bx^\#)-T({\bx}^k),~\by^\#-\by^k}$
for $k\geq 0$. {Since %$\bs^\#=\lim_n \bs^{k_n}$ and
$\lim_n \fprod{T(\bx^\#)-T({\bx}^{k_n}),\by^\#-\by^{k_n}}=0$ (from ${\bz}^{k_n}\rightarrow\bz^\#$)}, clearly
$\bs^\#=\lim_{n\rightarrow\infty}~\bs^{k_n}=0$, %$D_x(\bx^\#,{\bx}^{k_n})+D_y(\by^\#,\by^{k_n})=0$,
which together with \eqref{eq:min-eignvalue} implies that ${\bz}^k\rightarrow\bz^\#$.
\begin{comment}
\sa{Recall that DPDA-D iterates have the form $\tilde{\bz}^k=[\xt, {\by^k}^\top]^\top\in\cZ$ and $\tilde{\bx}^k=[{\bxi^k}^\top~{\bv^k}^\top]^\top\in\cX$, while ${\bz}^k=[{\bx^k}^\top, {\by^k}^\top]^\top\in\cZ$ and $\bx^k=[{\bxi^k}^\top~{\bw^k}^\top]^\top\in\cX$ for $k\geq 0$.
%and given $\{\bv^k\}_{k\geq 0}$ generated as in Fig.~\ref{alg:PDDual}, $\{\bw^k\}_{k\geq 0}$ sequence is defined according to \eqref{eq:pock-pd-3-v}. Furthermore, within the proof of Lemma~\ref{thm:dynamic-rate}, we also define $\bx^k=[{\bxi^k}^\top~{\bw^k}^\top]^\top\in\cX$ and ${\bz}^k=[{\bx^k}^\top, {\by^k}^\top]^\top\in\cZ$ for $k\geq 0$.}
From the definition of $\{\be^k\}_{k\geq 1}$ in \eqref{eq:prox-error-seq}, we have ${\bw}^{k}={\bv}^{k}-\gamma \be^k$ for all $k\geq 1$. Therefore, since $\be^k\rightarrow 0$, ${\bx}^{k}\rightarrow{\bx^\#}$ implies that $\tilde{\bx}^{k}\rightarrow{\bx^\#}$}\todo{This part may be deleted}.%also imply that $\tilde{\bx}^{k_n+1}\rightarrow{\bx^\#}$ and ${\bx}^{k_n+1}\rightarrow{\bx^\#}$.
%and $\tilde{\bz}^k\rightarrow\bz^\#$.
\end{comment}
\subsubsection{Convergence rate}
\label{sec:rate}
Recall that we initialize \nv{$\bv^0=\bw^0=\mathbf{0}$ and $\by^0=\mathbf{0}$;} hence, %$\bw^0=\mathbf{0}$ and
the inequality in \eqref{eq:dynamic-saddle-rate} can be written more explicitly as follows: let $\bar{\bxi}^K\triangleq {1\over K}\sum_{k=1}^K \bxi^k$, and $\bar{\bw}^K\triangleq {1\over K}\sum_{k=1}^K\bw^k$, then %\ey{using \eqref{eq:min-eignvalue} we have,}
for any ${\bxi}$, $\bw$ and $\by$, and for all $K\geq 1$, \vspace*{-1mm}
{\small
\begin{align}
\mathcal{L}(\bar{\bxi}^K,\bar{\bw}^K,\by) -\mathcal{L}({\bxi},\bw,\bar{\by}^K)\leq \Theta(\bz)/K,\label{eq:saddle-rate-dynamic}
\end{align}}%
{where $\Theta(\bz)\triangleq  {1\over 2\gamma}\|\bw\|^2{+}\langle {\by},~\bw\rangle+\sum_{i\in\mathcal{N}}\Big[{1\over 2\tau_i}\|\xi_i-\xi_i^0\|^2+{1\over 2\kappa_i}\|y_i\|^2+$ $\langle \nv{g_i(\xi_i)-g_i(\xi_i^0)},y_i\rangle \Big]$ $+\sum_{k=0}^{K-1}E^{k+1}(\bz)$.} %\label{eq:saddle-rate-dynamic-Theta}
%For $k\geq 1$, define $E^k(\bv)\triangleq\|\be^{k}\| \big(2\gamma\sqrt{N}~B_d+\|\bv-\bv^{k}\|\big)$, which is the error term in \eqref{eq:dynamic-saddle-rate} due to approximating $\cP_{C_d}$ in the $k$-th iteration of the algorithm. The result of Theorem~\ref{thm:dynamic-rate} can be written more explicitly as follows: let $\bar{\bxi}^K\triangleq {1\over K}\sum_{k=1}^K\bxi^k$, $\bar{\bv}^K\triangleq {1\over K}\sum_{k=1}^K\bv^k$, and $\bar{\by}^K\triangleq {1\over K}\sum_{k=1}^K\by^k$, then for any ${\bxi},\in \mathbb{R}^{n}$, $\by,\bv\in \mathbb{R}^{m|\cN|}$ for $n=\sum_{i\in\cN}n_i$, and for all $K\geq 1$, we have
%{\small
%\begin{align}\label{eq:saddle-rate-dynamic}
%\mathcal{L}(\bar{\bxi}^K,\bar{\bv}^K,\by)-&\mathcal{L}({\bxi},\bv,\bar{\by}^K)\leq \Theta(\bxi,\bv,\by)/K,\nonumber\\
%\Theta(\bxi,\bv,\by)\triangleq&{1\over 2\gamma}\|\bv-\bv^0\|^2-\langle {\by}-{\by}^0,~\bv-\bv^0\rangle+\sum_{k=1}^{K}E^k(\bv) \nonumber \\
%&+\sum_{i\in\mathcal{N}}\bigg[{1\over 2\tau_i}\|\xi_i-\xi_i^0\|^2+{1\over 2\kappa_i}\|y_i-y_i^0\|^2 \nonumber\\
%&-\langle R_i(\xi_i-\xi_i^0),y_i-y_i^0\rangle \bigg].
%\end{align}
Given the step-size condition in \eqref{eq:step-rule-dual}, Schur complement condition guarantees that
$\nv{\small
%\begin{equation*}
\begin{bmatrix}
\frac{1}{\tau_i} &C_{g_i}\\
C_{g_i}& \frac{1}{\kappa_i}\\
\end{bmatrix}
\preceq
\begin{bmatrix}
\frac{2}{\tau_i} & {0}\\
{0} & \frac{2}{\kappa_i} \\
\end{bmatrix}}$ for any $i\in\cN$;
%\end{equation*}}
therefore,
{\small
\begin{align}
\label{eq:simple-C-bound-dynamic}
\Theta(\bz)\leq &\sum_{i\in\mathcal{N}}\Big[{1\over\tau_i}\|\xi_i-\xi_i^0\|^2+{1\over \kappa_i}\|y_i\|^2\Big]+{1\over 2\gamma}\|\bw\|^2{+\langle {\by},\bw\rangle+\sum_{k=0}^{K-1}E^{k+1}(\bz).}
\end{align}
}%
%{As argued in the proof of Theorem~\ref{thm:static-error-bounds}, if Assumption~\ref{assump:existence} holds, one can construct a saddle point $(\bx^*,\bw^*,\by^*)$ for $\cL$ in \eqref{eq:dual-dist-problem};}
In the rest, fix %an arbitrary
$K\geq 1$ and a %arbitrary
saddle-point $({\bxi}^*,\bw^*,\by^*)$
%for $\min_{\bxi,\bw}\max_{\by\in\mathcal{B}}\cL(\bxi,\bw,\by)$
of $\cL$ in \eqref{eq:lagrangian-dual-implementation} such that $\bw^*=\mathbf{0}$. %for all $k\geq 1$
%hence, $\mathcal{L}(\bxi^*,\bw^*,\by^*)=\Phi(\bxi^*)$ and $y_i^*\in\cK^\circ$ for $i\in\cN$. As in the proof of Theorem~\ref{thm:static-error-bounds},
\nv{Let $\hat{y}^K\triangleq %2\norm{y^*}\cP_{\cK^\circ}\Big(-\sum_{i\in\cN}g_i(\bar{\xi}^K_i)\Big)/\norm{\cP_{\cK^\circ}\big(-\sum_{i\in\cN}g_i(\bar{\xi}^K_i)\big)}\in\cK^\circ$
2\norm{y^*}\cP_{\cK^\circ}\big(-g(\bar{\bxi}^K)\big)/\|\cP_{\cK^\circ}\big(-g(\bar{\bxi}^K)\big)\|\in\cK^\circ$, and define $\hat{\by}^K=[\hat{y}^K_i]_{i\in\cN}$ such that $\hat{y}^K_i=\hat{y}^K$
} for all $i\in\cN$, i.e., $\hat{\by}^K=\ones\otimes\hat{y}^K\in\Ct$, and also define $\hat{\bw}^K\triangleq \norm{\cP_{{\cC^\circ}}(\bar{\by}^K)}^{-1}\cP_{{\cC^\circ}}(\bar{\by}^K)$, where $\cC\supset\Ct$ defined in \eqref{eq:Cd} is a closed convex cone and $\cC^\circ$ denotes its polar cone. Note that $\hat{\by}^K\in\cC$ and $\hat{\bw}^K\in\cC^\circ$ imply $\fprod{\hat{\by}^K,\hat{\bw}^K}\leq 0$.
Recall that every closed convex cone $\cQ\subset\reals^m$ induces an orthogonal decomposition on $\reals^m$, i.e., according to Moreau decomposition, for any $y\in\reals^m$, there exist $y^1\in\cQ$, and $y^2\in\cQ^\circ$ such that $y=y^1+y^2$ and $y^1\perp y^2$; in particular, $y^1=\cP_\cQ(y)$ and $y^2=\cP_{\cQ^\circ}(y)$. Thus, %\todo{Erfan: Rephrased}
%{\small
%\begin{equation}
%\label{eq:tilde-v}
$\fprod{\hat{\bw}^K,~\bar{\by}^K}=\fprod{\hat{\bw}^K,~\cP_{\cC}(\bar{\by}^K)+\cP_{\cC^\circ}(\bar{\by}^K)}=\norm{\cP_{{\cC^\circ}}(\bar{\by}^K)}={d_{\cC}(\bar{\by}^K)}$.
%\end{equation}}%
Note that for each $i\in\cN$ we have \nv{$\bar{y}_i^{K}\in\cK^\circ$} since $y_i^k\in\cK^\circ$ for all $k=1,\ldots,K$ and $\cK$ is convex; hence, $\sigma_{\cK}(\bar{y}_i^K)=0$ for $i\in\cN$. Moreover, $\hat{\bw}^K\in\cC^\circ$ implies $\sigma_{\cC}(\hat{\bw}^K)=\ind{\cC^\circ}(\hat{\bw}^K)=0$; and since $\Ct\subset\cC$, we also have $\sigma_{\Ct}(\hat{\bw}^K)\leq\sigma_{\cC}(\hat{\bw}^K)=0$. Therefore, we can conclude that $\sigma_{\Ct}(\hat{\bw}^K)=0$ since $\mathbf{0}\in\Ct$. {These observations} %and \eqref{eq:tilde-v}
imply that
%$\cL$ in~\eqref{eq:lagrangian-dual-implementation} satisfies
{\small
\begin{align}
\label{eq:Lagrangian_1}
\cL(\bxi^*,\hat{\bw}^K,\bar{\by}^K)=\varphi(\bxi^*)-\sum_{i\in\cN}\fprod{\nv{g_i({\xi}_i^*)},~\bar{y}_i^K}-{d_{\cC}(\bar{\by}^K)}.
\end{align}}%
%Furthermore, similar to \eqref{eq:tilde-v},
{Similarly,} from the definition of \nv{$\hat{y}^K\in\cK^\circ$},
%\nv{\small
%\begin{equation}
%\label{eq:tilde-theta-dynamic}
{\small $-\sum_{i\in\cN}\fprod{g_i(\bar{\xi}_i^K),~\hat{y}^K}=2\norm{y^*}d_{\cK}\Big(-g(\bar{\bxi}^K)\Big)$},
%\end{equation}}%
and since $\hat{\by}^K\in \Ct$, we also have
%{\small
%\begin{align}\label{eq:support-zero}
$\fprod{\bar{\bw}^K,~\hat{\by}^K}-\sigma_{\Ct}({\bar{\bw}^K}) \leq \sup_{\substack{\bw}}\fprod{\bw,~\hat{\by}^K}-\sigma_{\Ct}({\bw})= \ind{\Ct}(\hat{\by}^K)=0$.
%\end{align}
%}%
Note \nv{$\sigma_{\mathcal{K}}(\hat{y}^K)=0$ since $\hat{y}^K\in\cK^\circ$}. Thus, %\eqref{eq:tilde-theta-dynamic} and \eqref{eq:support-zero} imply
{we conclude that $\cL$ satisfies}
\nv{\small
\begin{align}
\label{eq:Lagrangian_2}
\cL(\bar{\bxi}^K,\bar{\bw}^K,\hat{\by}^K)\geq\varphi(\bar{\bxi}^K)+2\norm{y^*}d_{\cK}\Big(-g(\bar{\bxi}^K)\Big).
\end{align}}%
Combining \eqref{eq:Lagrangian_1} and \eqref{eq:Lagrangian_2}, we get
\nv{\small
\begin{multline}
\label{eq:Lagrangian_difference_1}
\cL(\bar{\bxi}^K,\bar{\bw}^K,\hat{\by}^K)-\cL(\bxi^*,\hat{\bw}^K,\bar{\by}^K) \\
\geq \varphi(\bar{\bxi}^K)-\varphi(\bxi^*)+2\norm{y^*}d_{\cK}\Big(-g(\bar{\bxi}^K)\Big)+d_{\cC}(\bar{\by}^K)+\sum_{i\in\cN}\fprod{g_i(\xi_i^*),~\bar{y}_i^K}.
\end{multline}}%
Moreover, $\fprod{\hat{\by}^K,~\hat{\bw}^K}\leq 0$, \eqref{eq:saddle-rate-dynamic} and \eqref{eq:simple-C-bound-dynamic} imply that
{\small
\begin{align}
\label{eq:Lagrangian_difference_bound_1}
\cL(\bar{\bxi}^K,\bar{\bw}^K,\hat{\by}^K)-\cL(\bxi^*,\hat{\bw}^K,\bar{\by}^K)\leq\Theta(\hat{\bz}^K)/K\leq \frac{\Lambda_1+\sum_{k=0}^{K-1}E^{k+1}(\hat{\bz}^K)}{K}\triangleq\frac{\Lambda(\gamma,\beta)}{K},
\end{align}}%
where $\hat{\bz}^K=[{\bxi^*}^\top~(\hat{\bw}^K)^\top~(\hat{\by}^K)^\top]^\top$ and \nv{$\Lambda_1\triangleq\frac{1}{2\gamma}+\sum_{i\in\mathcal{N}}\big[{1\over \tau_i}\|\xi^*_i-\xi_i^0\|^2+\frac{4}{\kappa_i}\|y^*\|^2\big]$}.
Recall that we fixed a saddle point $(\bxi^*,\bw^*,\by^*)$ such that $\bw^*=\mathbf{0}$; hence, we have $\cL(\bxi^*,\bw^*,\by^*)=\varphi(\bxi^*)$ and $\sigma_{\Ct}(\bw^*)=0$. Moreover, since $(\bxi^*,\bw^*,\by^*)$ is a saddle-point, %for $\cL$ in~\eqref{eq:lagrangian-dual-implementation},
we have $\mathcal{L}(\bar{\bxi}^K,{\bw}^*,\by^*)-\mathcal{L}({\bxi}^*,\bw^*,\by^*) \geq 0$ and $\cL(\bxi^*,{\bw}^*,\by^*)-\cL(\bxi^*,\bw^*,\bar{\by}^K)\geq 0$; therefore, these facts imply that
\nv{\small
\begin{align}\label{eq:lower-bound-dynamic-2}
{\sum_{i\in\cN}\langle g_i(\xi^*_i),~\bar{y}_i^K\rangle\geq 0},
\qquad \varphi(\bar{\bxi}^K)-\varphi(\bxi^*)+\norm{y^*}d_{\cK}\Big(-g(\bar{\bxi}^K)\Big) \geq 0,
\end{align}}%
where we used %the fact that
\nv{$y^*\in\cK^\circ$, %and $\bw^*\in\cC^\circ$,
i.e., $\fprod{y^*,y}{\leq}\fprod{y^*,\cP_{\cK^\circ}(y)}\leq\norm{y^*}d_{\cK}(y)$} for all $y\in\reals^m$.
%and similarly $\fprod{\bw^*,\bw}\sa{\leq}\norm{\bw^*}d_{\cC}(\bw)$ for all $\bw$.
Therefore, combining \eqref{eq:Lagrangian_difference_1}, \eqref{eq:Lagrangian_difference_bound_1}, and \eqref{eq:lower-bound-dynamic-2} gives us the {\it infeasibility} and {\it consensus} results in \eqref{eq:infeasibility} and also the upper bound in \eqref{eq:suboptimality}; while \nv{the inequality on the left in~\eqref{eq:lower-bound-dynamic-2}} gives us the lower bound for the {\it suboptimality}.

\nv{To show that $\Lambda(\gamma,\beta)$ is finite and independent of $K$}, we %provide a
bound $\sum_{k=0}^{K-1}E^{k+1}(\hat{\bz}^K)$. %As also used in
\nv{As in~\eqref{eq:error-term-sum-bound}, using \eqref{eq:error-term-bound}, \eqref{eq:induction} and \eqref{eq:iterate-bound}, we get}
\vspace*{-2mm}
{\small
\begin{align}\label{eq:E_bound}
\sum_{k=0}^{K-1}E^{k+1}(\hat{\bz}^K)\leq
&3\gamma\max_{k=0,\ldots,K-1}\{\tfrac{1}{\gamma}\norm{\bw^{k+1}-\hat{\bw}^K}+\norm{\by^{k+1}-\hat{\by}^K}\}\sum_{k=0}^{K-1}\norm{\be^{k+1}}\nonumber \\
\nv{\triangleq} &\nv{\Lambda_2}\leq 3\nu C_0\Big(1+\sqrt{8A_0\gamma}+\gamma\big(24\nu C_0+\nv{3\sqrt{N}\norm{y^*}}\big)\Big),
\end{align}}%
-- recall $\sum_{k=0}^\infty\norm{\be^{k+1}}=\nu C_0$.
%Note that
%\sa{Similarly one can %easily
%derive that $\sum_{k=0}^{K-1}E^{k+1}(\bz^*)\leq 3(\sqrt{8A_0\gamma}+\gamma(24\nu C_0+2\norm{\by^*}))\nu C_0<+\infty$.}
Below we specify the bound in \eqref{eq:E_bound} for both cases.

\nv{For CASE 1 where $B$ is known, $\nu=4N^{\frac{3}{2}}\Gamma B$ (see~\eqref{eq:e-bound}) and $\norm{y^*}\leq B$; hence, using these facts and the bound on $A_0$ given in \eqref{eq:A_bound} within \eqref{eq:E_bound}, we get}
\vspace*{-2mm}
\nv{\small
\begin{align*}
\Lambda_2\leq %=\sum_{k=0}^{K-1}E^{k+1}(\hat{\bz}^K)
%&\leq3(1+\sqrt{8A_0\gamma}+\gamma(24\nu C_0+3B))\nu C_0\\
N^{\frac{3}{2}}\Gamma{C_0} B~\cO\Big(1+\sqrt{\gamma N(B\bar{R}_x^2+B^2\bar{C}_g)}+\gamma {N^{\frac{3}{2}}\Gamma{C_0} B}\Big).
\end{align*}}%
%which means that $\Lambda_2(K)=\cO(N^{\frac{3}{2}}\Gamma B+{\sqrt{\gamma}N^{\frac{3}{2}}\Gamma B^{\frac{3}{2}}R_x}+\gamma N^3\Gamma^2 B^2)$.
%On the other hand, we can observe that
Moreover, the second inequality in~\eqref{eq:A_bound} implies \nv{$\Lambda_1=%\cO(\frac{1}{\gamma}+A_0)=
\cO(\frac{1}{\gamma}+N(B\bar{R}_x^2+B^2\bar{C}_g)+\gamma NB^2)$}.
%Therefore, the magnitude of $\cO(1)$ constant in terms of $N$ and $\Gamma$ in the final bound is $\Lambda_1+\Lambda_2(K)=\cO(\frac{1}{\gamma}+N^{\frac{3}{2}}\Gamma B+\sqrt{\gamma}N^{\frac{3}{2}}\Gamma B^{\frac{3}{2}}R_x+\gamma N^3\Gamma^2 B^2)$.
\nv{Our aim is to optimize the $\cO(1)$ constant of $\Lambda_1+\Lambda_2$ via %optimizing the final bound over
carefully selecting the free parameter $\gamma$. Setting %\sa{$\gamma=(\bar{R}_x^2+\bar{C}_g B)/(N^2\Gamma^2C_0^2B)$}
{$\gamma=(N^{3/2}\Gamma C_0 B)^{-1}$} gives $\Lambda(\gamma,\beta)=\cO\Big(NB(\bar{R}_x^2+\bar{C}_g B)+N^{3/2}\Gamma C_0 B\Big(1+\sqrt{\tfrac{\bar{R}_x^2+\bar{C}_g B}{N^{1/2}\Gamma C_0}}\Big)\Big)$ which implies the $N$ dependency in \eqref{eq:Lambda_Bknown}. %$\Lambda_1+\Lambda_2(K)=\cO\big(N^2\Gamma^2C_0^2+NB(\bar{R}_x^2+\bar{C}_g B)\big)$
%for sufficiently large $N$ such that $\Gamma\sqrt{N}\geq (\bar{R}_x^2+B\bar{C}_g)$.
}

For CASE 2 where $B$ is not known, $\nu=\frac{24}{23}N\Gamma(\norm{\by^*}+\sqrt{\frac{8A_0}{\gamma}})$ \sa{-- see the discussion below~\eqref{eq:error-bound-recursive}; hence, from~\eqref{eq:E_bound}, we get $\Lambda_2\leq \cO(N^2\Gamma^{2}{C_0^2}(A_0+{\gamma} N\sa{\max\{1,\norm{y^*}^2\}}))$. For the sake of simplicity, suppose $\norm{y^*}\geq 1$.}
\begin{comment}
\vspace*{-3mm}
{\small
\begin{align*}
\Lambda_2(K)&\leq 3(1+\sqrt{8A_0\gamma}+\gamma(24\nu C_0+2\norm{\by^*}))\nu C_0\\
&=3(1+\frac{23\gamma}{24N\Gamma}\nu+24\gamma\nu C_0+\gamma\norm{\by^*})\nu C_0\\
&=\cO(\gamma\nu^2)=\cO(N^2\Gamma^{2}(A_0+{\gamma} N\norm{y^*}^2)).
\end{align*}}%
\end{comment}
%Similar to the previous case
\nv{Moreover, $\Lambda_1=\cO(\frac{1}{\gamma}+\bar{A}_0)$, and since $A_0\leq\bar{A}_0$ -- see~\eqref{eq:A_bound}, $\Lambda_1+\Lambda_2=\cO(\tfrac{1}{\gamma}+N^2\Gamma^2{C_0^2}\bar{A}_0)$.}
%Moreover, we need to calculate $A_0$ in terms of parameters of interest. In fact, we know that $A_0\leq (\beta+1)N\bar{R}_x^2+(\bar{C}_g+\frac{5\gamma}{2})N\norm{y^*}^2$ where an
Selecting \sa{$\gamma=N^{\frac{3}{2}}\Gamma {C_0}\bar{R}_x^2$}, \eqref{eq:A_bound} and the bound on $\beta$ %given
in \eqref{eq:beta-cond-sol} together imply that
\begin{comment}
Substituting the bound on $\beta$ in $A_0$ and some calculations reveal that
$A_0=\cO(\frac{1}{\gamma}N^4\Gamma^2\bar{R}_x^4+N^2\Gamma \bar{R}_x^2(\frac{1}{\sqrt{\gamma}}R_x+(1+\sqrt{\frac{\bar{C}_g}{\gamma}})\sqrt{N}\norm{y^*})+(\bar{C}_g+\frac{5\gamma}{2})N\norm{y^*}^2)$.
Therefore, the $\cO(1)$ constant in terms of $N$ and $\Gamma$ in the final bound is
\sa{\small
\begin{align*}
\Lambda_1+\Lambda_2(K)=\cO(N^2\Gamma^2(\frac{1}{\gamma}N^4\Gamma^2\bar{R}_x^4+\frac{1}{\sqrt{\gamma}}N^{\frac{5}{2}}\Gamma \bar{R}_x^2(\bar{R}_x+\sqrt{\bar{C}_g}\norm{y^*})+N^{\frac{5}{2}}\Gamma \bar{R}_x^2\norm{y^*}+\gamma N\norm{y^*}^2))
\end{align*}}%
\end{comment}
%which can be optimized by selecting $\gamma=N^{\frac{3}{2}}\Gamma \bar{R}_x^2/\norm{y^*}$, leading to
$\Lambda_1+\Lambda_2=\cO(N^{\frac{9}{2}}\Gamma^3{C_0^3}\bar{R}_x^2\sa{\norm{y^*}^2})$, assuming %$N^{\frac{3}{2}}\Gamma>1/\norm{y^*}$ and $N^{\frac{3}{2}}\Gamma>\bar{C}_g\norm{y^*}/\bar{R}_x^2$
\sa{$N^{\frac{3}{2}}\Gamma>\bar{C}_g/\bar{R}_x^2$ and $N>1/\bar{R}_x^2$} which are reasonable since we are interested in the bounds when $N$ is large. \nv{Moreover, when $g_i$'s are linear functions ($L_{g_i}=0$) the bound $\bar{A}_0$ can be simplified, i.e.,
%and $\beta$ will disappear in~\eqref{eq:A_bound}. It is easy to verify that
$\bar{A}_0=N(\bar{R}_x^2+(\bar{C}_g+\frac{5\gamma}{2})\norm{y^*}^2)$. Therefore, %$\Lambda_1+\Lambda_2(K)=\cO(\frac{1}{\gamma}+N^3\Gamma^2\bar{R}_x^2+N^3\Gamma^2\bar{C}_g\norm{y^*}^2+\gamma N^3\Gamma^2\norm{y^*}^2)$ and by
choosing \sa{$\gamma=(N^{\frac{3}{2}}\Gamma{C_0})^{-1}$}, we %obtain the final bound of
get $\Lambda_1+\Lambda_2=\cO(N^3\Gamma^2{C_0^2}(\bar{R}_x^2+\bar{C}_g\norm{y^*}^2))$.}
\begin{remark}\label{rem:stronger-cond}
For CASE 1, assuming %more on $q_k$ such that
$\sum_{k=0}^\infty \alpha^{q_k} (k+1)^2<+\infty$ in addition to $C_0<+\infty$, one can observe that using \eqref{eq:mu-bound} and \eqref{eq:e-bound}, the bound $\cO(1)$ bound takes a simpler form: $\Lambda(\gamma,\beta)=\Lambda_1+\sum_{k=0}^{K-1} E^{k+1}(\hat{\bz}^K)\leq \frac{1}{\gamma}+N(B\bar{R}_x^2+B^2\bar{C}_g)+\gamma NB^2+12N^{\frac{3}{2}}\Gamma B[\sum_{k=1}^K\alpha^{q_{k-1}}k+4\sqrt{N}B\gamma\sum_{k=1}^K\alpha^{q_{k-1}}k(k+1)]$.%~\todo{Verify the bound.}
%which leads to the final bound $\Lambda_1+\Lambda_2(K)=\cO(\frac{1}{\gamma}+N^{\frac{3}{2}}\Gamma B +\gamma N^2\Gamma B^2)$, and it can be optimized by setting $\gamma=(N\sqrt{\Gamma}B)^{-1}$ to get $\Lambda_1+\Lambda_2(K)=\cO(N^{\frac{3}{2}}\Gamma B)$.
\end{remark}
\section{Fully distributed step-size rule}
\label{sec:stepsize}
Recall that step-size selection rule in \eqref{eq:step-rule-dual} of Theorem~\ref{thm:dynamic-rate} requires some sort of coordination among the nodes in $\cN$ because there is a fixed $\gamma>0$ coupling and affecting all nodes' step-size choice. To overcome this issue, we will define $\gamma_i>0$ for each node, and let the nodes to choose this parameter independently. Let {$\mathbf{D}_\gamma\triangleq\diag([\frac{1}{\gamma_i}\id_{m}]_{i\in\cN})\succ 0$}
%$\cD=\diag([\gamma_i\bI_m]_{i\in\cN})\succ 0$
and define {$\bgm\triangleq [\gamma_i]_{i\in\cN}$ and
$\widehat{\cC}\triangleq\Big\{{\bp}\in\cY:\ \exists \bar{y}\in\reals^m\ \hbox{s.t.}\ {1\over \sqrt{\gamma_i}}~p_i=\bar{y}\quad \forall i\in\cN, \quad \norm{\bar{y}}\leq 2B  \Big\}$ -- here, $\bp=[{p}_i]_{i\in\cN}$.
}%
Recall the definition of Bregman distance function given in Definition~\ref{def:bregman}: $D_x(\bx,\bar{\bx})=\frac{1}{2}\norm{\bxi-\bar{\bxi}}_{\mathbf{D}_\tau}^2+\frac{1}{2}\norm{\bw-\bar{\bw}}_{\mathbf{D}_\gamma}^2$. Switching to $\mathbf{D}_\gamma$ as defined above, \eqref{eq:pock-pd-3-v} should be replaced with
$\bw^{k+1}\gets{\argmin_{\bw} \sigma_{\Ct}(\bw)-\langle {\by}^k,~\bw \rangle +{1\over 2}\|\bw-\bv^k\|_{\mathbf{D}_\gamma}^2}.$
Using the change of variables $\hat{\bw}\triangleq\mathbf{D}_\gamma^{{1\over 2}}\bw$, it can be rewritten as
{\small
\begin{align}\label{eq:v-update}
{\bw^{k+1}\gets\mathbf{D}_\gamma^{-{1\over 2}}\argmin_{\hat{\bw}} \sigma_{\widehat{\cC}}~(\hat{\bw})+{1\over 2}\|\hat{\bw}-(\mathbf{D}_\gamma^{{1\over 2}}\bv^k+\mathbf{D}_\gamma^{-{1\over 2}}\by^k)\|^2}, %\quad \bv^{k+1}\gets\bw^{k+1},
\end{align}
}%
where we use the fact that {$\sigma_{\Ct}(\mathbf{D}_\gamma^{-{1\over 2}}\hat{\bw})=\sigma_{\widehat{\cC}}~(\hat{\bw})$}. Now, we can write \eqref{eq:v-update} in a proximal form and using Moreau's decomposition, we get
{\small
\begin{align*}
\bw^{k+1}&=\Dg^{-{1\over 2}}\prox{\sigma_{\widehat{\cC}}}(\Dg^{{1\over 2}}\bv^k+\Dg^{-{1\over 2}}\by^k) =\Dg^{-{1\over 2}}\big(\Dg^{{1\over 2}}\bv^k+\Dg^{-{1\over 2}}\by^k-\cP_{\widehat{\cC}}(\Dg^{{1\over 2}}\bv^k+\Dg^{-{1\over 2}}\by^k)\big).
\end{align*}}%
{Note $\bp\in\widehat{\cC}$ implies that $\Dg^{{1\over 2}}\bp=\ones_N\otimes\bar{y}$ for some $\bar{y}\in\reals^m$ such that $\norm{\bar{y}}\leq 2B$.} Therefore, for $\by=[y_i]_{i\in\cN}\in\reals^{n_0}$, the projection of $\Dg^{-{1\over 2}}\by$ onto $\widehat{\cC}$ can be computed as %follows: %simplified in terms of weighted average and projecting on ball $\cB$ as follows,
{\small
\begin{align}
\cP_{\widehat{\cC}}(\Dg^{-{1\over 2}}\by)&=\argmin_{\bp\in\widehat{\cC}} \tfrac{1}{2}\Big\|\Dg^{-{1\over 2}}\by-\bp\Big\|^2=\Dg^{-{1\over 2}} \Big( \ones \otimes \argmin_{\norm{\bar{y}}\leq 2B}{1\over 2}\Big\|\Dg^{-{1\over 2}}\by-\Dg^{-{1\over 2}}(\ones\otimes\bar{y})\Big\|^2\Big) \nonumber \\
%&=\Dg^{-{1\over 2}} \left( \ones \otimes \argmin_{\norm{\bar{y}}\leq 2B}{1\over 2}\sum_{i\in\cN}\gamma_i\norm{y_i-\bar{y}}^2\right) \nonumber \\
%&=\Dg^{-{1\over 2}}\left( \ones \otimes \argmin_{\norm{\bar{y}}\leq 2B}{1\over 2}\norm{\bar{y}-{\sum_{i\in\cN}\gamma_i y_i \over \sum_{i\in\cN}\gamma_i}}^2\right)\nonumber\\
&=\Dg^{-{1\over 2}} \cP_\cB\Big(\frac{1}{\sum_{i\in\cN}\gamma_i}(\one\one^\top\otimes\id_m)\Dg^{-1}\by\Big).
\end{align}
}%
Let %$\cP_\gamma(\by)\triangleq \cP_\cB\left(\frac{1}{\sum_{i\in\cN}\gamma_i}(\ones\ones^\top\otimes\id_m)\Dg^{-1}\by\right)$
$\cP_\gamma(\by)\triangleq \ones_N \otimes \cP_{\cB_0}\left(\frac{1}{\sum_{i\in\cN}\gamma_i}\sum_{i\in\cN}\gamma_i y_i\right)$; hence, we get that
{\small
\begin{align}\label{eq:w-update-moreau}
{\bw^{k+1}=\Dg^{-1}\Big(\Dg\bv^k+\by^k-\cP_\gamma(\Dg\bv^k+\by^k)\Big)}.
\end{align}
}%
Thus, we propose approximating \eyy{$\cP_\gamma(\cdot)$} using an approximate convex combination operator $\cR_\gamma^k(\cdot)=[\cR_{i}^k(\cdot)]_{i\in\cN}$ such that it can be computed in a distributed way, i.e., $\cR_{i}^k(\cdot)$ can be computed at $i\in\cN$ using local communication. More precisely, suppose $\cR_\gamma^k$ satisfies a slightly modified version of Assumption~\ref{assump:approximate-average}, where \eqref{eq:approx_error-for-full-vector} is replaced with
\begin{align}
\label{eq:approx_error-for-full-vector-gamma}
{\mathcal{R}_\gamma^k(\bw)\in\cB,\qquad \|\mathcal{R}_\gamma^k(\bw)-\mathcal{P}_{\gamma}(\bw)\| \leq N~\Gamma \alpha^{q_k}\norm{\bw}},\quad\forall~\bw\in\reals^{n_0}.
\end{align}
Provided that such an operator exists, instead of \eqref{eq:inexact-rule-dual}, we set $\bv^{k+1}$ as follows:
{\small
\begin{align}\label{eq:v-update-moreau}
{\bv^{k+1}\gets\Dg^{-1}\Big(\Dg\bv^k+\by^k-\cR_\gamma^k(\Dg\bv^k+\by^k)\Big)}.
\end{align}
}%
%\todo{Erfan: Removed a sentence}
%Thus, after setting $\mathbf{D}_\gamma\triangleq\diag([\frac{1}{\gamma_i}\id_{m}]_{i\in\cN})\succ 0$, the iterate sequence $\{\bxi^k,\bv^k$ $,\by^k\}_{k\geq 0}$ corresponding to the slightly modified version of Algorithm~DPDA-D is generated by the recursion in \eqref{eq:v-update-moreau}, \eqref{eq:pock-pd-3-xi}, and \eqref{eq:pock-pd-3-y}.
With this modification, we can still show that the iterate sequence converges to a primal-dual optimal solution with $\cO(1/K)$ ergodic rate provided that primal-dual step-sizes $\{\tau_i,\kappa_i\}_{i\in\cN}$ and $\{\gamma_i\}_{i\in\cN}$ are chosen such that %\ey{$\tau_i<\frac{1}{L_{f_i}+2B L_{g_i}}$, $\kappa_i<\frac{1}{\gamma_i}$, and $\left({1\over \tau_i}-(L_{f_i}+2B L_{g_i})\right)\left( {1\over \kappa_i}-\gamma_i\right) > C_{g_i}^2$},
\nv{$\tau_i=(\max\{1, L_{f_i}+\beta L_{g_i}\}+C_{g_i})^{-1}$, $\kappa_i=(C_{g_i}+\frac{5\gamma_i}{2})^{-1}$}
for all $ i\in\mathcal{N}$.

In the rest of this section, for both undirected and directed time-varying communication networks, we provide an operator $\cR_\gamma^k$ satisfying \eqref{eq:approx_error-for-full-vector-gamma}. For $\by=[y_i]_{i\in\cN}\in\cY$, define
${p}_\gamma(\by)\triangleq{1\over \sum_{i\in\cN}\gamma_i}\sum_{i\in\mathcal{N}}\gamma_i y_i;$
hence, we have $\cP_\gamma(\by)=\ones_N\otimes \cP_{\cB_0}\big({p}_\gamma(\by)\big)$. Therefore, we should consider distributed approximation of ${p}_\gamma(\by)$. Given $y_i\in\reals^m$ and $\gamma_i>0$, which are only known at node $i\in\cN$, we next discuss extensions of techniques discussed in Section~\ref{sec:inexact-averaging} to compute the convex combination
${\sum_{i\in\cN} \gamma_i y_i/\sum_{i\in\cN}\gamma_i}$.

First, suppose that $\{\cG^t\}$ is a time-varying undirected graph and $\{V^t\}_{t\in\integers_+}$ be a corresponding sequence of weight matrices satisfying Assumption~\ref{assump:undirected}. For $\bw=[w_i]_{i\in\cN}\in\cY$ such that $w_i\in\reals^m$ for $i\in\cN$, define
{\small
\begin{equation}
\label{eq:approx-average-dual-gamma}
\cR_\gamma^k(\bw)\triangleq\mathcal{P}_{\mathcal{B}}\left(\left(\diag(W^{t_k+q_k,t_k}\bgm)^{-1}W^{t_k+q_k,t_k}\otimes\id_m\right)~\Dg^{-1}\bw\right)
\end{equation}
}%
to approximate $\mathcal{P}_{\gamma}(\cdot)$ in \eqref{eq:w-update-moreau}. Note that $\mathcal{R}_\gamma^k(\cdot)$ can be computed in a \emph{distributed fashion} requiring $q_k$ communications with the neighbors for each node.
%\todo{Erfan: Removed a sentence}
%According to Lemma~\ref{lem:approximation}, we have $(W^{t_k+q_k,t_k}\otimes\id_m)~\Dg^{-1}\bw\rightarrow \ones_N\otimes\tfrac{1}{N}\sum_{i\in\cN}\gamma_i w_i$ and $W^{t_k+q_k,t_k}\bgm\rightarrow \ones_N\otimes\tfrac{1}{N}\sum_{i\in\cN}\gamma_i$ as $k\rightarrow\infty$; hence, $\cR_\gamma^k(\bw)\rightarrow \mathcal{P}_{\gamma}(\bw)$ as $k\rightarrow\infty$. Moreover,
Using Lemma~\ref{lem:approximation}, it is easy to show that $\cR_\gamma^k$ given in \eqref{eq:approx-average-dual-gamma} satisfies the condition in \eqref{eq:approx_error-for-full-vector-gamma}.

Second, suppose that $\{\cG^t\}$ is a time-varying $M$-strongly-connected directed graph, and $\{V^t\}_{t\in\integers_+}$ be the corresponding weight-matrix sequence as defined in \eqref{eq:directed-weights} within Section~\ref{sec:directed} -- so that \eqref{eq:approx-average-dual-gamma} can be computed over time-varying \emph{directed} network.
Given any $\bw=[w_i]_{i\in\cN}$ and $\{\gamma_i\}_{i\in\cN}$, the results in~\cite{nedic2015distributed} immediately imply that for any $s\in\integers_+$, the vector $\big(\diag(W^{t,s}\bgm)^{-1}W^{t,s}\otimes\id_m\big)\mathbf{D}_\gamma^{-1}\bw$ converges to the consensus convex combination vector
%of $[w_i]_{i\in\cN}$ corresponding to the node-specific coefficients $\{\gamma_i\}_{i\in\cN}$
%$\frac{1}{\sum_{i\in\cN}\gamma_i}\sum_{i\in\cN}\gamma_i w_i$
$\ones_N\otimes p_\gamma(\bw)$ with a geometric rate as $t$ increases. Indeed, this can be trivially achieved by using a different initialization for the {push-sum method}. %in \eqref{push-sum}
%from a different set of points.
%\todo{I removed a sentence here to be consistent with section 2.1.2} %; in particular, instead of setting $\eta_i^0=w_i$ for $i\in\cN$ and $\nu^0=\one\in\reals^{|\cN|}$, we set $\eta_i^0=\gamma_i w_i$ for $i\in\cN$ and $\nu^0=\bgm\in\reals^{|\cN|}$.
Next, we state a slightly modified version of the convergence result in Lemma~\ref{lem:push-sum}.
%\todo{Erfan: Removed a sentence}
%to quantify the approximation quality of $\cR_\gamma^k(\cdot)$ in \eqref{eq:approx-average-dual-gamma} for computing convex combinations.
%Although the proof immediately follows from the proof Lemma~1 in~\cite{nedic2015distributed}, for the sake of completeness we provide the proof below.
\begin{lemma}
\label{lem:push-sum-gamma}
Suppose that the digraph sequence $\{\cG^t\}_{t\geq 1}$ is uniformly strongly connected ($M$-strongly connected), where $\cG^t=(\cN,\cE^t)$. Given node-specific data $\{w_i\}_{i\in\cN}$ $\subset\reals^m$ and $\{\gamma_i\}_{i\in\cN}\subset\reals_{++}$,
%consider the sequence $\{\vartheta_i^t\}_{t\geq 1}$ generated for each $i\in\cN$ by push-sum iterations as in~\eqref{push-sum} starting from $\eta_i^0=\gamma_i w_i$ for $i\in\cN$ and $\nu^0=\bgm\in\reals^{|\cN|}$.
for any fixed integer $s\geq 0$, the following bound holds for all integers $t>s$:
\vspace*{-2mm}%
{\small
\begin{align*}
\norm{\diag(W^{t,s}\bgm\otimes\id_m)^{-1}(W^{t,s}\otimes\id_m)~\Dg^{-1}\bw - \one_N\otimes p_\gamma(\bw)}\leq {8\sqrt{N}\over \gamma_{\min}\delta} \sum_{i\in\cN}\gamma_i\norm{w_i}~\alpha^{t-s-1},
\end{align*}
}%
for some $\delta\geq {1\over N^{NM}}$ and $0<\alpha\leq \left(1-{1\over N^{NM}}\right)^{1\over M}$, where  $N=|\cN|$ and $\gamma_{\min}=\min_{i\in\cN}\gamma_i$.
%and $\norm{{\by}}_{1,2}\triangleq\sum_{k=1}^m\sqrt{\sum_{j\in\cN}|[{y}_j]_k|^2}$ for any ${\by}=[{y}_i]_{i\in\cN}$ and ${y}_i\in\reals^m$ for $i\in\cN$.
\end{lemma}
\begin{proof}
The proof follows from Corollary 2 and the proof of Lemma~1 in \cite{nedic2015distributed}.
\end{proof}
Thus, $\cR_\gamma^k(\cdot)$ defined in \eqref{eq:approx-average-dual-gamma} satisfies the requirement $\|\mathcal{R}_\gamma^k(\bw)-\mathcal{P}_{\gamma}(\bw)\|\leq N \Gamma~\alpha^{q_k}\norm{\bw}$ in \eqref{eq:approx_error-for-full-vector-gamma} for $\Gamma=\frac{\norm{\bgm}}{\gamma_{\min}\sqrt{N}}{8\over \delta\alpha}$ and for some $\alpha\in(0,1)$ and $\delta>0$ as stated in Lemma~\ref{lem:push-sum-gamma}.

\section{\sa{A Distributed Algorithm for Static Network Topology}}
\label{sec:static-nonlinear}
%In this section,
We extend the results in \cite{aybat16_dual_static} to %the case where
nonlinear constraint functions $\{g_i\}_{i\in\cN}$. %In fact, considering the problem \eqref{eq:central_problem},
Given an \emph{undirected}, \emph{static} communication network $\cG=(\cN,\cE)$,
%The notations in this section can be simply understood from the previous notations by dropping the time index $t$.
following the discussion in Section~\ref{sec:static}, the corresponding SP problem for the static network %can be formulated
is given as $\min_{\substack{\bxi},\bw} \max_{\substack{\by}}\big\{\sum_{i\in \cN}\varphi_i(\xi_i)-\langle g_i(\xi_i),~y_i \rangle -\langle \bw,~M\by\rangle:\ y_i\in\cK^\circ,\ \  \forall i\in\cN \big\}$. In Fig.~\ref{alg:DPDA-S}, we propose a modified version of DPDA-S algorithm~\cite{aybat16_dual_static} %which is an extension of the algorithm
(see Section \ref{sec:static}) to solve \eqref{eq:central_problem} over $\cG$.
\begin{figure}[h]
\vspace*{-2mm}
\centering
\framebox{\parbox{0.98\columnwidth}{
{\small
\textbf{Algorithm DPDA-S} ( $\bxi^{0},\gamma,\{\tau_i,\kappa_i\}_{i\in\cN}$ ) \\[1.5mm]
Initialization: $y_i^0\gets 0$, $s_i^0\gets 0$, \quad $i\in\cN$\\
Step $k$: ($k \geq 0$)\\
\text{ } 1. $\xi_i^{k+1}\gets\prox{\tau_i\rho_i}\Big(\xi_i^k-\tau_i\Big(\nabla f_i(\xi_i^k)-\bJ g_i^\top y^k_i\Big)\Big)$, \quad $p_i^{k+1}\gets \sum_{j\in\cN_i}(s_i^k-s_j^k)$, $\quad i\in\cN$\\[0.1mm]
%\text{ } 2. $p_i^{k+1}\gets \sum_{j\in\cN_i}(s_j^k-s_i^k), \quad i\in\cN$\\[0.1mm]
\text{ } 2. $y_i^{k+1}\gets\cP_{\cK^\circ\cap \cB_0}\Big[y_i^k-\kappa_i\Big(2g_i(\xi_i^{k+1})-g_i(\xi_i^k)+\gamma p_i^{k+1}\Big)\Big], \quad i \in \cN$\\[0.1mm]
\text{ } 3. $s_i^{k+1}\gets y_i^{k+1}+\sum_{\ell=0}^{k+1} y_i^\ell, \quad i \in \cN$
}}}
\vspace*{-2mm}
\caption{\small Distributed Primal Dual Algorithm for Static $\cG$ (DPDA-S)}
\label{alg:DPDA-S}
\end{figure}

%For the static case,
DPDA-S needs only \emph{one} communication round per iteration; moreover, since DPDA-S %is free from
does not require inexact averaging (hence no error accumulation), its analysis is much simpler than and directly follows from the analysis of DPDA-D. The following theorem states the convergence rate for the iterates of DPDA-S. %in Fig.~\ref{alg:DPDA-S}.
\begin{theorem}\label{thm:DPDA-S}
Suppose Assumptions~\ref{assum:functions},~\ref{assump:communication_general} with $\cG^t=\cG$ for $t\geq 0$ and \ref{assump:existence} hold. For any $\gamma>0$, let the primal-dual step-sizes $\{\tau_i,\kappa_i\}_{i\in\cN}$ be chosen such that
{\small
\begin{equation}\label{eq:step-rule-dual-static}
\tau_i=(\max\{1, L_{f_i}+\beta L_{g_i}\}+C_{g_i})^{-1},\quad \kappa_i=(C_{g_i}+{\gamma}(4d_{\max}+\tfrac{1}{2}))^{-1},\quad \forall\ i\in\cN.
\end{equation}}%
for some $\beta>0$. Given $B\in(0,\infty]$, let $\cB_0\triangleq\{y\in\reals^m:\ \norm{y}\leq {2B}\}$. Starting from $\bs^0=\by^0=\mathbf{0}$ and an arbitrary $\bxi^0$, let $\{(\bxi^k)\}_{k\geq 0}$ be the primal, and $\{\by^k\}_{k\geq 0}$ be the dual iterate sequence generated by DPDA-S, displayed in Fig.~\ref{alg:DPDA-S}.
For any $\gamma>0$, if $\beta>0$ is chosen as discussed below, then $\{({\bxi}^k,{\by}^k)\}_{k\geq 0}$ converges to $(\bxi^*,\by^*)$ such that $\by^*=\one \otimes y^*$ and $(\bxi^*,y^*)$ is {an optimal} primal-dual solution to \eqref{eq:central_problem}. Moreover, both infeasibility, $F(\bar{\bxi}^K,\bar{\by}^K)$, and suboptimality, $|\varphi(\bar{\bxi}^K)-\varphi({\bxi}^*)|$ are $\cO(1/K)$, i.e.,
{\small
\begin{eqnarray}
&&F(\bar{\bxi}^K,\bar{\by}^K)\triangleq \norm{M\bar{\by}}+\norm{y^*}d_{\mathcal{K}}\left(- g(\bar{\bxi}^K)\right)\leq \frac{\Lambda(\gamma,\beta)}{K}, \label{eq:infeasibility-static}\\
&&0\leq \varphi(\bar{\bxi}^K)-\varphi({\bxi}^*)+\norm{y^*} d_{\mathcal{K}}\left(-g(\bar{\bxi}^K)\right)\leq \frac{\Lambda(\gamma,\beta)}{K}-F(\bar{\bxi}^K,\bar{\by}^K),\label{eq:suboptimality-static}
\end{eqnarray}}%
for all $K\geq 1$, where $\Lambda(\gamma,\beta)=\frac{1}{2\gamma}+\sum_{i\in\cN}\frac{1}{\tau_i}\|\xi_i^*-\xi_i^0\|^2+\tfrac{4}{\kappa_i}\|y^*\|^2$. %$\bar{\bxi}^K={1\over K}\sum_{k=1}^K\bxi^k$ and $\bar{\by}^K={1\over K}\sum_{k=1}^K\by^k$ for $K\geq 1$.

(CASE 1): If a dual bound is known, i.e., $B<\infty$, then \eqref{eq:infeasibility-static} and \eqref{eq:suboptimality-static} hold for $\beta=2B$; moreover, setting %the free parameter
$\gamma=(NB)^{-1}$ gives $\Lambda(\gamma,\beta)=\cO(NB(\bar{R}_x^2+\bar{C}_gB+1))$.

\indent (CASE 2): If the dual bound does not exist, then set $B=\infty$ within DPDA-S. There exists $\bar{\beta}>0$ such that \eqref{eq:infeasibility-static} and \eqref{eq:suboptimality-static} hold for all $\beta\geq \bar{\beta}$; moreover, selecting \sa{$\gamma=\bar{R}_x^2\sqrt{N/d_{\max}}$ leads to $\Lambda(\gamma,\beta)=\cO(N^{\tfrac{3}{2}}\sqrt{d_{\max}}\bar{R}_x^2\max\{1,\norm{y^*}^2\})$ and \sa{$\bar{\beta}=\cO(\sqrt{N d_{\max}}~\max\{1,\norm{y^*}\})$} for $N$ sufficiently large\footnote{For simple bounds, we assume $N>1/\bar{R}_x^2$ and $\sqrt{N d_{\max}}\geq \bar{C}_g/\bar{R}_x^2$.}}.
\end{theorem}
\begin{proof}
The results follow from the analysis of DPDA-D in Section \ref{sec:convergence} and \cite{aybat16_dual_static}.
\end{proof}
\begin{remark}
%In Section \ref{sec:static}, we discussed
\sa{In Theorem~2 of \cite{aybat16_dual_static}, the %convergence
rate result is provided for the case $g_i$ is affine for $i\in\cN$.
%Note that
For this case, a dual bound is not needed; %Therefore,
hence, the suboptimality and infeasibility rate is $\cO(\Lambda/K)$ for some %$\Lambda\in\reals_+$. Furthermore, selecting $\gamma=1/N$ for the step-size conditions gives
$\Lambda=\cO(N(\bar{R}_x^2+\bar{C}_g\norm{y^*}^2))$ when $\gamma=1/N$.}
\end{remark}
\begin{comment}
\begin{remark}
Let $Q(\epsilon)$ be the total number of communications needed to achieve $\epsilon$-optimal and $\epsilon$-feasible solution $(\bxi^\epsilon,\by^\epsilon)$. Then, Theorem \ref{thm:DPDA-S} implies that if the dual bound $B<\infty$ is known, selecting $\gamma=\cO(1/N)$ gives $Q(\epsilon)=\cO(\tfrac{N}{\epsilon})$. If dual bound is not known ,i.e., $B=\infty$, then $\gamma=\cO(1)$ and $\gamma=\cO(1/\sqrt{N})$ leads to $Q(\epsilon)=\cO(\tfrac{N^2}{\epsilon})$ and $Q(\epsilon)=\cO(\tfrac{N^{\frac{7}{4}}}{\epsilon})$ for the case $g_i$ is nonlinear and affine for $i\in\cN$, respectively.
\end{remark}
\end{comment}
\section{Computing a dual bound}
\label{sec:dual-bound}
\noindent Recall that the definition of $\Ct$ in \eqref{eq:C} involves a bound $B$ such that $\norm{y^*}\leq B$ for some dual optimal solution $y^*$. In this section, we show that given a Slater point we can find a ball containing the optimal dual set for problem \eqref{eq:central_problem}. %Indeed, we will prove this result for a more general case where the feasible set $\{\bxi=[\xi_i]_{i\in\cN}: g(\bxi)\in {\cK}\}$ is described by a function $g(\bxi)\triangleq\sum_{i\in\cN}g_i(\xi_i)$ such that $-g$ is $\cK$-convex and $g_i(\xi_i)$ is a private function of $i\in\cN$, e.g., \ey{$g_i(\xi_i)=-G_i(\xi_i)$} in \eqref{eq:central_problem}.
To this end, we first derive some results without assuming convexity.

Let $\varphi:\reals^n\rightarrow\reals\cup\{+\infty\}$ and $g:\reals^n\rightarrow\reals^m$ be arbitrary functions of $\bxi$, and $\cK\subset\reals^m$ be a cone. For now, we do not assume convexity for $\varphi$, $g$, and $\cK$, which are the components of the following generic problem \vspace*{-1mm}
{\small
\begin{equation}
\label{eq:generic problem}
\varphi^*\triangleq\min_{\bxi} \varphi(\bxi)\quad \hbox{s.t.}\quad g(\bxi)\in\nv{-\cK}\ :\ y\in\cK^\circ,
\end{equation}}%
where $y\in\reals^m$ denotes the dual vector. %of dual variables.
Let $q$ denote the dual function, i.e.,
\vspace*{-1mm}%
{\small
\begin{equation}
q(y)\triangleq\left\{
       \begin{array}{ll}
         \inf_{\bxi}\varphi(\bxi)\nv{-} y^\top g(\bxi), & \hbox{if $y\in\cK^\circ$;} \\
         -\infty, & \hbox{o.w.}
       \end{array}
     \right.
\end{equation}}%
We assume that there exists $\hat{y}\in\cK^\circ$ such that $q(\hat{y})>-\infty$. Since $q$ is a closed concave function, this assumption implies that $-q$ is a \emph{proper} closed convex function. Next we show that for any $\bar{y}\in\dom q=\{y\in\reals^m:\ q(y)>-\infty\}$, the superlevel set $Q_{\bar{y}}\triangleq\{y \in \dom q:\ q(y)\geq q(\bar{y})\}\subset\cK^\circ$ is contained in a Euclidean ball centered at the origin, of which radius can be computed efficiently. A special case of this dual boundedness result is well known when $\cK=\reals^m_{+}$~\cite{uzawa58}
%\todo{Erfan: Removed a sentence}
%, and its proof is very simple and based on exploiting the componentwise separable structure of $\cK=\reals^m_{+}$
-- see Lemma~1.1 in~\cite{Nedic08_1J}; however, it is not trivial to extend this result to %our setting where $\cK$ is
%\todo{Erfan: Rephrased}
{an \emph{arbitrary} cone $\cK$} with $\intr(\cK)\neq\emptyset$. %which is provided in the following lemma.
%and the proof can be found in Lemma 4.3 of \cite{aybat2016primal}.
\begin{lemma}\label{dual-bound}
Let $\bar{\bxi}$ be a Slater point for \eqref{eq:generic problem}, i.e., $\bar{\bxi}\in\relint(\dom \varphi)$ such that $\nv{-}g(\bar{\bxi})\in\intr(\cK)$. Then for all $\bar{y}\in \dom q$, the superlevel set $Q_{\bar{y}}$ is bounded as follows,
\vspace*{-2mm}%
{\small
\begin{equation}
\label{eq:dual-bound-radius}
\norm{y}\leq (\varphi(\bar{\bxi})-q(\bar{y})) / r^*,\quad \forall y \in Q_{\bar{y}},
\end{equation}}%
where $0<r^*\triangleq \min_w\{\sa{-}w^\top g(\bar{\bxi}):\ \|w\|= 1,\ w\in \cK^*\}$.
%\label{dual-bound-problem}
Note that this is not a convex problem due to the equality constraint; instead, one can upper bound \eqref{eq:dual-bound-radius} using $0<\tilde{r}\leq r^*$, which can be efficiently computed by solving a convex problem
\vspace*{-2mm}%
{\small
\begin{equation}\label{dual-bound-problem-tight}
\tilde{r}\triangleq\min_w\{\nv{-}w^\top g(\bar{\bxi}):\ \|w\|_1= 1,\ w\in \cK^*\}.
\end{equation}}%
\end{lemma}
\begin{proof}
For any $y\in Q_{\bar{y}}\subset\cK^\circ$, we have that
%\todo{I have updated the proof accordingly -- formulation has been changed}
{\small
\begin{equation}
q(\bar{y})\leq q(y)=\inf_{\bxi} \{\varphi(\bxi)-y^\top g(\bxi)\} \leq \varphi(\bar{\bxi})-y^\top g(\bar{\bxi}) \label{eq3},
\end{equation}}%
which implies that $y^\top g(\bar{\bxi})\leq \varphi(\bar{\bxi})-q(\bar{y})$.
Since $-g(\bar{\bxi})\in\intr(\cK)$ and $y\in\cK^\circ$, we clearly have $y^\top g(\bar{\bxi})>0$ whenever $y\neq\zero$. Indeed, since $-g(\bar{\bxi})\in\intr(\cK)$, there exist $r>0$ such that $-g(\bar{\bxi})+ru\in \cK$ for all $\|u\|\leq 1$. Hence, for $y\neq\zero$, by choosing $u=y/\|y\|$ and using the fact that $y\in\cK^\circ$, we get that $0\geq(-g(\bar{\bxi})+r y/\|y\|)^\top y$. Therefore, \eqref{eq3} implies that for all $y\in Q_{\bar{y}}$,
%\vspace*{-2mm}%
%{\small
%\begin{equation}
{$r\|y\| \leq y^\top g(\bar{\bxi})\leq \varphi(\bar{\bxi})-q(\bar{y})$; hence, $\|y\| \leq {\varphi(\bar{\bxi})-q(\bar{y}) \over r}$.}
%\end{equation}}%
Now, we will characterize the largest radius $r^*>0$ such that $\cB(-g(\bar{\bxi}),r^*)\subset\cK$, where $\cB(-g(\bar{\bxi}),r)\triangleq\{-g(\bar{\bxi})+r u:\ \norm{u}\leq 1\}$. Note that $r^*>0$ can be written explicitly as follows:
%optimal value of a maximization problem:
$r^*=\max\{r:\ d_\cK\big(-g(\bar{\bxi})+ru\big)\leq 0,\quad \forall u~\st\ \|u\| \leq 1\}$.
Let $\gamma(r)\triangleq\sup\{d_\cK\big(-g(\bar{\bxi})+ru\big):\ \norm{u}\leq 1\}$; hence, $r^*=\max\{r:\ \gamma(r)\leq 0\}$. Note that for any fixed $u\in\reals^m$, $d_\cK\big(-g(\bar{\bxi})+ru\big)$ as a function of $r$ is a composition of a convex function \eyy{$d_\cK(\cdot)$} with an affine function in $r$; hence, it is convex in $r\in\reals$ for all $u\in\reals^m$. Moreover, since supremum of convex functions is also convex, $\gamma(r)$ is convex in $r$. From the definition of $d_{\cK}(\cdot)$, we have
{\small
\begin{equation}
\gamma(r)=\sup_{\|u\|\leq 1}\inf_{\bxi\in \cK}\|\bxi+g(\bar{\bxi})-ru\|
=\sup_{\|u\|\leq 1}\inf_{\bxi\in\cK} \sup_{\|w\|\leq 1}w^\top \big(\bxi+g(\bar{\bxi})-ru\big) \label{eq5}.
\end{equation}}%
Since $\{w\in\reals^m:\ \|w\|\leq 1\}$ is a compact set, and the function in \eqref{eq5} is a bilinear function of $w$ and $\bxi$ for each $u$, the inner $\inf_{\bxi}$ and $\sup_w$ can be interchanged to obtain,\vspace{-2mm}%
{\small
\begin{align*}
\gamma(r)=\sup_{\|u\|\leq 1}\sup_{\|w\|\leq 1} \inf_{\bxi\in \cK} w^\top \Big(\bxi+g(\bar{\bxi})-ru\Big) = \sup\limits_{\substack{\|u\|\leq 1 \\ \|w\| \leq 1\\ w\in \cK^*}} w^\top \big(g(\bar{\bxi})-ru\big) = \sup\limits_{\substack{\|w\| \leq 1\\ w\in \cK^*}} w^\top g(\bar{\bxi})+r\|w\|. %\label{gamma-radius-problem}
\end{align*}}%
Let $w^*(r)$ be %an $\argmax$ of \eyh{rhs of previous equality}.
one of the maximizers. It is easy to see that $\norm{w^*(r)}=1$, since the supremum of a convex function over a convex set is attained on the boundary of the set. Therefore,
$
\gamma(r)=\sup\limits_{\substack{\|w\|= 1\\ w\in \cK^*}} w^\top g(\bar{\bxi})+r.
$
Since $r^*=\max\{r:\ \gamma(r)\leq 0\}$, %one can conclude that,
\vspace*{-2mm}%
{\small
\begin{align*}
(P_1):\qquad r^*&=\max\Big\{r:\ r\leq -\sup\{ w^\top g(\bar{\bxi}):\ \|w\|= 1,\ w\in \cK^*\}\Big\}
=\min\limits_{\substack{\|w\|= 1\\ w\in \cK^*}} -w^\top g(\bar{\bxi}).
\end{align*}}%
Note that $(P_1)$ is not a convex problem due to boundary constraint, $\|w\|= 1$. Next, we define a related convex problem:
%\todo{Erfan: Rephrased}
{$\min\limits_{\substack{\|w\|_1= 1\\ w\in \cK^*}} -w^\top g(\bar{\bxi})
\leq r^*= \min\limits_{\substack{\|w\|= 1\\ w\in \cK^*}} -w^\top g(\bar{\bxi})$,} to lowerbound $r^*$ so that we can upper bound the right hand side of \eqref{eq:dual-bound-radius}.
%{\small
%\begin{equation*}
%(P_2):\qquad \min\limits_{\substack{\|w\|_1= 1\\ w\in \cK^*}} -w^\top g(\bar{\bxi})
%\leq r^*= \min\limits_{\substack{\|w\|= 1\\ w\in \cK^*}} -w^\top g(\bar{\bxi}). %
%\end{equation*}}%
Let $w^*$ be an optimal solution to $(P_1)$ and define $\bar{w}=w^*/\|w^*\|_1$. Clearly, $\|\bar{w}\|_1=1$ and $\bar{w}\in \cK^*$. Moreover, since $\|w^*\|_1\geq \|w^*\|=1$ we have that
{\small
\begin{align*}
0<\tilde{r}=\min\limits_{\substack{\|w\|_1= 1\\ w\in \cK^*}} -w^\top g(\bar{\bxi}) \leq-\bar{w}^\top g(\bar{\bxi})&=-\frac{1}{\|w^*\|_1} {w^{*}}^\top g(\bar{\bxi}) \leq -{w^{*}}^\top g(\bar{\bxi})
%=\min\limits_{\substack{\|w\|= 1\\ w\in \cK^*}} w^\top g(\bar{\bxi})
=r^*.
\end{align*}}%
\end{proof}
%\begin{remark}
%Consider the problem in \eqref{eq:generic problem}. If we further assume that $\varphi$ is convex, \nv{$g$} is $\cK$-convex, and $\cK$ is a closed convex cone, and $\varphi^*>-\infty$, i.e., \eqref{eq:generic problem} is a convex problem with a finite optimal value, then it is known that the Slater condition in Lemma~\ref{dual-bound} is sufficient for the existence of a dual optimal solution, $y^*\in\cK^\circ$, and for zero duality gap. Hence, the dual optimal solution set $Q^*\triangleq\{y\in\cK^\circ:\ q(y)=\varphi^*\}$ can be bounded using Lemma~\ref{dual-bound}. In particular, $\norm{y}\leq\big(\varphi(\bar{\bxi})-\varphi^*\big)/r^*$ for all $y\in Q^*$.
%\end{remark}
\vspace*{-2mm}
\begin{remark} %\todo{Erfan: Removed a remark and add a new one}
\nv{Consider the problem in \eqref{eq:central_problem}. Given a Slater point $\bar{\bxi}$, one needs to solve the minimization problem \eqref{dual-bound-problem-tight} in a distributed fashion, e.g., using the method in \cite{aybat2016primal}, to obtain a dual bound $B\in(0,+\infty)$. Suppose $\varphi_i(\cdot)\geq \underline{\varphi}$ for all $i\in\cN$ and $N$ is known by all agents. Once $\tilde{r}$, the optimal value to~\cref{dual-bound-problem-tight} is computed, one can set $B=(\varphi(\bar{\bxi})-N\underline{\varphi})/\tilde{r}$, i.e., $\bar{y}=\mathbf{0}$. Moreover, if a Slater point exists but not available, one can solve the problem of $\bar{\bxi}=\argmin_{\bxi}~\cF(\sum_{i\in\cN}g_i(\xi_i))$ in a distributed fashion using methods proposed in \cite{chang2014distributed} to obtain a Slater point where $\cF:\reals^m\to\reals$ is a generalized logarithm function for the proper cone $\cK$ (see Section 11.6.1 in~\cite{boyd2004convex} for the definition). Next, $B$ can be computed as discussed previously.}
\end{remark}
\begin{remark}
Let $g_j:\reals^n\rightarrow\reals$ be the components of $g:\reals^n\rightarrow\reals^m$ for $j=1,\ldots,m$, i.e., $g(\bxi)=[g_j(\bxi)]_{j=1}^m$. When $\cK=\reals^m_+$, Lemma~1.1 in~\cite{Nedic08_1J} implies that for any $\bar{y}\in\dom q$ and $\bar{\bxi}$ such that %$g(\bar{\bxi})\in\intr(\cK)$, i.e.,
$g_j(\bar{\bxi})<0$ for all $j=1,\ldots,m$, %the superlevel set
every $y\in Q_{\bar{y}}$ %can be bounded as follows
satisfies
$\norm{y}\leq \big(\varphi(\bar{\bxi})-q(\bar{y})\big)/\bar{r}$, %for all $y \in Q_{\bar{y}}$,
where $\bar{r}\triangleq\min\{-g_j(\bar{\bxi}):\ j=1,\ldots,m\}$. Note our result in Lemma~\ref{dual-bound} gives the same bound since $r^*=\min_w\{-w^\top g(\bar{\bxi}):\ \|w\|= 1,\ w\in\reals^m_+\}=\bar{r}$.
\end{remark}

\section{Numerical Experiments}
\label{sec:numerics}
\sa{%In this section,
We implemented %our proposed algorithm
the DPDA-D algorithm and tested its performance on two different set of problems.} %the performance with other existing methods.}
\subsection{ Basis Pursuit Denoising (BPD) Problem}
Let %$n\gg 1$, and
$\bxi^*\in\reals^{n}$ be an unknown sparse vector, i.e., most of its elements are zero. Suppose $r\in\reals^{m}$ denotes a vector of $m\ll n$ noisy linear measurements of $\bxi^*$ using the measurement matrix $R\in\reals^{m\times {n}}$, i.e., $\norm{R\bxi^*-r}\leq\epsilon$ for some $\epsilon\geq 0$. The BPD problem can be formulated as
%\vspace*{-1mm}
{\small
\begin{align}\label{basis-pursuit}
\min_{\substack{\bxi}}~\norm{\bxi}_1 \quad \st \quad \norm{R\bxi-r}\leq\epsilon.
\end{align}
}%
BPD appears in the context of {compressed sensing~\cite{donoho2006compressed}} and the objective is to recover the unknown sparse $\bxi^*$ from a small set of measurement or transform values in $r$.

Given a set of computing nodes $\cN$, suppose each node $i\in\cN$ knows $r\in\reals^m$ and stores only $n_i$ columns of $R$ corresponding to a submatrix $R_i\in\reals^{m\times n_i}$  such that $n=\sum_{i\in\cN}n_i$ and $R=[R_i]_{i\in\cN}$. Partitioning the decision vector $\bxi=[\xi_i]_{i\in\cN}$ accordingly, {BPD problem} in~\eqref{basis-pursuit} can be rewritten as follows:
\vspace*{-1mm}
{\small
\begin{align}\label{basis-pursuit-dist}
\min_{\xi_i\in\reals^{n_i},~i\in\cN}~\sum_{i\in\cN}\norm{\xi_i}_1 \quad \st \quad \big\|\sum_{i\in\cN}R_i\xi_i-r\big\|\leq\epsilon.
\end{align}
}%
Note \eqref{basis-pursuit-dist} can be cast into the form similar to \eqref{eq:central_problem}. Indeed, let $\chi:\reals\rightarrow\reals\cup\{+\infty\}$ such that $\chi(t)=0$ if $t=\epsilon$, and $+\infty$ otherwise; and let $\cK$ denote the second-order cone, i.e., $\cK=\{(y,t)\in\reals^m\times\reals:\ \norm{y}\leq t\}$. Hence, \eqref{basis-pursuit-dist} can be %equivalently
written as
\vspace*{-2mm}
{\small
\begin{align*}
\min_{t\in\reals, \xi_i\in\reals^{n_i},~i\in\cN}~\sum_{i\in\cN}\norm{\xi_i}_1+\chi(t) \quad \st \quad \big(\sum_{i\in\cN}R_i\xi_i-r, t\big)\in\cK.
\vspace*{-3mm}
\end{align*}
}%

%This section is dedicated to illustrate the performance of DPDA-D to solve BPD problem in a decentralized manner.
First, we test the effect of network topology on the performance of the proposed algorithm, and then to benchmark this distributed algorithm, we also solve the same problem in a centralized way using Prox-JADMM algorithm proposed in~\cite{deng2013parallel}. Note that Prox-JADMM solves the problem in a centralized fashion which naturally {has a} faster convergence than a decentralized algorithm. The aim of this comparison is to show that the convergence of {the} proposed decentralized algorithm {would be competitive} with a \emph{centralized} method %{corresponding to Prox-JADMM}
when nature of the problem {requires to store and access the data in a decentralized manner.} %{In the online technical report \cite{aybat2016distributed}, we also examined the performance of DPDA-S algorithm~\cite{aybat16_dual_static} and benchmark it against Prox-JADMM as well.}
\subsubsection{Problem generation}
\label{sec:generation}
In the rest, we consider two different forms {of the problem in} \eqref{basis-pursuit}: {noisy, i.e., $\epsilon>0$ and noise free, i.e., $\epsilon=0$}. {In our experiments,} we set $n=120$ and $m=20$. For {the noisy case}, as suggested in \cite{aybat2012first}, the target signal $\bxi^*$ is generated {by choosing {$\kappa=20$ of its elements,} uniformly at random, drawn from the standard Gaussian distribution and the rest of the elements are set to $0$}. Moreover, each element of {$R=[R_{ij}]$ is i.i.d with} standard normal distribution, and the measurement $r=R\bxi^*+\eta$ where {$\eta\in\reals^m$ such that each of its elements is i.i.d. according to Gaussian distribution} with mean 0 and variance $\sigma^2=\kappa~10^{-S/10}$ -- this would generate a measurement vector $r$ with the signal-to-noise ratio~(SNR) equal to $S$ where $\hbox{SNR}(r)\triangleq 10 \log_{10}(\mathbb{E}[\norm{R\bxi^*}^2]/\mathbb{E}[\norm{\eta}^2])$. In our experiments, we consider $S=30\mathrm{dB}$ or $40\mathrm{dB}$. Finally, $\epsilon>0$ is chosen such that $\text{Pr}(\norm{\eta}^2\leq \epsilon^2)=1-\alpha$, and we let $\alpha=0.05$. {For the noise-free case, the noise parameters, i.e., $\sigma^2$ and $\epsilon$ are set to 0; hence, the constraint for the noise-free case is a linear one, i.e., $\sum_{i\in\cN}R_i\xi_i=r$ -- the rest of the problem components are generated as in the noisy case.}

%Note that static directed network case is done similar to time-varying directed network by fixing the network in all iterations.
%Consider a random connected graph $G=(\mathcal{N},\mathcal{E})$ and $N\triangleq |\mathcal{N}|$.
{
{\bf Generating an undirected small-world network:} Let $\cG_u=(\cN,\cE_u)$ be generated as a random small-world network. Given $|\cN|$ and the desired number of edges $|\cE_u|$, we choose $|\cN|$ edges creating a random cycle over nodes, and then the remaining $|\cE_u|-|\cN|$ edges are selected uniformly at random.}

{
{\bf Generating a time-varying undirected network:} %Given $|\cN|$ and the desired number of edges $|\cE_*|$ for the initial graph,
We first generate a random small-world $\cG_u=(\cN,\cE_u)$ as described above. Next, given $M\in\integers_+$, and $p\in(0,1)$, for each $k\in\integers_+$, we generate $\cG^t=(\mathcal{N},\mathcal{E}^t)$, the communication network at time $t\in\{(k-1)M,\ldots,kM-2\}$ by sampling $\lceil p~|\cE_u|\rceil$ edges of $\cG_u$ uniformly at random and we set $\mathcal{E}^{kM-1}=\cE_u\setminus \bigcup_{t=(k-1)M}^{kM-2}\cE^t$. In all experiments, we set $M=5$, $p=0.8$ and the number of communications per iteration is set to $q_k=10\log(k+1)$.}
\vspace*{-1mm}
\subsubsection{Effect of Network Topology}
\label{sec:network-effect}
In this section, we test the effect of network topology on the performance of DPDA-S and
DPDA-D on \emph{undirected} communication networks. We consider four scenarios in which the number of nodes $N\in\{10,~40\}$ and the average number of edges per node, $|\cE^t|/N$, is either $1.2$ or $\approx 3.6$. For each scenario, we plot relative suboptimality, i.e., $\big|\varphi(\bxi^k)-\varphi(\bxi^*)\big|/|\varphi(\bxi^*)|$, infeasibility, i.e., $\big(\norm{\sum_{i\in\cN}R_i\xi_i^k-r}-\epsilon\big)_+$, and consensus violation, i.e., $\max_{i\in\cN}\big\|y^k_i-\frac{1}{|\cN|}\sum_{j\in\cN}y_j^k\big\|$ versus iteration number $k$. All the plots show the average statistics over 50 randomly generated replications. In each of these independent replications, both $R$ and $\bxi^*$ are also randomly generated in addition to random communication networks.
%\sa{Erfan: in these randomly generated replications, do you generate $R$ and $\bxi^*$ as well, or do you only generate the networks but fix $\bxi^*$?} \ey{Yes, in addition to networks, both $R$ and $\bxi^*$ are randomly generated at every replication.}

{\bf Testing DPDA-S on static undirected communication networks:} %For the static network setting,
We generated the static small-world networks $\cG=(\cN,\cE)$, as described above in Section~\ref{sec:generation}, for $(|\cN|,|\cE|)\in\{(10,15),$ $(10,45),~(40,60),~(40,180)\}$ and solved the BPD problem in~\eqref{basis-pursuit} using DPDA-S described in Fig.~\ref{alg:DPDA-S} {corresponding to SNR values $S\in\{30, ~+\infty\}$ -- when $S=+\infty~\mathrm{dB}$, we have $R\bxi^*=r$, i.e., noise-free case.} {The elements of the initial point $\bxi^0$} are i.i.d standard uniform distribution; and step-sizes are chosen as follows: $\gamma=2|\cN|/|\cE|$, $\tau_i=\frac{1}{\norm{R_i}}$, and $\kappa_i=\frac{1}{2\gamma d_i+\norm{R_i}}$ for $i\in\cN$. {The performance of DPDA-S in terms of suboptimality, infeasibility and consensus violation is displayed in Fig.~\ref{static-net}.} It is clear that when compared to the effect of average edge density, the network size $|\cN|$ has more influence on the convergence rate, i.e., the smaller the network faster the convergence is. On the other hand, for fixed size network, as expected, higher the density faster the convergence is, especially for consensus violation statistics.
\begin{figure}[h]
\centering
\vspace*{-4mm}
\hspace*{-5mm}
\includegraphics[scale=0.25]{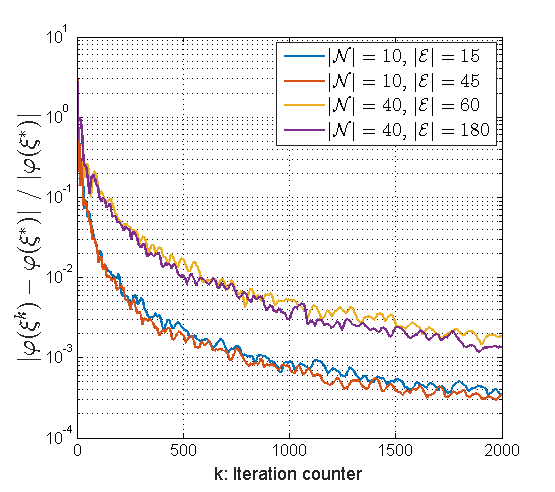}
\hspace*{-5mm}
\includegraphics[scale=0.25]{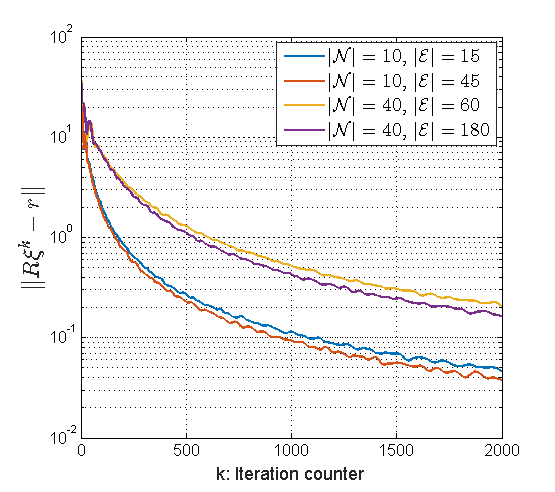}
\hspace*{-5mm}
\includegraphics[scale=0.25]{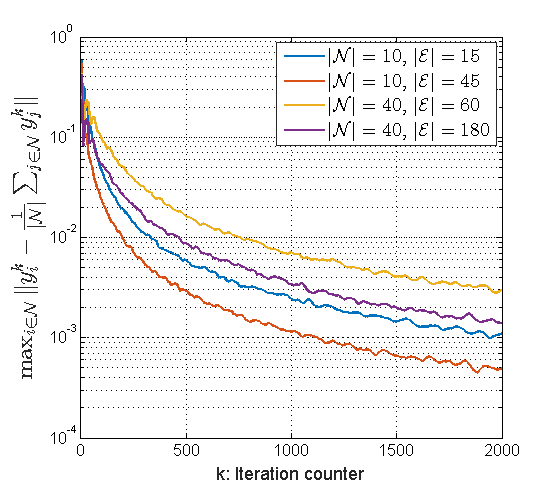}
\hspace*{-5mm}
\includegraphics[scale=0.25]{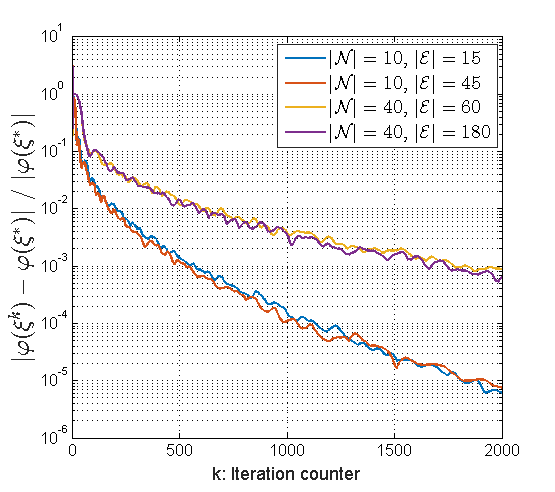}
\hspace*{-5mm}
\includegraphics[scale=0.25]{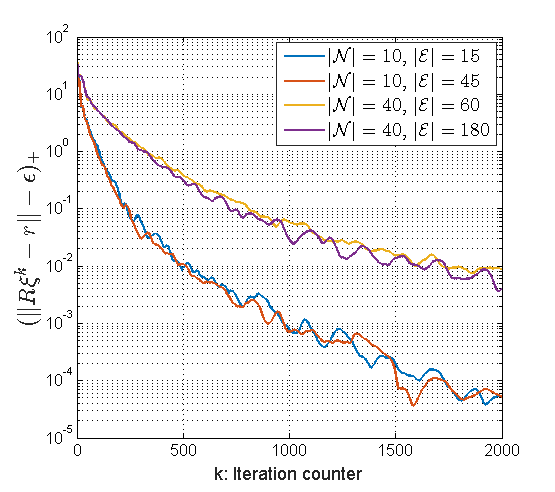}
\hspace*{-5mm}
\includegraphics[scale=0.25]{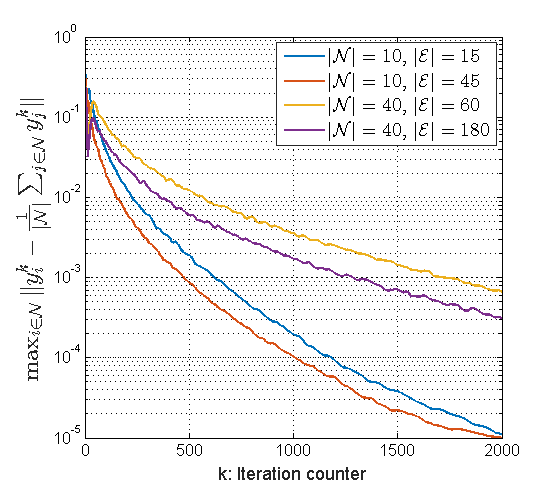}
\vspace*{-3mm}
\caption{Effect of network topology on the convergence rate of DPDA-S: \textbf{top} row corresponds to noise free and \textbf{bottom} row corresponds to noisy experiments with $S=30\mathrm{dB}$.}
\label{static-net}
%\vspace*{-3mm}
\end{figure}

{\bf Testing DPDA-D on time-varying undirected communication networks:} We first generated an undirected small-world network $\cG_u=(\cN,\cE_u)$ as described earlier. %and let $\cG_*=\cG_u$.
Next, we generated $\{\cG^t\}_{t\geq 0}$ as described in Section~\ref{sec:generation}. %by setting $M=5$, $p=0.8$, and $q_k=10\ln(k+1)$. For each consensus round $t\geq 0$, $V^t$ is formed according to Metropolis weights, i.e., for each $i\in\cN$, $V^t_{ij}=1/(\max\{d_i,d_j\}+1)$ if $j\in\cN_i^t$, $V^t_{ii}=1-\sum_{i\in\cN_i^t}V^t_{ij}$, and $V^t_{ij}=0$ otherwise -- see~\eqref{eq:approx-average-dual} for our choice of $\cR^k$.
%for this setting.
%For DPDA-D, displayed in Fig.~\ref{alg:PDDual}, we sample
We chose the initial point $\bxi^0$ of DPDA-D such that the components are i.i.d with the standard uniform distribution, and set the step-sizes as follows: $\gamma=1$, $\tau_i=\frac{1}{|\cN|\norm{R_i}}$, and $\kappa_i=\frac{1}{\gamma+\norm{R_i}/|\cN|}$ for $i\in\cN$. {The performance of DPDA-D in terms of suboptimality, infeasibility and consensus violation is displayed in Fig.~\ref{dynamic-net}.
It is clear that when compared to the effect of average edge density, the network size $|\cN|$ has more influence on the convergence rate, i.e., the smaller the network faster the convergence is; however, the average edge density does not seem to have a significant impact on the convergence.}
%\vspace*{-1mm}
\begin{figure}[h]
\centering
\vspace*{-3mm}
\hspace*{-5mm}
\includegraphics[scale=0.14]{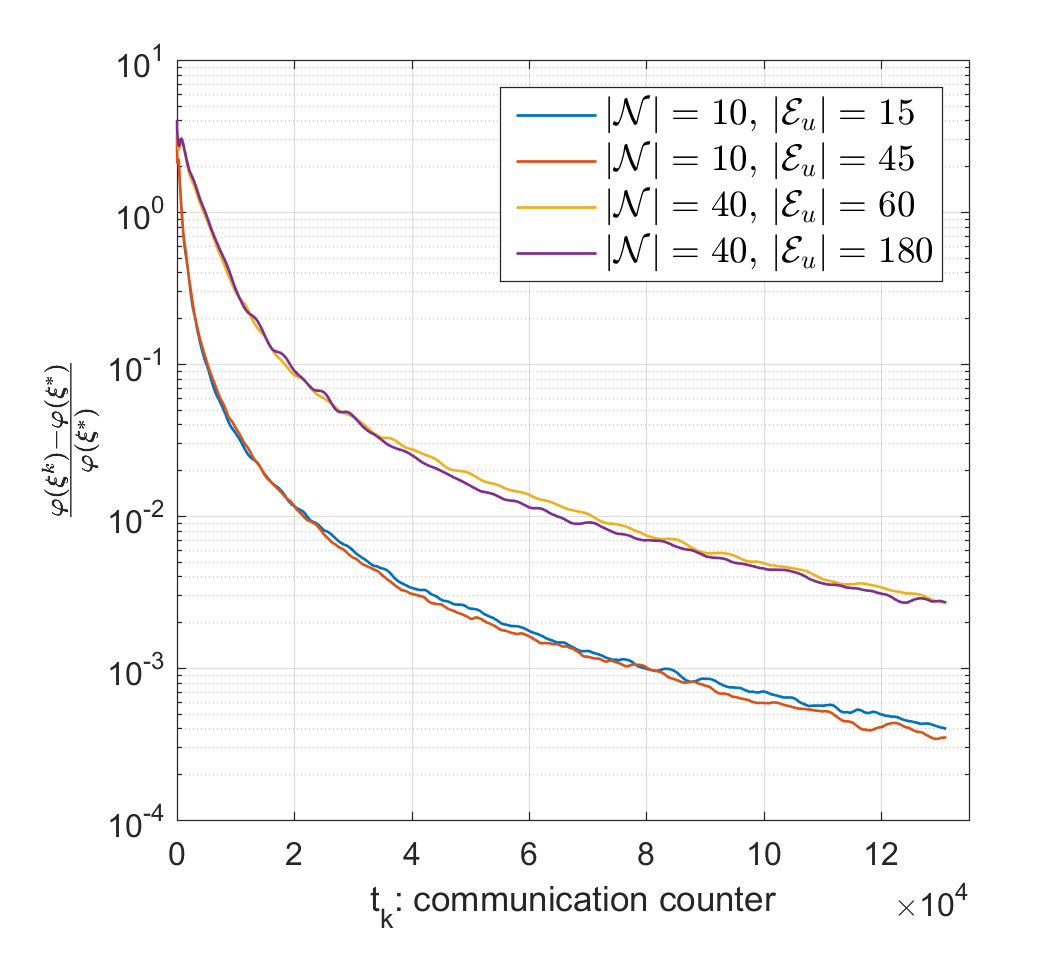}
\hspace*{-5mm}
\includegraphics[scale=0.14]{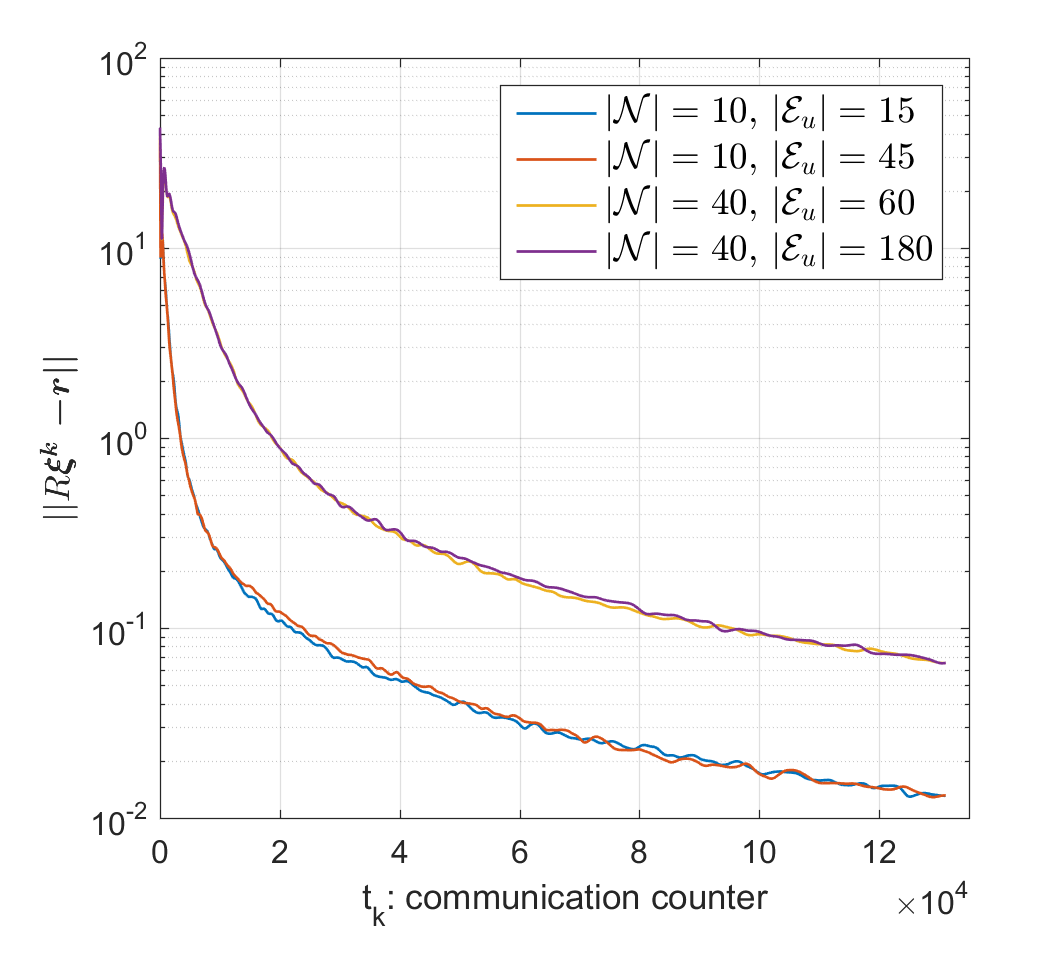}
\hspace*{-5mm}
\includegraphics[scale=0.14]{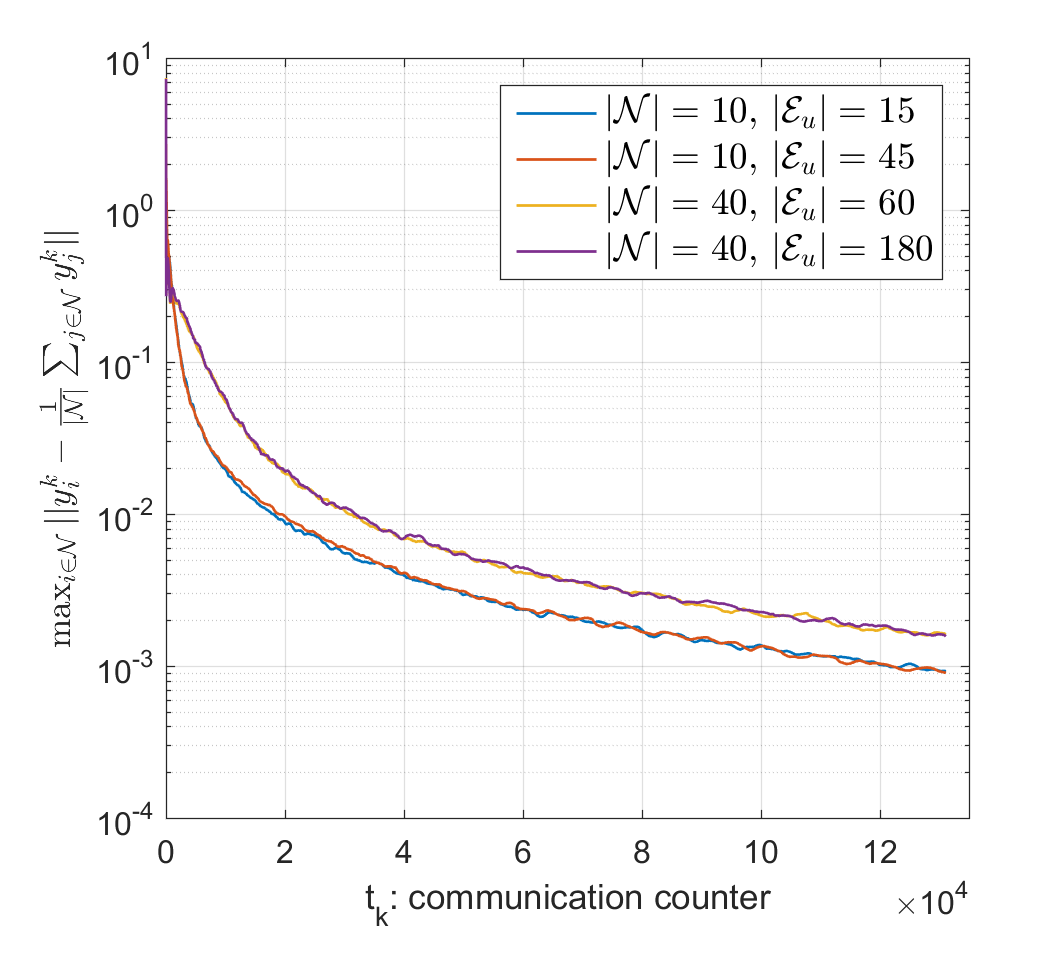}
\hspace*{-5mm}
\includegraphics[scale=0.14]{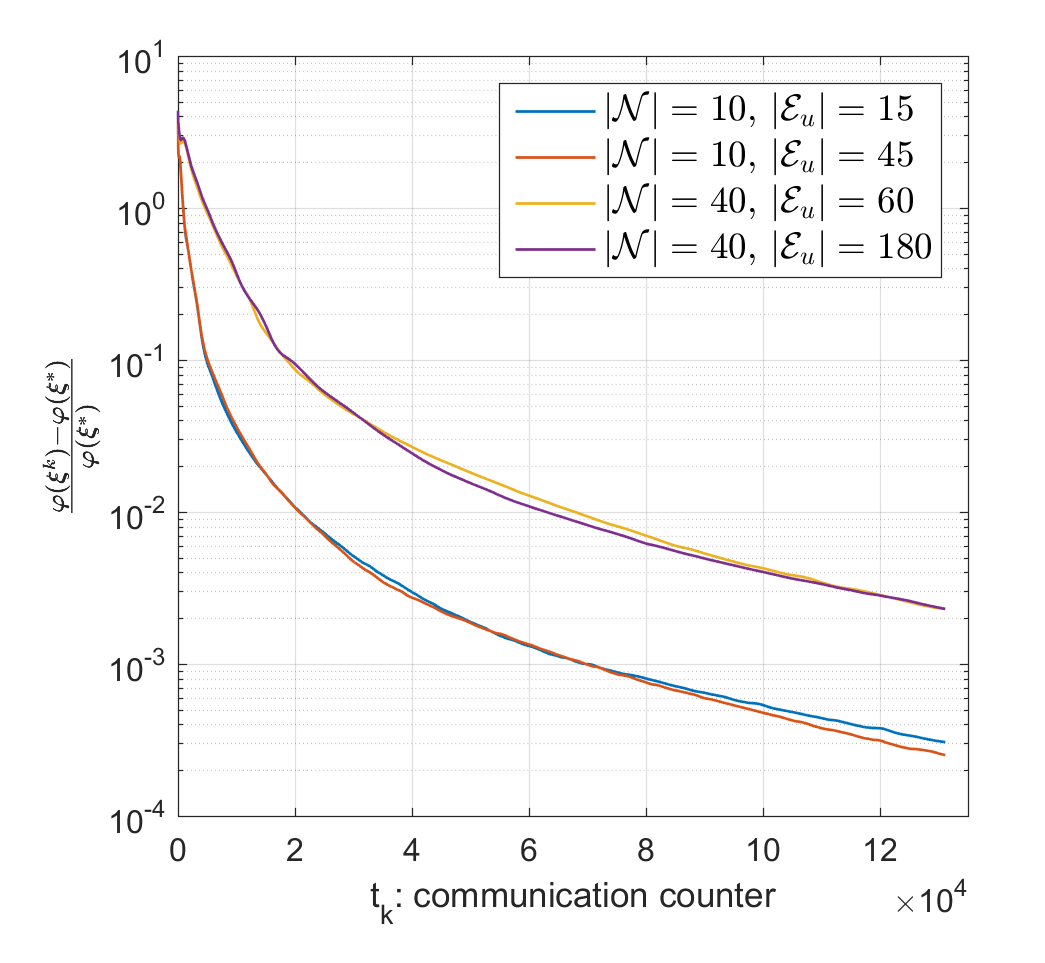}
\hspace*{-5mm}
\includegraphics[scale=0.14]{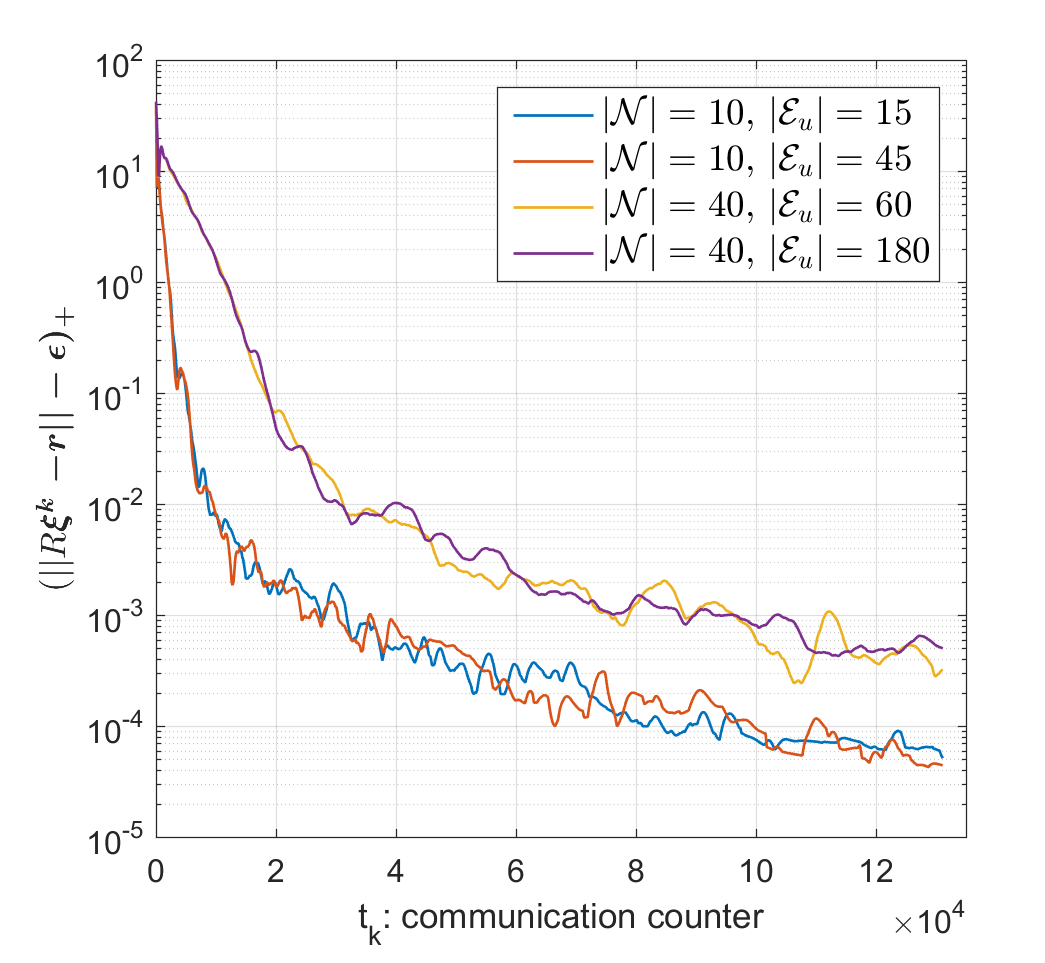}
\hspace*{-5mm}
\includegraphics[scale=0.14]{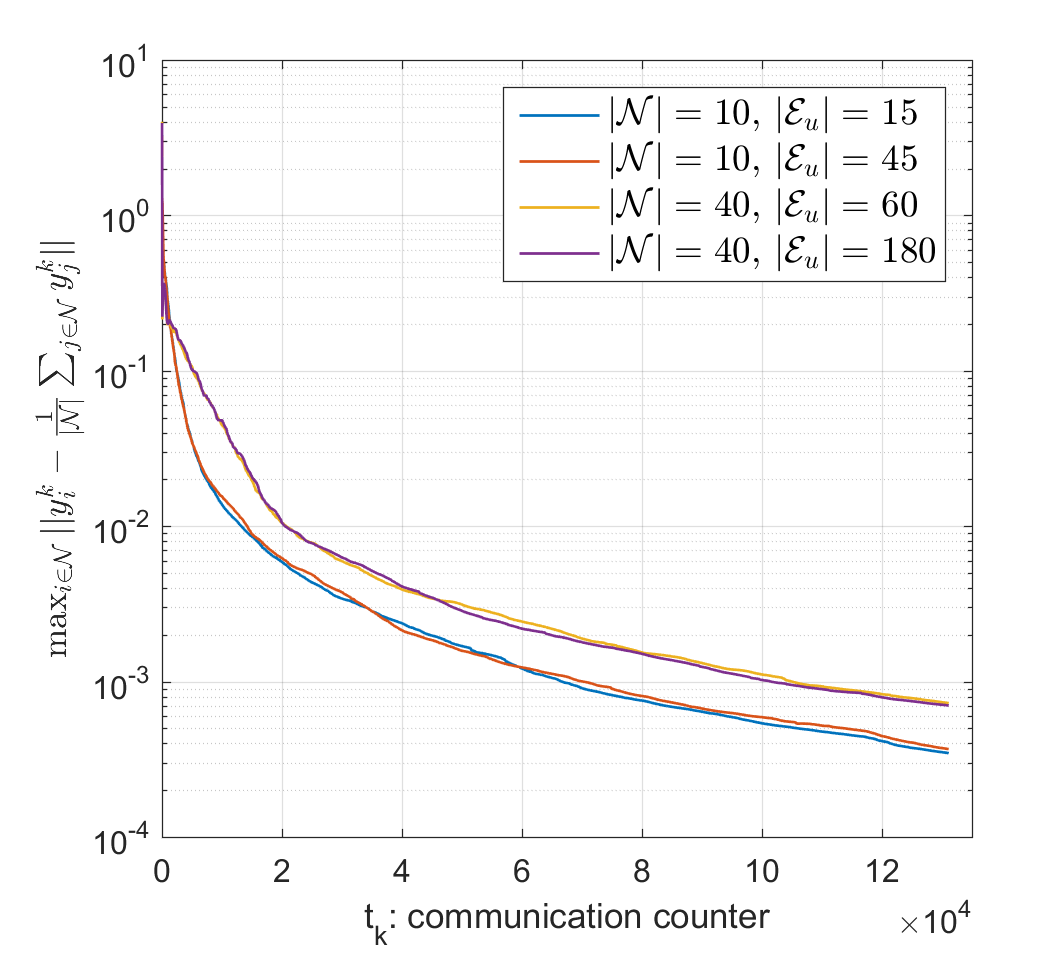}
\vspace*{-3mm}
\caption{Effect of network topology on the convergence of DPDA-D: \textbf{top} row corresponds to noise free and \textbf{bottom} row corresponds to noisy experiments with $S=30\mathrm{dB}$.}
\label{dynamic-net}
\end{figure}
\vspace*{-3mm}
\subsubsection{{Benchmarking against a centralized algorithm}}\label{subsec:benchmark}
In this section, we benchmark DPDA-S on undirected network and DPDA-D on both undirected and directed networks against Prox-JADMM algorithm implemented on BPD problems under three different noise levels; $S=30~\mathrm{dB}$, $S=40~\mathrm{dB}$ and noise free, i.e., $S=+\infty~\mathrm{dB}$. Prox-JADMM is a multi-block ADMM using Jacobian type updates and block-$i$ update has an additional proximal term $\tfrac{1}{2}\norm{\xi_i-\xi_i^k}_{P_i}^2$ for each $i\in\cN$, where $\{P_i\}_{i\in\cN}$ are positive-definite matrices satisfying certain conditions. We choose the parameters for Prox-JADMM algorithm as suggested in Section 3.2. of \cite{deng2013parallel}, i.e., by setting the matrix $P_i$ in the proximal term to be $P_i=(N\id-10~R_i^\top R_i)/\norm{r}_1$ for $i\in\cN$ and $\{P_i\}_{i\in\cN}$ are adaptively updated by the strategy discussed in Section 2.3. of~\cite{deng2013parallel}.

{\bf Static undirected network:} In each replication, we generate a random small-world network $\cG=(\cN,\cE)$ and choose the algorithm parameters as in static network experiments of Section~\ref{sec:network-effect}. The comparison between DPDA-S and Prox-JADMM in terms of suboptimality, infeasibility and consensus violation is displayed in Fig.~\ref{static-compare} for different levels of noise. We observe that the lower signal-to-noise ratio leads to faster convergence, and the noise-free case has the slowest convergence. For all noise levels, DPDA-S is competitive with Prox-JADMM. %-- slightly slower rate of DPDA-S is the price we pay for the distributed setting to reach consensus on the dual price.
\vspace*{-3mm}
\begin{figure}[h]
\centering
\hspace*{-5mm}
\includegraphics[scale=0.18]{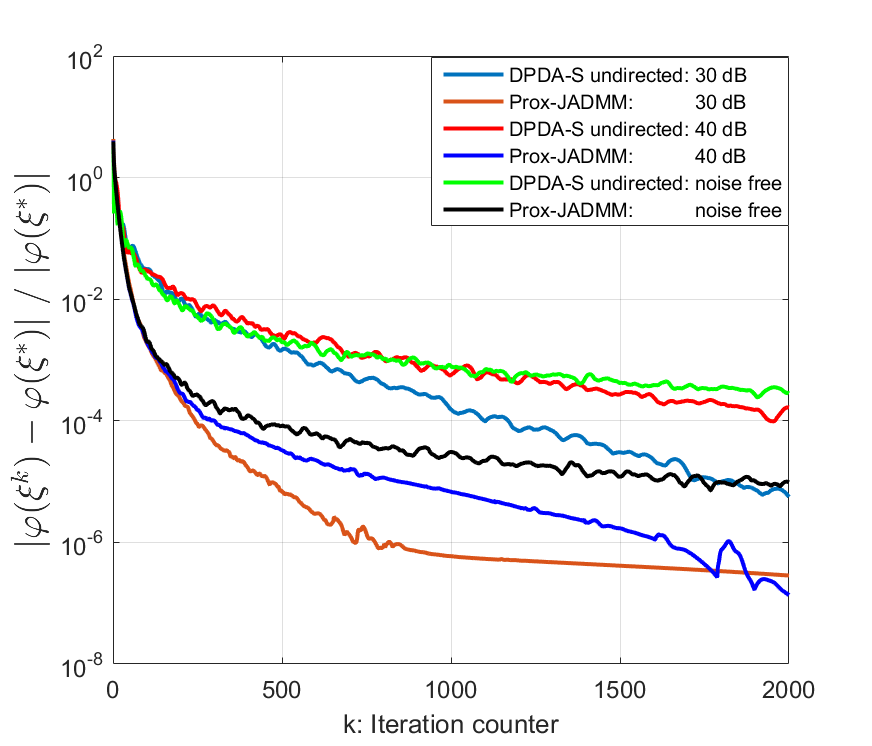}
\hspace*{-5mm}
\includegraphics[scale=0.18]{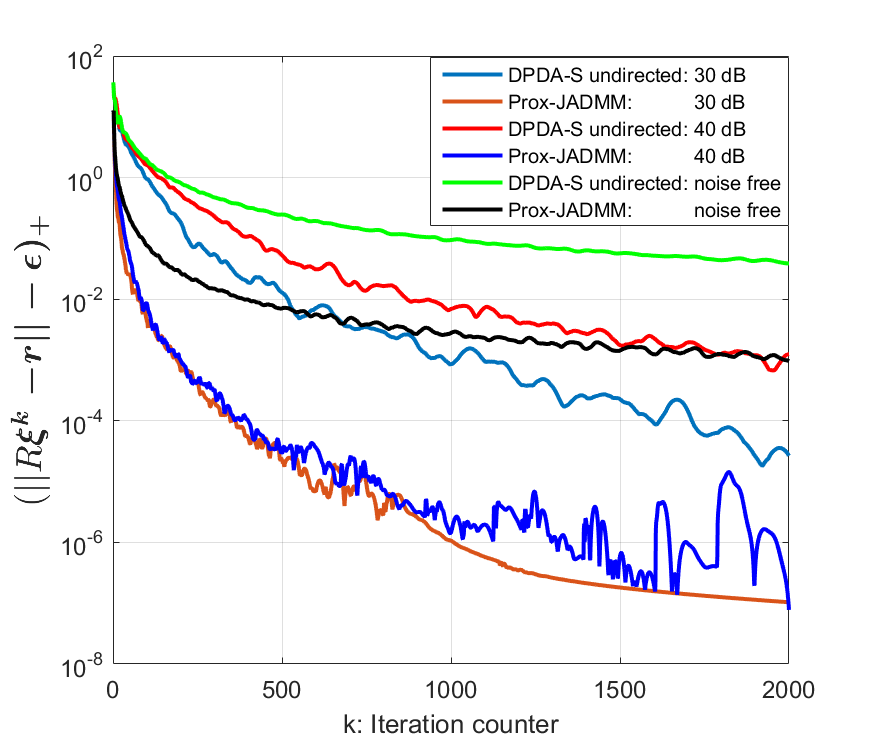}
\hspace*{-5mm}
\includegraphics[scale=0.18]{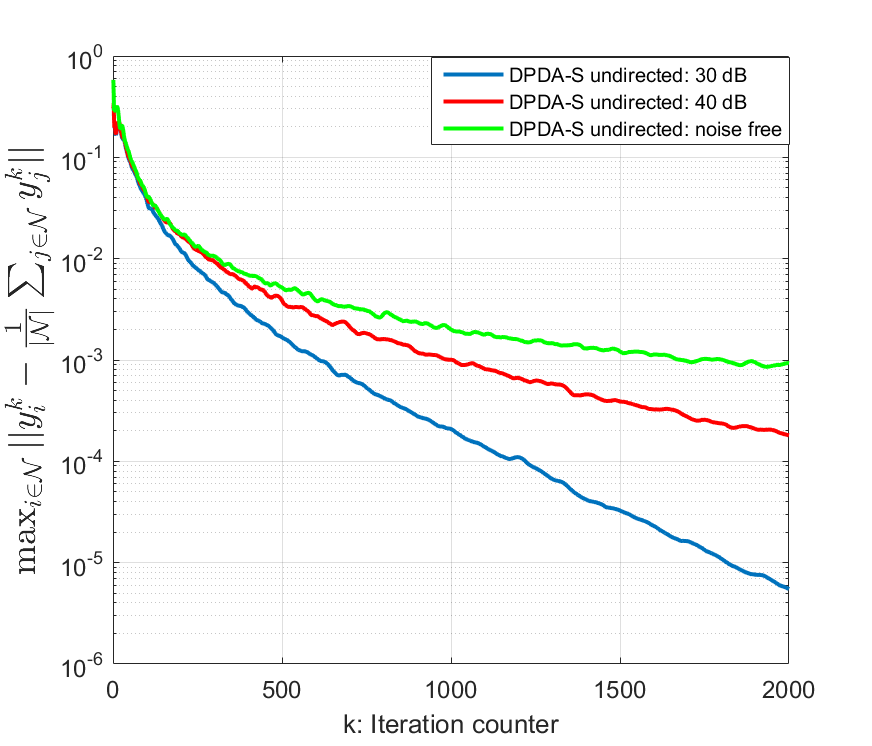}
\caption{Comparison of DPDA-S and Prox-JADMM over undirected static network for three different noise levels}
\label{static-compare}
\vspace*{-3mm}
\end{figure}

{\bf Time-varying undirected network:} For undirected time-varying networks we fix $N=10$ and $|\cE^t|/N=1.2$, i.e., $|\cE_u|/N=1.5$ -- we observe the same convergence behavior for the other network scenarios discussed in Section~\ref{sec:network-effect}. In each replication, we generate the network sequence $\{\cG^t\}_{t\geq 0}$ and choose the parameters as in time-varying network experiments of Section~\ref{sec:network-effect}. Fig.~\ref{undirected-compare} shows the comparison between {the two methods} % DPDA-D and Prox-JADMM
in terms of suboptimality, infeasibility and consensus violation.
%The same convergence pattern as in the static case is observed; however, {the impact of the noise levels on the convergence is less apparent.}
{We observe that different noise-levels lead to similar convergence patterns; however, the lower signal-to-noise ratio leads to faster convergence, and the noise-free case has the slowest convergence. For all noise levels DPDA-D is competitive against Prox-JADMM -- slightly slower rate of DPDA-D is the price we pay for the decentralized setting to reach consensus on the dual price over the time-varying %communication
network.}
\begin{figure}[h]
\centering
\vspace*{-3mm}
\hspace*{-4mm}
\includegraphics[scale=0.18]{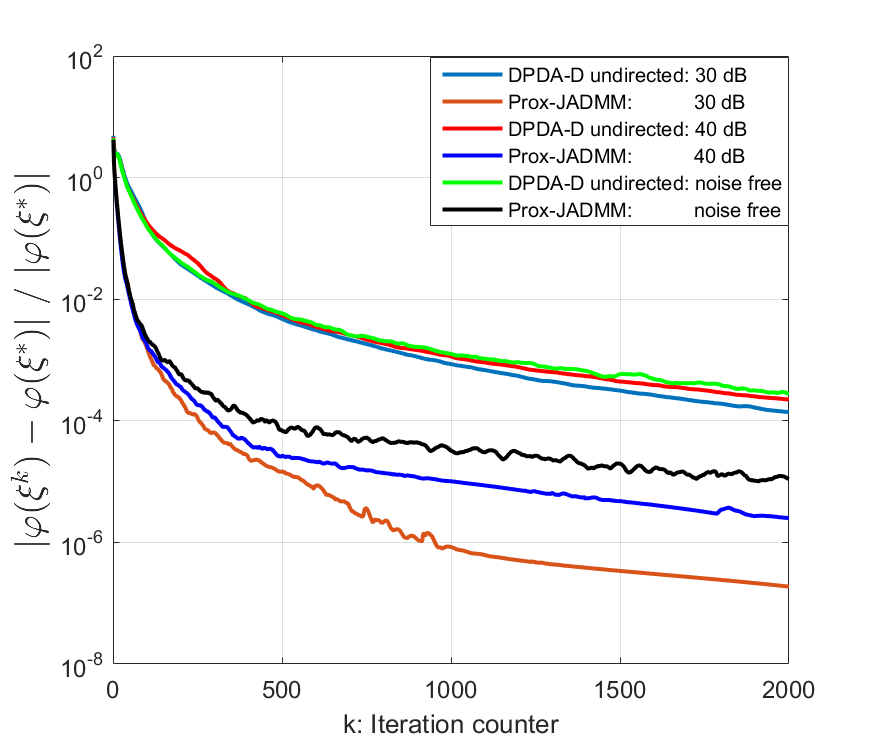}
\hspace*{-4mm}
\includegraphics[scale=0.18]{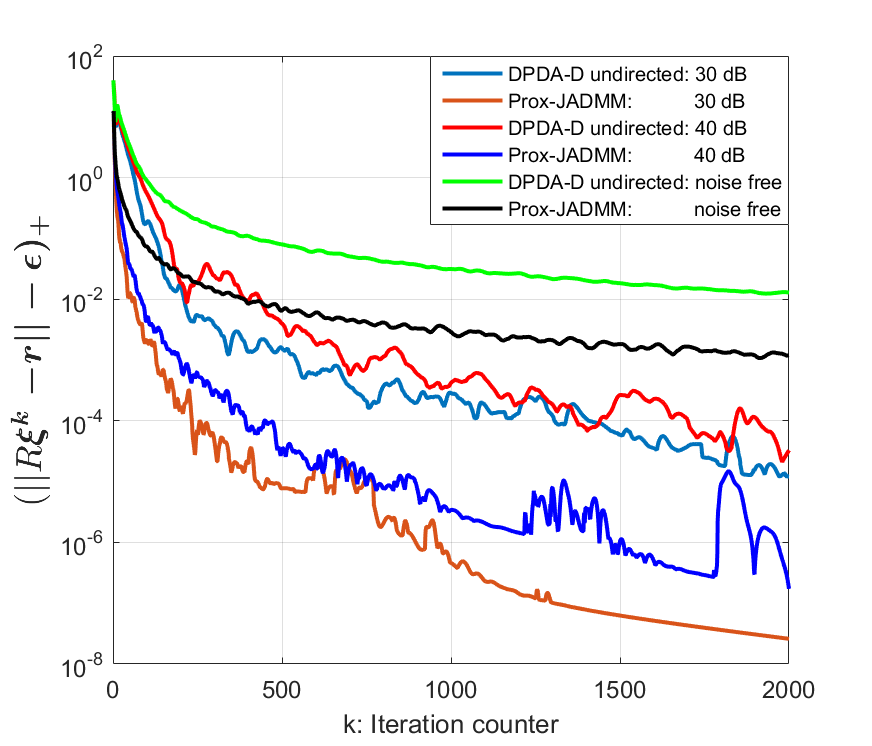}
\hspace*{-4mm}
\includegraphics[scale=0.18]{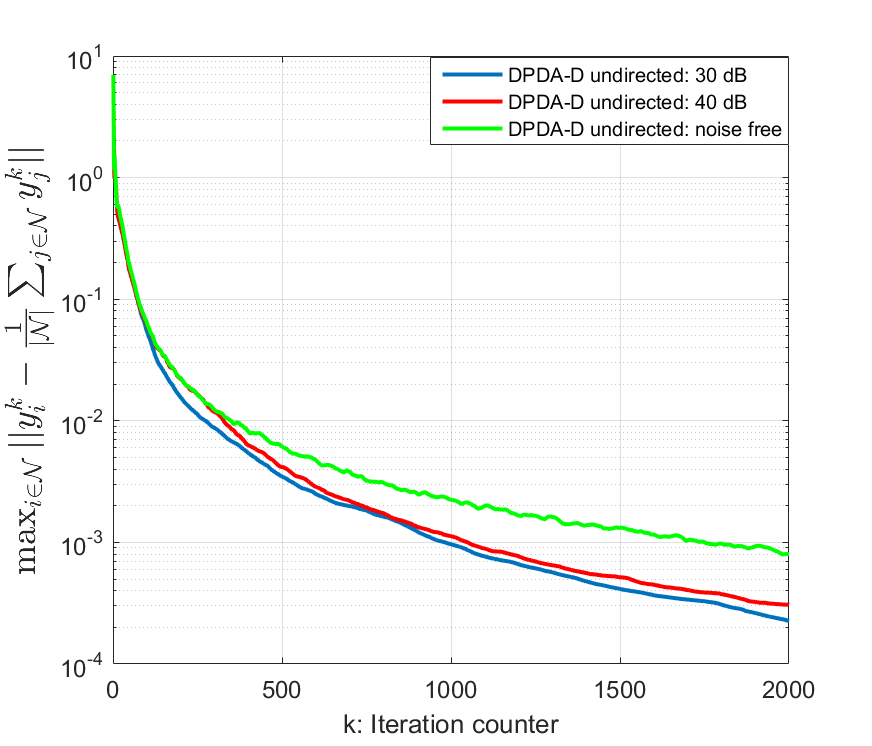}
\vspace*{-3mm}
\caption{Comparison of DPDA-D and Prox-JADMM over undirected time-varying network for three different noise levels}
\label{undirected-compare}
%\vspace*{-5mm}
\end{figure}
\begin{figure}[h]
\centering
\vspace*{-3mm}
\hspace*{-4mm}
\includegraphics[scale=0.18]{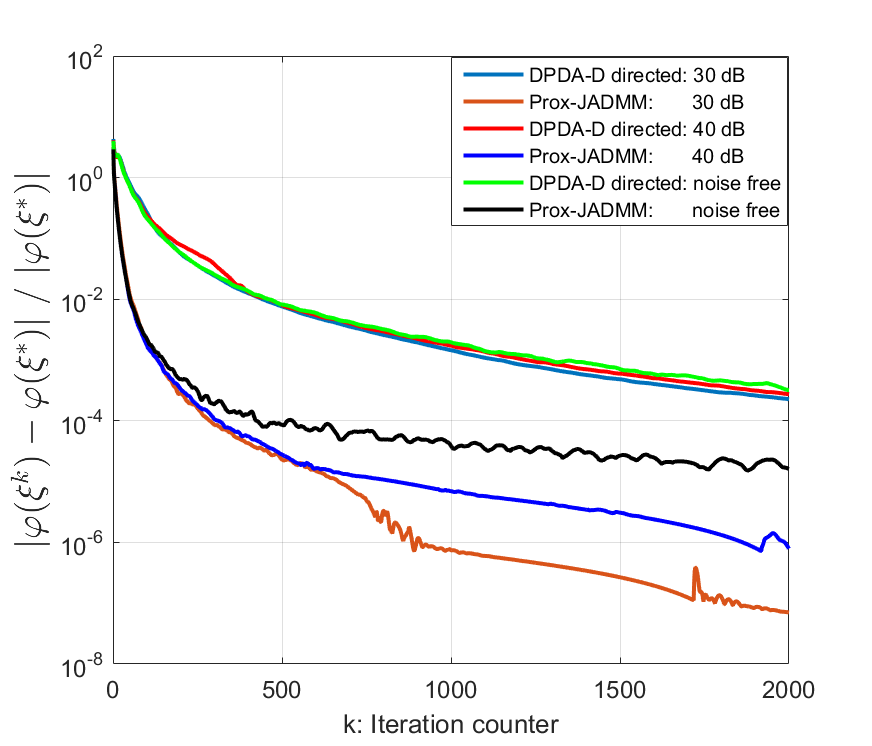}
\hspace*{-4mm}
\includegraphics[scale=0.18]{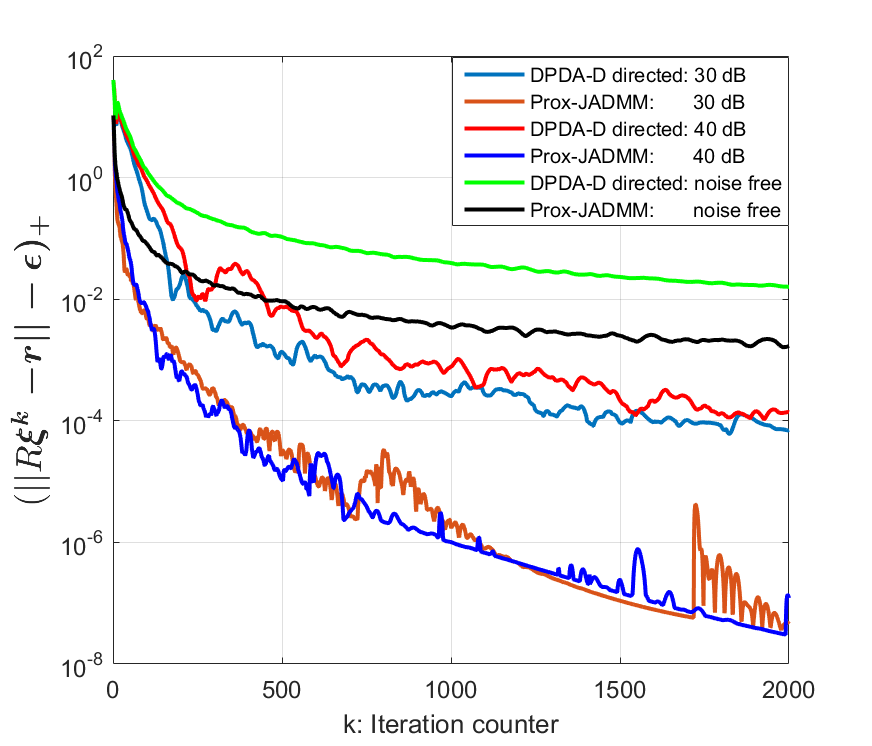}
\hspace*{-4mm}
\includegraphics[scale=0.18]{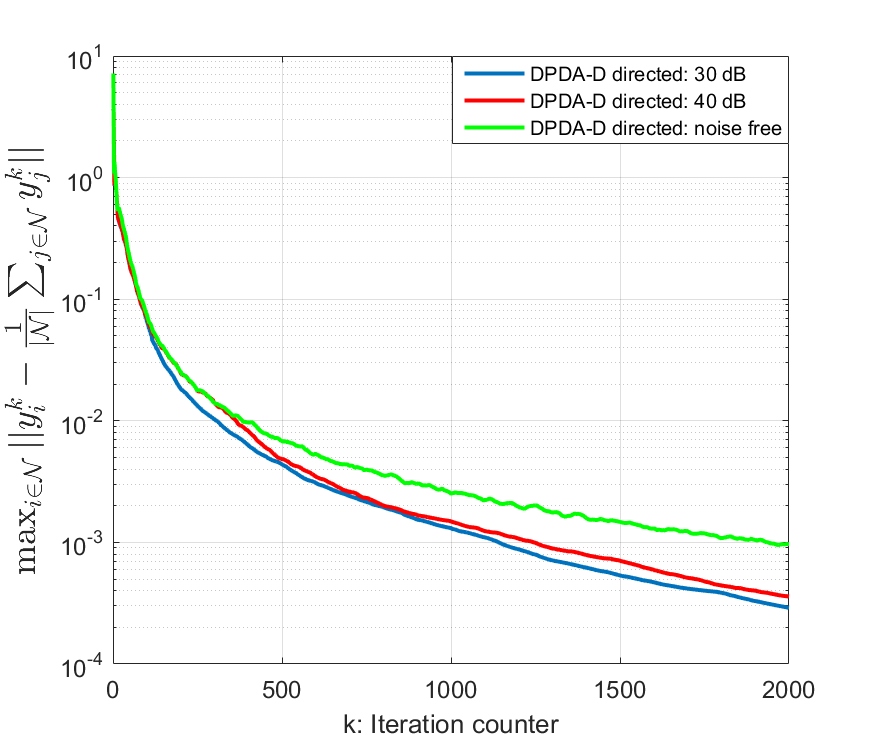}
\vspace*{-3mm}
\caption{Comparison of DPDA-D and Prox-JADMM over directed time-varying network with three noise levels}
\label{directed-compare}
\end{figure}
\vspace*{-2mm}
\begin{figure}[h]
\vspace*{-5mm}
\centering
\begin{minipage}[c]{1\linewidth}
    \begin{center}
      \begin{tikzpicture}
        \coordinate (x1) at (-0.5,0.9);
        \coordinate (x2) at (4.1,1.8);
        \coordinate (x3) at (4.3,0.3);
        \coordinate (x4) at (2,2.2);
        \coordinate (x5) at (3,-0.2);
        \coordinate (x6) at (1.5,-0.2);
        \coordinate (x7) at (3,1.9);
        \coordinate (x8) at (0,0.2);
        \coordinate (x9) at (1.25,1.5);
        \coordinate (x10) at (0.6,1.1);
        \coordinate (x11) at (0.4,1.9);
        \coordinate (x12) at (3.5,1);

        %\node[align=left, below] at (x1) {\small $\Phi_1(\xi_1)$};

        \draw [arrows={- triangle 45}] (x1) -- (x10);
        \node[align=left, above] at (x10) {\small $10$};

        \draw [arrows={- triangle 45}] (x1) -- (x6);
        \node[align=left, below] at (x6) {\small $6$};

        \draw [arrows={- triangle 45}] (x8) -- (x1);
        \node[align=left, left] at (x1) {\small $1$};

        \draw [arrows={- triangle 45}] (x8) -- (x10);
        %\node[align=left, below] at (x10) {\small $\Phi_10(\xi_10)$};

        \draw [arrows={triangle 45 - triangle 45}] (x8) -- (x6);
        \node[align=left, left] at (x8) {\small $8$};

        \draw [arrows={- triangle 45}] (x6) -- (x3);
        \node[align=left, right] at (x3) {\small $3$};

        \draw [arrows={- triangle 45}] (x11) -- (x1);
        %\node[align=left, above] at (x1) {\small $\Phi_1(\xi_1)$};

        \draw [arrows={- triangle 45}] (x9) -- (x11);
        \node[align=left, left] at (x11) {\small $11$};

        \draw [arrows={- triangle 45}] (x9) -- (x3);
        %\node[align=left, above] at (x3) {\small $\Phi_3(\xi_3)$};

         \draw [arrows={- triangle 45}] (x9) -- (x5);
        \node[align=left, left] at (x5) {\small $5$};

        \draw [arrows={- triangle 45}] (x4) -- (x9);
        \node[align=left, right] at (x9) {\small $9$};

        \draw [arrows={- triangle 45}] (x4) -- (x11);
        %\node[align=left, above] at (x11) {\small $\Phi_11(\xi_11)$};

        \draw [arrows={- triangle 45}] (x7) -- (x4);
        \node[align=left, above] at (x4) {\small $4$};

        \draw [arrows={- triangle 45}] (x7) -- (x12);
        \node[align=left, right] at (x12) {\small $12$};

        \draw [arrows={- triangle 45}] (x7) -- (x6);
        %\node[align=left, above] at (x6) {\small $\Phi_6(\xi_6)$};

        \draw [arrows={- triangle 45}] (x2) -- (x10);
        %\node[align=left, above] at (x10) {\small $\Phi_10(\xi_10)$};

        \draw [arrows={- triangle 45}] (x12) -- (x6);
        %\node[align=left, above] at (x6) {\small $\Phi_6(\xi_6)$};

        \draw [arrows={- triangle 45}] (x12) -- (x2);
        \node[align=left, above] at (x2) {\small $2$};

        \draw [arrows={- triangle 45}] (x12) -- (x5);
        %\node[align=left, above] at (x5) {\small $5$};

        \draw [arrows={- triangle 45}] (x3) -- (x12);
        %\node[align=left, above] at (x12) {\small $\Phi_12(\xi_12)$};

        \draw [arrows={- triangle 45}] (x5) -- (x3);
        %\node[align=left, above] at (x3) {\small $\Phi_3(\xi_3)$};

        \draw [arrows={- triangle 45}] (x10) -- (x7);
        \node[align=left, above] at (x7) {\small $7$};

        \draw [arrows={- triangle 45}] (x10) -- (x5);
        %\node[align=left, above] at (x5) {\small $\Phi_5(\xi_5)$};

        \filldraw[fill=red!50!](x1) circle [radius=0.08];
        \filldraw[fill=red!50!] (x2) circle [radius=0.08];
        \filldraw[fill=red!50!] (x3) circle [radius=0.08];
        \filldraw[fill=red!50!] (x4) circle [radius=0.08];
        \filldraw[fill=red!50!] (x5) circle [radius=0.08];
        \filldraw[fill=red!50!](x6) circle [radius=0.08];
        \filldraw[fill=red!50!] (x7) circle [radius=0.08];
        \filldraw[fill=red!50!] (x8) circle [radius=0.08];
        \filldraw[fill=red!50!] (x9) circle [radius=0.08];
        \filldraw[fill=red!50!] (x10) circle [radius=0.08];
        \filldraw[fill=red!50!] (x11) circle [radius=0.08];
        \filldraw[fill=red!50!] (x12) circle [radius=0.08];
      \end{tikzpicture}
    \end{center}
\end{minipage}
\vspace*{-3mm}
 \caption{$\cG_d=(\cN,\cE_d)$ directed strongly connected graph} \label{fig:G1}
\end{figure}

{\bf Time-varying directed network:} %In this scenario, we generated time-varying \emph{directed} communication networks similar to \cite{nedich2016achieving}. In particular, $\cG_d=(\cN,\cE_d)$ is
%\todo{Erfan: Rephrased}
{In this scenario, similar to \cite{nedich2016achieving} we consider the \emph{strongly-connected} directed graph $\cG_d=(\cN,\cE_d)$} in Fig.~\ref{fig:G1} with $N=12$ nodes and $|\cE_d|=24$ directed edges.
%\todo{Erfan: Removed a sentence}% -- $\cG_d$ is also used for the numerical experiments in~\cite{nedich2016achieving}.
%We set $\cG_*=\cG_d$, and
We generated $\{\cG^t\}_{t\geq 0}$ as in the undirected case, but using $\cG_d$ instead of $\cG_u$, with parameters $M=5$, $p=0.8$, and \nv{$q_k=10\log(k+1)$}; hence, $\{\cG^t\}_{t\geq 0}$ is $M$-strongly-connected. Moreover, communication weight matrices $V^t$ are formed according to rule \eqref{eq:directed-weights}, {and we used the approximate averaging operator $\cR^k$ given in \eqref{eq:approx-average-dual-directed}}. We set the step-sizes as in the time-varying undirected case. Fig.~\ref{directed-compare} illustrates the comparison between DPDA-D and Prox-JADMM in terms of suboptimality, infeasibility and consensus violation when the network is both time-varying and directed. The results of this experiment are similar to those for the time-varying undirected case; hence, using unidirectional communications instead of bidirectional did not adversely affect the convergence of DPDA-D.
\subsection{Multi-channel Power Allocation Problem}
\label{sec:nonlinear-num}
\sa{%Here, we consider a problem of minimizing a linear function subject to log-constraints that arises
Multi-channel power allocation is a classic problem in information theory. Suppose there are a set of nodes connected to each other over a time-varying wireless communication network and all transmitting information to a receiver. Let the communication graph be $\cG^t=(\cN,\cE^t)$ at time $t>0$. Each node $i\in\cN$ transmits information over a different channel with bandwidth $b_i$ (given) with a signal transmission power $s_i\in[0,u_i]$ Watts and the signal is exposed to Gaussian White (uncorrelated) noise of additive nature, with power $w_i$ Watts (given). According to Shannon-Hartley equation the maximum capacity of the channel associated with node $i\in\cN$ is $b_i \log_2(1+s_i/w_i)$. Suppose we want to minimize the total power of the system subject to certain capacity requirement $\delta>0$, i.e., $\min\{\sum_{i\in\cN} s_i:\ \sum_{i\in\cN} b_i\log_2(1+s_i/w_i)\geq \delta,\ 0\leq s_i \leq u_i\}$.}

\sa{For numerical experiments, we consider the particular setup described in \cite{mateos2017distributed}:
{\small
\begin{align}\label{shannon1}
\min_{\bxi=[\xi_i]_{i\in\cN}} \sum_{i\in\cN}c_i\xi_i\quad \hbox{s.t.}\quad \sum_{i\in\cN} b_i\log(1+\xi_i)\geq \delta, \ \ \bxi\in [0,1]^{|\cN|}.
\end{align}}%
%To implement our method, similar to
where $\bc=[c_i]_{i\in\cN}\in\reals^{|\cN|}$ and $\bb=[b_i]_{i\in\cN}\in\reals^{|\cN|}$ are chosen uniformly at random between 0 and 1. We consider both static and dynamic networks; dynamic ones are generated %similar to our setting
as in Section~\ref{sec:generation} with $|\cN|=50$ nodes and $|\cE_u|=150$ edges, and the static one is set to $\cG=(\cN,\cE_u)$. %$n=1$
In the experiments we set $\delta=5$. %The graphs
 For benchmarking, we compared our algorithm against Consensus-based Saddle-Point Subgradient (CoBa-SPS) \footnote{The code is available online and it is used to implement on problem \eqref{shannon1}.} \cite{mateos2017distributed} and Mirror-prox \cite{he2015mirror} -- the former one is a decentralized algorithm while the latter one is a centralized algorithm. Mirror-prox algorithm requires the global Lipschitz constant of $\grad \cL$, where %the Lagrangian function
 $\cL(\bxi,\by)=\bc^\top\bxi+\fprod{\sum_{i\in\cN}b_i\log(1+\xi_i)-\delta,~\by}$, for $\bxi\in[0,1]^{|\cN|}$, which is $\sqrt{2}\norm{\bb}$. Mirror-prox and CoBa-SPS, also, require a bound on the dual solutions. Similar to \cite{mateos2017distributed}, for the Slater Point $\bar{\bxi}=\mathbf{1}_{|\cN|}$, we have that $\norm{y^*}\leq \frac{N\max_{i\in\cN}{c_i}}{\log(2)(\sum_{i\in\cN}{b_i}) -\delta}$. %\todo{is there parenthesis in the denominator?}.
 We compare DPDA-S against CoBa-SPS and Mirror-prox. Since CoBa-SPS can only handle static network, when the network topology is time-varying, we compare DPDA-D only against Mirror-prox, where we set $q_k=10\log(k+1)$ within DPDA-D.
We choose our step-sizes according to \eqref{eq:step-rule-dual} where $L_{f_i}=0$, $L_{g_i}=C_{g_i}=b_i$, and we set $\gamma=1/|\cN|$. Fig.~\ref{fig:dynamic_paper} shows the performance of DPDA-D in terms of suboptimality and infeasibility as well as consensus violation. The performance of our method is comparable with the centralized Mirror-prox and slightly slower rate of DPDA-D is the price we pay for the decentralized setting. Fig.~\ref{fig:static_paper} compares the performance of DPDA-S against CoBa-SPS and Mirro-prox. Although CoBa-SPS finds a feasible solution and remains feasible, the iterates are far from optimality; and both suboptimality and consensus violations decrease with a slow rate. DPDA-S has a superior performance compared to CoBa-SPS. %and it is the same as the centralized algorithm. Note that
The discontinuity within infeasibility plots (middle figures) is due to achieving occasional feasibility when both primal and dual iterates are approaching their optimal solutions.
}
\vspace*{-3.5mm}
\begin{figure}[htb]\label{fig:dynamic_paper}
\includegraphics[scale=0.15]{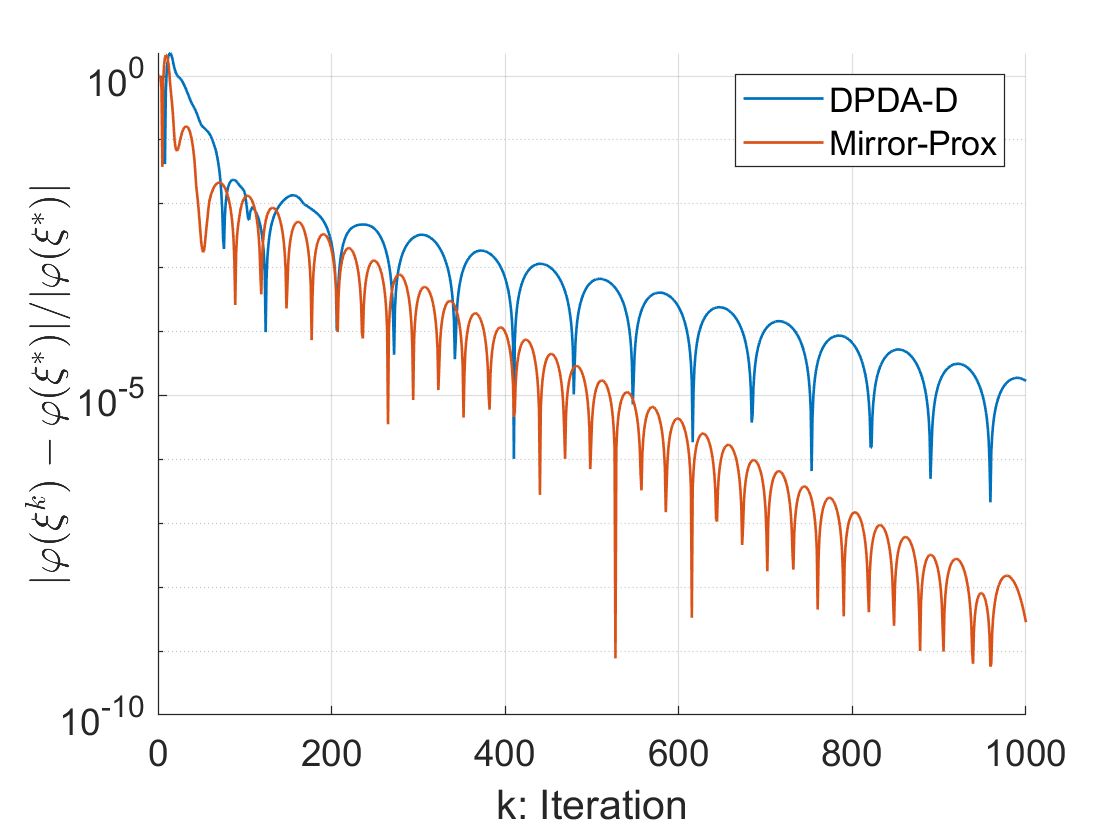}
\hspace*{-4mm}
\includegraphics[scale=0.15]{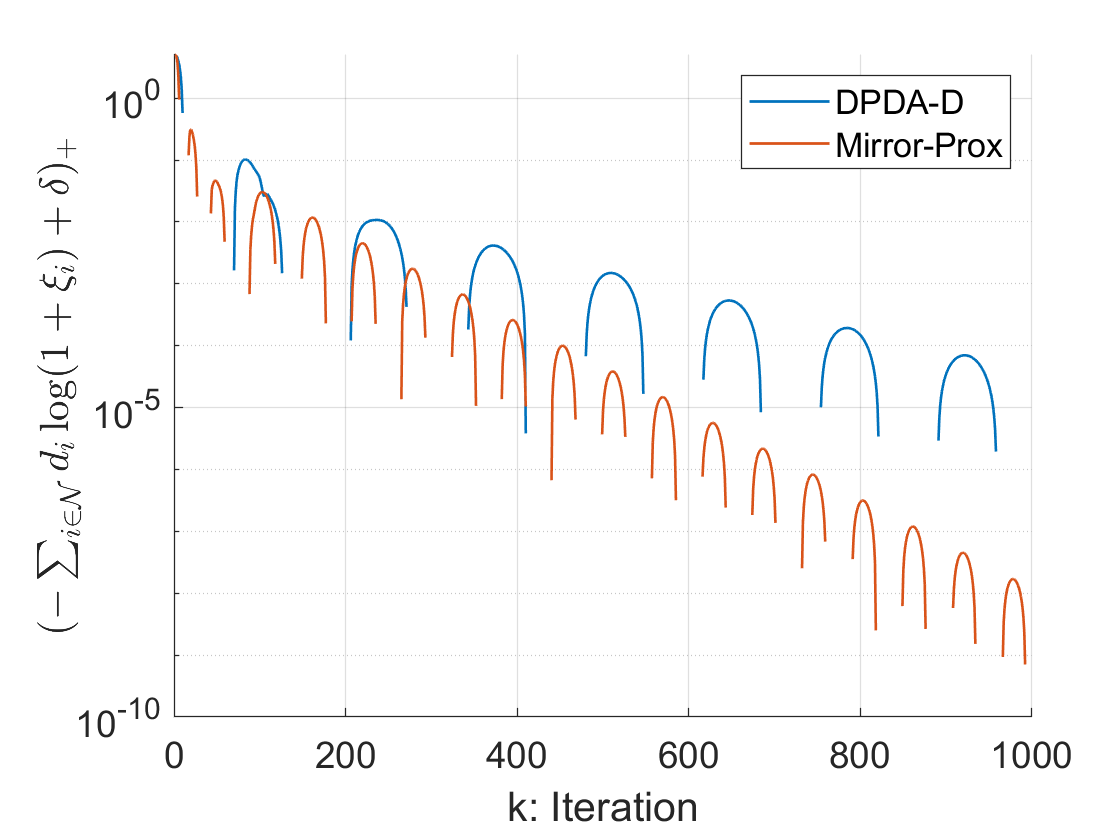}
\hspace*{-4mm}
\includegraphics[scale=0.15]{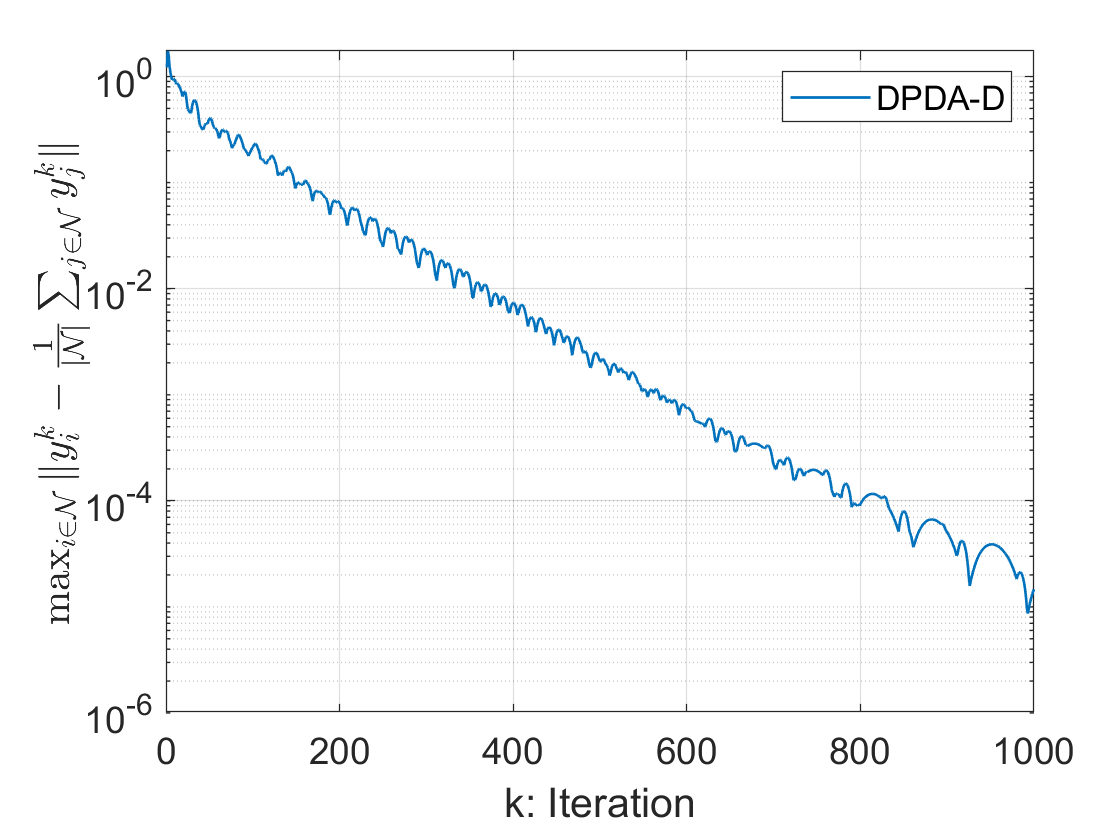}
\vspace*{-6mm}
\caption{Comparison between DPDA-D and Mirror-prox.}
\end{figure}
\vspace*{-7mm}
\begin{figure}[htb]\label{fig:static_paper}
\includegraphics[scale=0.15]{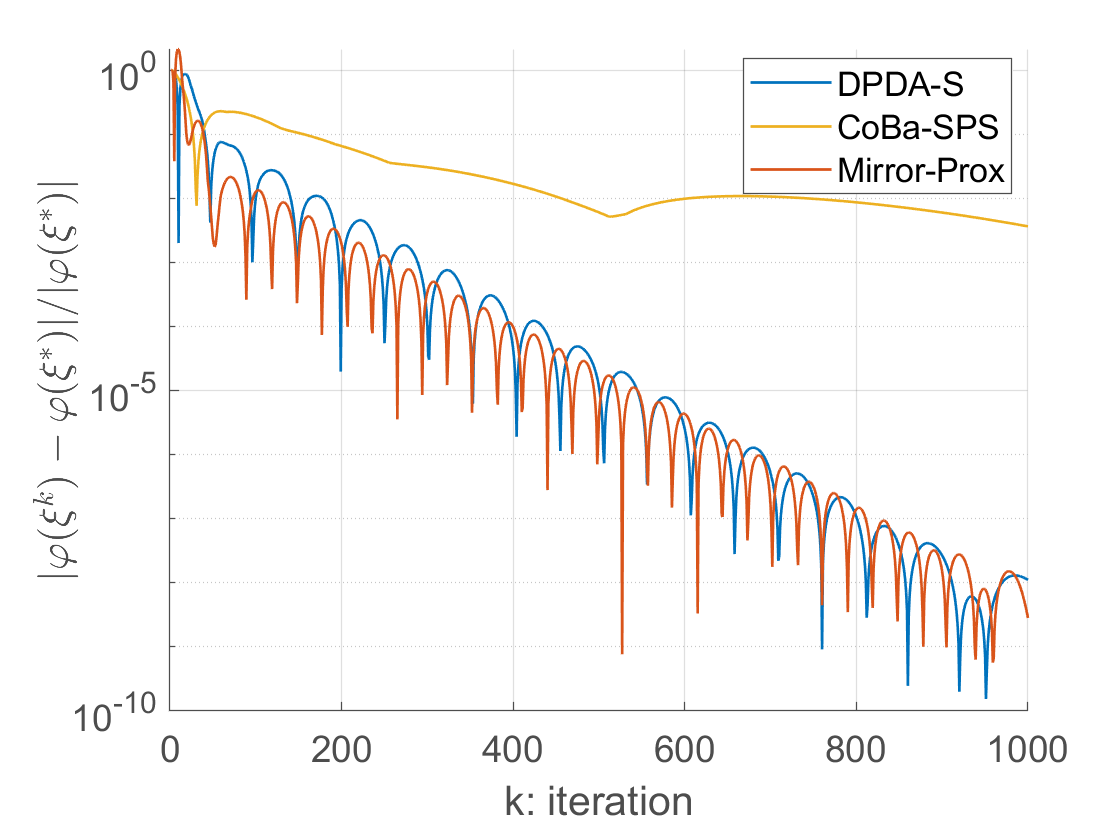}
\hspace*{-4mm}
\includegraphics[scale=0.15]{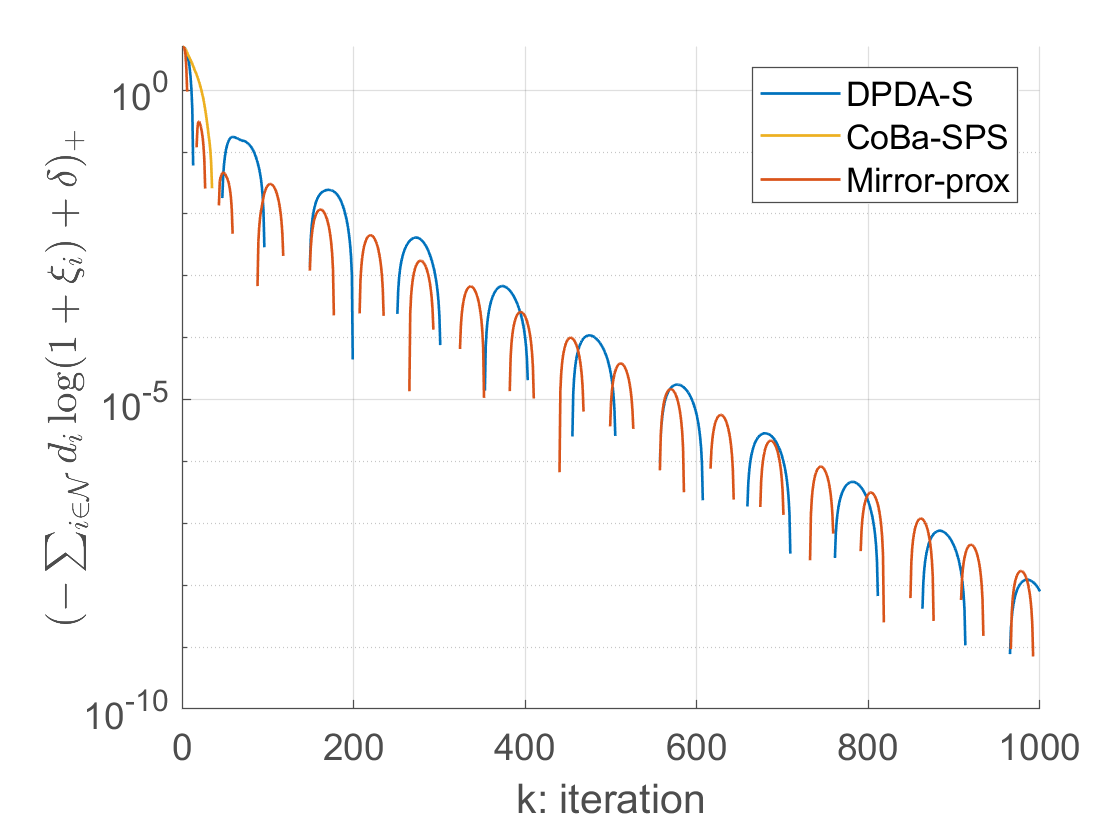}
\hspace*{-4mm}
\includegraphics[scale=0.15]{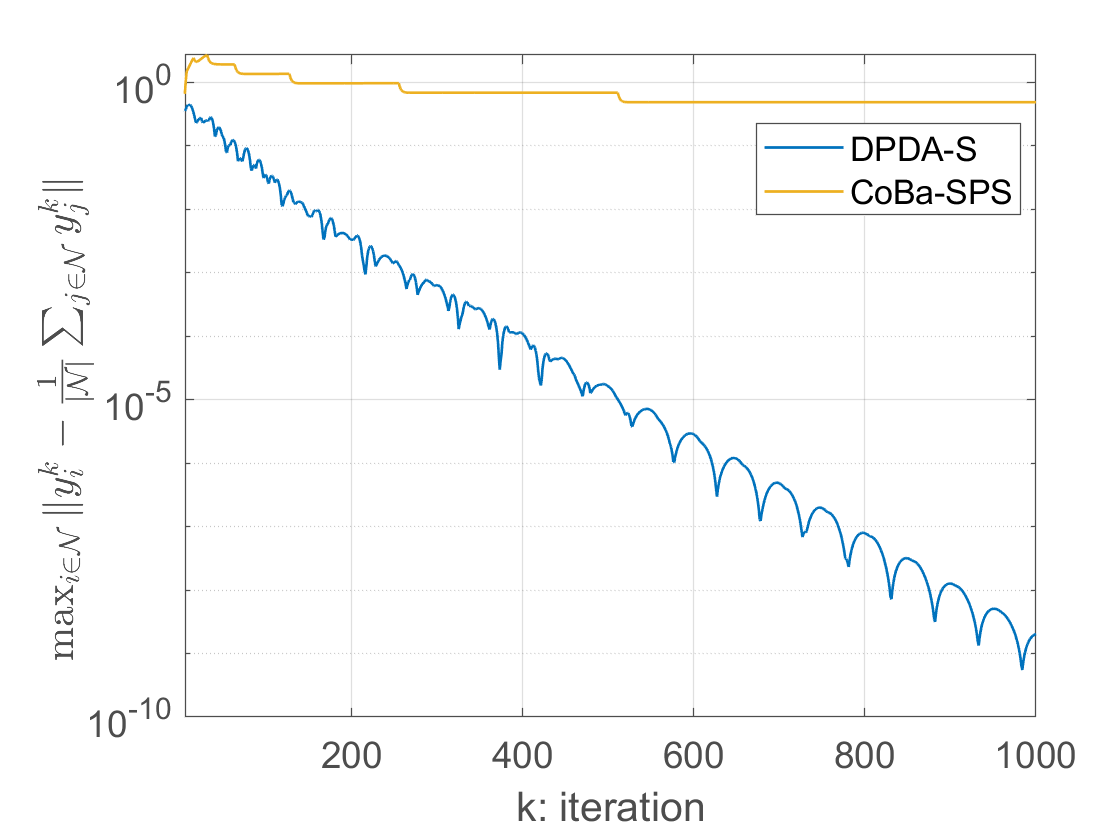}
\vspace*{-6mm}
\caption{Comparison among DPDA-S, Mirror-prox, and CoBa-SPS.}
\end{figure}
\vspace*{-4mm}
\section{Conclusions}%\todo{Erfan: I removed and changed some sentences.}
\label{sec:conclusions}
We propose a distributed primal-dual algorithm, DPDA-D, for solving cooperative multi-agent convex resource sharing problems over time-varying (un)directed communication networks, where only local communications are allowed. The objective is to minimize the sum of agent-specific composite convex functions subject to a conic constraint that couples agents' decisions. %We propose a distributed primal-dual algorithm
We show that the DPDA-D iterate sequence {converges to $\epsilon$-suboptimality/infeasibility within ${\cO}(1/\epsilon)$ number of iterations.} %rates where $k$ is the number of local communication rounds.
To the best of our knowledge, this is the best rate result for our setting. Moreover, DPDA-D employs agent-specific constant step-sizes {using local information.} %and each agent's constant step length can be computed with local information only, without requiring agents to know any parameter dependent on the global topology of the communication network.
As a potential future work, we plan to analyze convergence rates of similar primal-dual algorithms under certain strong convexity assumptions.

\bibliography{papers}{}
\bibliographystyle{siamplain}

\section{Appendix}\label{append}
\begin{lemma}\label{lem:supermartingale}
{\cite{Robbins71} Let $\{a^k\}$, $\{b^k\}$, $\{c^k\}$, and $\{d^k\}$ be non-negative real sequences such that $a^{k+1}\leq (1+d^k)a^k - b^k + c^k$ for all $k\geq 0$, $\sum_{k=0}^\infty c^k<\infty$, and $\sum_{k=0}^\infty d^k<\infty$. Then $a=\lim_{k\rightarrow \infty}a^k$ exists, and $\sum_{k=0}^\infty b^k<\infty$.}
\end{lemma}
\begin{lemma}\label{lem:technical-lemma}
\nv{Assume that $\{u_k\}_{k=0}^K\subset\reals_+$ satisfies $u_0^2\leq S_0$ and $u_k^2\leq S_k+\sum_{i=1}^k \lambda_iu_i$ for all $k\in\{1,\ldots,K\}$ for some $\{S_k\}_{k=0}^K$ non-decreasing in $k$ and $\{\lambda_k\}_{k=1}^K\subset\reals_+$. Then, the following inequality holds for all $k\in\{1,\ldots, K\}$:
\vspace*{-3mm}
\begin{eqnarray*}
u_k\leq \frac{1}{2}\sum_{i=1}^k \lambda_i+\left(S_k+\Big(\frac{1}{2}\sum_{i=1}^k\lambda_i \Big)^2 \right)^{1/2}.
\end{eqnarray*}}%
\end{lemma}
\end{document}